\documentclass[a4paper,10pt,reqno]{amsart}
\usepackage{amsmath}
\usepackage{cases}

\UseRawInputEncoding

\usepackage{amsfonts}
\usepackage[colorlinks,linkcolor=blue,citecolor=blue]{hyperref}
\usepackage{latexsym, amssymb, amsmath, amsthm, bbm}
\usepackage[all]{xy}
\usepackage{pgfplots}
\usepackage{enumerate}

\DeclareSymbolFont{EulerExtension}{U}{euex}{m}{n}
\DeclareMathSymbol{\euintop}{\mathop} {EulerExtension}{"52}
\DeclareMathSymbol{\euointop}{\mathop} {EulerExtension}{"48}

\allowdisplaybreaks[4]

\def \id{\operatorname{id}}
\def \Id{\operatorname{Id}}

\def \C{\mathcal{C}}

\def \e{\varepsilon}
\def \M{\mathrm{M}}

\def \Z{\mathbb{Z}}

\def \k{\Bbbk}

\def \dim{\operatorname{dim}}

\def \Hom{\operatorname{Hom}}

\def \Id{\operatorname{Id}}

\def \Rep{\operatorname{Rep}}

\def \Rex{\operatorname{Rex}}

\def \C{\mathcal{C}}
\def \D{\Delta}

\def \e{\varepsilon}
\def \M{\mathrm{M}}

\def \End{\operatorname{End}}

\def \C{\mathcal{C}}
\def \D{\mathcal{D}}

\def \M{\mathcal{M}}

\def \Z{\mathcal{Z}}

\def \dast{{\ast\ast}}
\def \ev{\mathsf{ev}}
\def \coev{\mathsf{coev}}
\def \tr{\mathrm{tr}}
\def \ptr{\mathrm{ptr}}
\def \Vec{\mathsf{Vec}}
\def \1{\mathbf{1}}
\def \Rep{\mathsf{Rep}}
\def \op{\mathrm{op}}
\def \cop{\mathrm{cop}}

\def \FPdim{\operatorname{FPdim}}

\def \intHom{\underline{\operatorname{Hom}}}
\def \ra{\mathrm{ra}}
\def \la{\mathrm{la}}
\def \FSexp{\operatorname{FSexp}}

\def \a{\mathfrak{a}}
\def \b{\mathfrak{b}}
\def \c{\mathfrak{c}}
\def \pams{partially admissible mapping system}
\def \biop{\mathrm{op\,cop}}
\usepackage{mathtools}
\usepackage{stmaryrd}

\def \pd{\boldsymbol}

\numberwithin{equation}{section}

\newtheorem{theorem}{Theorem}[section]
\newtheorem{lemma}[theorem]{Lemma}
\newtheorem{proposition}[theorem]{Proposition}
\newtheorem{corollary}[theorem]{Corollary}
\newtheorem{definition}[theorem]{Definition}

\newtheorem{remark}[theorem]{Remark}

\newtheorem{conjecture}[theorem]{Conjecture}
\newtheorem{notation}[theorem]{Notation}%[section]

\begin{document}
\title[\tiny{Indicators of dual fusion categories and partially dualized quasi-Hopf algebras}]{Frobenius-Schur indicators of dual fusion categories and semisimple partially dualized quasi-Hopf algebras}

\author[K. Li]{Kangqiao Li}
\address{School of Mathematics, Hangzhou Normal University, Hangzhou 311121, China}
\email{kqli@hznu.edu.cn}

\thanks{2020 \textit{Mathematics Subject Classification}.
18M15, 16T05, 18M20.}
\keywords{Frobenius-Schur indicator; Fusion category; Dual tensor category; Hopf algebra; Partial dualization; Exponent}

\thanks{$\dag$ This work was supported by National Natural Science Foundation of China [grant number 12301049].}

\date{}

\begin{abstract}
Frobenius-Schur indicators (or indicators for short) of objects in pivotal monoidal categories were defined and formulated by Ng and Schauenburg in 2007. In this paper, we introduce and study an analogous formula for indicators in the dual category $\C_\M^\ast$ to a spherical fusion category $\C$ (with respect to an indecomposable semisimple module category $\M$) over $\mathbb{C}$. Our main theorem is a relation between indicators of specific objects in $\C_\M^\ast$ and $\C$. As consequences:
1) We obtain equalities on the indicators between certain representations and the exponents of a semisimple complex Hopf algebra as well as its left partially dualized quasi-Hopf algebra;
2) We deduce formulas on indicators of certain modules over some particular semisimple Hopf algebras - bismash products and quantum doubles;
3) We show that for each semisimple left partially dualized quasi-Hopf algebra, its exponent and Frobenius-Schur exponent are identical.
\end{abstract}

\maketitle

\tableofcontents

\section{Introduction}

The notion of (higher) Frobenius-Schur indicators $\nu_n$ of modules over a finite-dimensional Hopf algebras or objects in pivotal monoidal categories is a classical invariant under equivalences.
It was defined firstly by Linchenko and Montgomery \cite{LM00} for a semisimple Hopf algebra $H$, and they provided a criteria for a simple $H$-module $V$ being self-dual with the value of $\nu_2(V)$. Afterwards in \cite{KSZ06}, Kashina, Sommerh\"{a}user and Zhu described various gauge invariants related to Frobenius-Schur indicators for semisimple Hopf algebras, where some essential formulas were established. One of the formulas was later expressed to define the regular Frobenius-Schur indicators of a finite-dimensional Hopf algebra $H$ (which is not necessarily semisimple) by Kashina, Montgomery and Ng in \cite{KMN12}, and it was generalized and studied further by Shimizu \cite{Shi15(b)}.

Moreover for a $\k$-linear pivotal monoidal category $\C$, Ng and Schauenburg defined in \cite{NS07(a)} the Frobenius-Schur indicators of objects in $\C$ with a categorical version of another formula in \cite{KSZ06}. The definition indeed generalizes the indicators for a semisimple Hopf algebra $H$, as they coincide when $\C=\Rep(H)$ is the category of finite-dimensional left $H$-modules. Furthermore, if $\C$ is a spherical fusion category over the complex field $\mathbb{C}$, they obtained in \cite{NS07(b)} a formula on the Frobenius-Schur indicators, which is described by the pivotal trace ``$\ptr(-)$'' of the power of the ribbon structure $\theta$ on the center $\Z(\C)$: For any object $X\in\C$, its $n$-th Frobenius-Schur indicator equals
\begin{equation}\label{eqn:indsformula(intro)}
\nu_n(X)=\frac{1}{\dim(\C)} \ptr\left(\theta_{I(X)}^n\right),
\end{equation}
where $\dim(\C)$ denotes the categorical dimension of $\C$, and $I:\C\rightarrow\Z(\C)$ is a two-sided adjoint to the forgetful functor. This formula was used to estimate the values of indicators and Frobenius-Schur exponents. As for non-pivotal categories, a theory of Frobenius-Schur indicators was developed by Shimizu \cite{Shi17(a)} with the construction of the pivotal cover.

Clearly, the results on Frobenius-Schur indicators for fusion categories would help to find properties and applications of indicators for semisimple (quasi-)Hopf algebras. One might see \cite{MN05,Sch04,NS08} and so on for the case of quasi-Hopf algebras.

To the knowledge of the author, Frobenius-Schur indicators for tensor categories were mostly studied under tensor equivalences or center constructions. Our goal of the paper to obtain some invariance properties under dualizations of fusion categories, or to find a categorical version of the results such as
\begin{equation}\label{eqn:invdualHopfalg}
\nu_n(H^\ast)=\nu_n(H)
\end{equation}
in the case of representations of a Hopf algebra $H$, which were shown in \cite{Shi12} and \cite{Shi15(b)}. This motivates us to consider indicators of objects in $\C_\M^\ast=\Rex_\C(\M)^\mathrm{rev}$, the dual category to $\C$ with respect to some $\C$-module category $\M$, since we know that $\Rep(H)_\Vec^\ast\approx\Rep(H^\ast)$ as tensor categories in the cases over Hopf algebras.

However, since we do not know when $\C_\M^\ast$ becomes pivotal, the Frobenius-Schur indicators are not well-defined in the original sense by Ng and Schauenburg. Thus we assume that $\C$ is a spherical fusion category, and attempt to redefine $\nu_n(\mathbf{F})$ for each object $\mathbf{F}\in\C_\M^\ast$ with a formula which is completely similar to Equation (\ref{eqn:indsformula(intro)}):
\begin{equation*}
\nu_n(\mathbf{F})=\frac{1}{\dim(\C)} \ptr\left(\theta_{K(\mathbf{F})}^n\right),
\end{equation*}
where $K$ is a two-sided adjoint to the composition functor
$$\Z(\C)\xrightarrow{\text{Schauenburg's equivalence}} \Z(\C_\M^\ast)
\xrightarrow{\text{forgetful functor}}\C_\M^\ast.$$
Fortunately, there is a description of the right adjoint functor $K$ by Shimizu \cite{Shi20} via the construction of ends, namely,
$$K(\mathbf{F})=\left(\int_{M\in\M} \underline{\Hom}(M,F(M)),\sigma^{\mathbf{F}}\right)\in\Z(\C).$$
With the help of this functor constructed, we are able to compute indicators $\nu_n(\mathbf{F})$ in some situations and obtain our main result(=Theorem \ref{thm:bimod-inds}) in Subsection \ref{subsection:mainthm}:

\begin{theorem}\label{thm:bimod-inds(intro)}
Let $\C$ be a spherical fusion category over $\mathbb{C}$. Suppose that $A$ is an algebra in $\C$ such that $\M=\C_A$ is an indecomposable $\C$-module category which is semisimple. Then for any $M\in \C_A$, the $n$-th indicator of
$-\otimes M\in\C_{\C_A}^\ast$ is
$$\nu_n(-\otimes M)=\nu_n(M)\;\;\;\;\;\;\;\;(\forall n\geq1),$$
where:
\begin{itemize}
\item
The $\C$-module functor $-\otimes M$ is canonically isomorphic to $(-)\otimes_A(A\otimes M)$ with structure as the associativity constraint of $\C$;

\item
The object $M$ in the right side is regarded in $\C$ (as the image under the forgetful functor from $\C_A$ to $\C$).
\end{itemize}
\end{theorem}

The key point of the proof is an observation(=Corollary \ref{cor:endsiso2}) that
$$K(-\otimes M)\cong(\mathsf{Z}(M),\sigma^{\mathsf{Z}(M)})$$
as objects in $\Z(\C)$, where $\mathsf{Z}$ denotes the central Hopf comonad on $\C$.
Besides,
we remark that any finite (left) $\C$-module category $\M$ must be equivalent to $\C_A$, the category of right $A$-modules in $\C$, for some algebra $A$ in $\C$. Therefore, Theorem \ref{thm:bimod-inds(intro)} could provide a relation between indicators in $\C$ and its arbitrary dual fusion category $\C_\M^\ast$.

Afterwards, this result is applied to the following algebraic case: $\C=\Rep(H)$ is the integral fusion category of representations of a semisimple complex Hopf algebra, and the indecomposable exact $\Rep(H)$-module category $\M=\Rep(B)$ for a left coideal subalgebra $B$ of $H$. As it is reconstructed in \cite{Li23} the so-called (left) partially dualized quasi-Hopf algebra $(H/B^+H)^\ast\#B$ from the dual tensor category $\Rep(H)_{\Rep(B)}^\ast$, we show by Theorem \ref{thm:bimod-inds(intro)} that (=Theorem \ref{thm:inds-leftpartialdual}):
\begin{theorem}\label{thm:inds-leftpartialdual(intro)}
Let $H$ be a semisimple Hopf algebra over $\mathbb{C}$. Suppose $B$ is a left coideal subalgebra, and $(H/B^+H)^\ast\#B$ is a left partially dualized quasi-Hopf algebra of $H$.
Then for any $L\in\Rep(B)$,
\begin{equation*}
\nu_n((H/B^+H)^\ast\# L)=\nu_n(L\square_{B^\ast}H^\ast)
\;\;\;\;\;\;\;\;(\forall n\geq1),
\end{equation*}
where $(H/B^+H)^\ast\#L\in\Rep((H/B^+H)^\ast\#B)$ and $L\square_{B^\ast}H^\ast\in\Rep(H)$ with structures defined in Subsection \ref{subsection:mainthmPD}.
\end{theorem}
In particular, the indicators of the regular representations of $(H/B^+H)^\ast\#B$ and $H$ coincide. Since the left partial dual $(H/B^+H)^\ast\#B$ is regarded as a generalization of the dual Hopf algebra $H^\ast$ in a sense, we could obtain Equation (\ref{eqn:invdualHopfalg}) directly from Theorem \ref{thm:inds-leftpartialdual(intro)}.

Furthermore, note in \cite{Li23,HKL25} that some classical structures are realized as special forms of left partially dualized Hopf algebras, including bismash products of matched pair of groups (\cite{Tak81}) and (generalized) quantum doubles (\cite{DT94}). Thus we might apply
Theorem \ref{thm:inds-leftpartialdual(intro)} to these examples, and deduce following formulas (=Propositions \ref{prop:indicatorsbismashprod} and \ref{prop:indicatorsquantumdouble}) on the indicators of certain modules:

\begin{proposition}
Let $(F,G,\triangleright,\triangleleft)$ be a matched pair of finite groups.
%For any $a\in F$, regard $\mathbb{C}p_a\in\Rep(\mathbb{C}^F)$ as the 1-dimensional representation.
Then
$$\nu_n(\mathbb{C}^G\#p_a)=\big|\{x\in G\mid(a,x)^n=(1,1)\text{ in }F\bowtie G\}\big|
\;\;\;\;\;\;\;\;(\forall n\geq1),$$
where $\{p_b\}_{b\in G}$ is the basis of $\mathbb{C}^G$ dual to $G$, and $\mathbb{C}^G\#p_a\in\Rep(\mathbb{C}^G\#\mathbb{C}F)$ with structures defined in Subsection \ref{subsection:5.4}.
\end{proposition}

\begin{proposition}\label{prop:1.4}
Let $H$ and $K$ be semisimple Hopf algebras over $\mathbb{C}$ with Hopf pairing $\sigma:K^\ast\otimes H\rightarrow\mathbb{C}$. Then for any $V\in\Rep(H)$,
$$\nu_n(K^{\ast\cop}\bowtie_\sigma V)=\sum_{W\in\mathsf{Irr}(K)}\nu_n^H(V\otimes W)\nu_n^{K^\cop}(W)
\;\;\;\;\;\;\;\;(\forall n\geq1),$$
where $\mathsf{Irr}(K)$ denotes the set of isoclasses of irreducible left $K$-modules, and:
\begin{itemize}
\item
$K^{\ast\cop}\bowtie_\sigma V$ is a left module over the (generalized) quantum double $K^{\ast\cop}\bowtie_\sigma H$ with structures defined in Subsection \ref{subsection:5.4};
\item
$V\otimes W\in\Rep(H)$ with diagonal $H$-action in a sense, and $W\in\Rep(K^\cop)$.
\end{itemize}
\end{proposition}
We remark as well that Proposition \ref{prop:1.4} implies an analogous formula (=Corollary \ref{cor:indicatorDrinfelddouble}) on indicators of the Drinfeld double $D(H)$ as a particular case.

Another gauge invariant for semisimple (quasi-)Hopf algebras is the \textit{exponent} (\cite{Kas99,EG99,Eti02} etc.), which is in fact the period of the sequence of indicators (of the regular representation). This observation was used to defined the \textit{Frobenius-Schur exponent} (\cite{NS07(b)}) of spherical fusion categories over $\mathbb{C}$, which induces frequently the Frobenius-Schur exponent of finite-dimensional semisimple quasi-Hopf algebras.
It is known that these two notions could differ with amounts up to a factor $2$, but we might show that
they coincide for semisimple left partial dualized quasi-Hopf algebras as an invariant under partial dualizations (Proposition \ref{cor:partialdualexp}).
Besides, some other properties on the exponent of semisimple Hopf algebras are also discussed.

The paper is organized as follows: In Section \ref{section2}, we collect some necessary notions of tensor categories and module categories, especially pivotal structures preserved by Schauenburg's equivalence $\Omega:\Z(\C)\approx\Z(\C_\M^\ast)$, as well as constructions of adjoint functors to the forgetful functors from centers. The definition of Frobenius-Schur indicators in $\C_\M^\ast$ is defined in Section \ref{section3}, and some basic invariance properties of indicators in various senses are discussed then. Section \ref{section4} is devoted to our main theorem.
Finally we introduce applications to indicators and exponent for semisimple Hopf algebras and its left partially dualized quasi-Hopf algebras in Section \ref{section5}, including formulas in the cases of some classical structures.

\section{Preliminaries on monoidal categories, module categories and centers}\label{section2}

We refer to \cite[Chapters 1 to 4]{EGNO15} for the elementary definitions and properties about monoidal categories and tensor categories, and tensor categories are assumed to be over an algebraically closed field.
In this paper, associativity and unit constraints in a monoidal category
are always abbreviated without the loss of generality, according to the coherence theorem \cite{MacL63}.

Let $\C$ be a monoidal category with tensor product functor $\otimes$ and unit object $\1$.
For any object $X\in\C$, we would denote its left and right dual objects in $\C$ by $X^\ast$ and ${}^\ast X$ respectively when exist. Also, the evaluation and coevaluation for the left dual $X^\ast$ of $X$ are denoted by
$$\ev_X:X^\ast\otimes X\rightarrow\1\;\;\;\;\text{and}\;\;\;\;
\coev_X:\1\rightarrow X\otimes X^\ast.$$
When $\C$ is rigid, the canonical isomorphisms
$$(X\otimes Y)^\ast\cong Y^\ast\otimes X^\ast\;\;\;\;\text{and}\;\;\;\;
{}^\ast(X\otimes Y)\cong {}^\ast Y\otimes {}^\ast X\;\;\;\;\;\;\;\;(\forall X,Y\in\C)$$
are always abbreviated in this paper,
and the \textit{(left) categorical trace} could be defined for a morphism $f:X\rightarrow X^\dast$ in $\C$ as
$$\tr(f):=\ev_{X^\ast}\circ(f\otimes\id_{X^\ast})\circ\coev_X\in\End_\C(\1).$$
If $\C$ has furthermore a pivotal structure $j$, denote the \textit{(right) pivotal trace} (\cite[Section 2]{NS07(b)}) of an endomorphism $g:X\rightarrow X$ in $\C$ by
$$\ptr(g):=\tr(j_X\circ g).$$

%However, when $\C$ is rigid,
%we both identify endofunctors ${}^\ast((-)^\ast)$ and $({}^\ast(-))^\ast$ on $\C$ with the identity functor $\Id_\C$.

For later use, we remark that pivotal traces of (the powers of) a natural endomorphism are independent on the choices of isomorphic objects in $\C$:

\begin{lemma}\label{lem:independadj}
Let $\C$ be a rigid monoidal category with pivotal structure $j$. Suppose $\theta$
is a natural endomorphism on the identity functor $\Id_\C$. If two objects $X,Y\in\C$ are isomorphic, then
$$\ptr\left(\theta_{X}^n\right)=\ptr\left(\theta_{Y}^n\right)$$
holds for each positive integer $n$.
\end{lemma}

\begin{proof}
Suppose $\phi:X\cong Y$ is an isomorphism in $\C$. Then
the naturalities of $\theta$ and the pivotal structure $j$ imply respectively that
$$\theta_{Y}\circ\phi=\phi\circ\theta_{X}
\;\;\;\;\text{and}\;\;\;\;
j_Y\circ \phi=\phi^\dast\circ j_X.$$
Thus for any positive integer $n$, we obtain that
\begin{eqnarray*}
\ptr\left(\theta_Y^n\right)
&=& \ptr\left(\phi\circ\theta_X^n\circ\phi^{-1}\right)
~=~ \tr\left(j_Y\circ\phi\circ\theta_X^n\circ\phi^{-1}\right)  \\
&=& \tr\left(\phi^\dast\circ j_X\circ\theta_X^n\circ\phi^{-1}\right)
~=~ \tr\left(j_X\circ\theta_X^n\circ\phi^{-1}\circ\phi\right)  \\
&=& \tr\left(j_X\circ\theta_X^n\right)
~=~ \ptr\left(\theta_X^n\right),
\end{eqnarray*}
where the forth equality is due to \cite[Proposition 4.7.3(4)]{EGNO15}.
\end{proof}

Moreover, recall in \cite[Definition 4.7.14]{EGNO15} that a pivotal structure $j$ of a rigid monoidal category $\C$ is said to be \textit{spherical}, if $\tr(j_X)=\tr(j_{X^\ast})$ holds for all $X\in\C$. In fact, it is equivalent to the condition that (right) pivotal traces of any endomorphism and its dual morphisms are equal:
\begin{lemma}
(\cite[Definition 2.5]{BW99} or \cite[Section 2.1]{Shi12})%\label{lem:independadj}
Let $\C$ be a rigid monoidal category with pivotal structure $j$. Then $j$ is spherical if and only if
\begin{equation}\label{eqn:sphericaltraces}
\ptr\left(f\right)=\ptr\left(f^\ast\right)
\end{equation}
holds for each endomorphism $f:X\rightarrow X$ in $\C$.
\end{lemma}

%\begin{proof}
%\textcolor{red}{
%This result seems known, but here we provide a proof for safety. WHERE TO REFER???}
%
%\begin{eqnarray*}
%\ptr(f^\ast)
%&=&
%\tr\left(j_{X^\ast}\circ f^\ast\right)
%~=~
%\ev_{X^\ast}\circ\left((j_{X^\ast}\circ f^\ast)\otimes\id\right)\circ\coev_{V^\ast}???
%\end{eqnarray*}
%\end{proof}

\subsection{Monoidal functors preserving pivotal structure}

Since pivotal structures are crucial in defining and studying Frobenius-Schur indicators for monoidal categories, we should pay special attention to monoidal functors said to preserve the pivotal structure, which are introduced in \cite{NS07(a)}.

Recall that if $\Phi:\C\rightarrow\D$ is a (strong) monoidal functor with structures
$$J:\Phi(-)\otimes \Phi(-)\cong \Phi(-\otimes-)
\;\;\;\;\text{and}\;\;\;\;
\varphi:\1\cong \Phi(\1)$$
between rigid monoidal categories, then there must be a canonical isomorphism
\begin{equation}\label{eqn:tau}
\chi_X:\Phi(X^\ast)\xrightarrow{\cong} \Phi(X)^\ast,
\end{equation}
which is called the \textit{duality transformation} of $\Phi$, such that the following diagrams both commute for each $X\in \C$:
\begin{equation}\label{eqn:dualitytrans}
\xymatrix{
\Phi(X^\ast)\otimes \Phi(X) \ar[r]^{\chi_X\otimes\id} \ar[d]_{J_{X^\ast,X}}
& \Phi(X)^\ast\otimes \Phi(X) \ar[dd]^{\ev_{\Phi(X)}}  \\
\Phi(X^\ast\otimes X) \ar[d]_{\Phi(\ev_X)}  \\
\Phi(\1) \ar[r]^{\varphi^{-1}} & \;\1\;,
}\;\;\;\;\;\;\;\;
\xymatrix{
\1 \ar[r]^{\varphi} \ar[dd]_{\coev_{\Phi(X)}}
& \Phi(\1) \ar[d]^{\Phi(\coev_X)}  \\
& \Phi(X\otimes X^\ast) \ar[d]_{J_{X,X^\ast}^{-1}}  \\
\Phi(X)\otimes \Phi(X)^\ast \ar[r]^{\id\otimes\chi_X^{-1}} & \;\Phi(X)\otimes \Phi(X^\ast)\;.
}
\end{equation}
By \cite[Lemma 1.1]{NS07(a)}, $\chi$ is in fact a monoidal transformation natural in $X\in\C$, and it is unique as
$$\chi_X=(\Phi(\ev_X)\otimes\id_{\Phi(X)^\ast})\circ(\id_{\Phi(X^\ast)}\otimes\coev_{\Phi(X)}),$$
where monoidal structures $J$ and $\varphi$ are abbreviated.

\begin{notation}
For convenience in this paper, we denote by $\xi:\Phi((-)^\dast)\cong \Phi(-)^\dast$ the monoidal natural isomorphism defined as follows:
\begin{equation}\label{eqn:xi}
\xi_X:\Phi(X^\dast)\xrightarrow{\chi_{X^\ast}} \Phi(X^\ast)^\ast \xrightarrow{({\chi_X}^\ast)^{-1}} \Phi(X)^\dast,
%\;\;\;\;\;\;\;\;(\forall X\in\C)
\end{equation}
where $\chi$ is the duality transformation (\ref{eqn:tau}) of $\Phi$.
\end{notation}

It is shown in \cite{NS07(a)} that monoidal functors $\Phi$ preserve categorical traces in the following sense, which is described with terms of $\chi$ (or $\xi$).
Here we recall this result without identifying $\Phi(\1)$ with the unit object in $\D$:

\begin{lemma}(\cite[Lemma 6.1]{NS07(a)})\label{lem:monoidalfunspresevetr}
Let $\Phi:\C\rightarrow\D$ be a monoidal functor between rigid monoidal categories, with structures $J$ and $\varphi:\1\cong \Phi(\1)$. Suppose $f:X\rightarrow X^\dast$ is a morphism in $\C$. Then:
\begin{itemize}
\item[(1)]
$\Phi(\tr(f))
=\varphi\circ\tr\left(\xi_X\circ \Phi(f)\right)\circ\varphi^{-1}\in\End_\D(\Phi(\1))$;
\item[(2)]
If $\C$ and $\D$ are tensor categories over a field $\k$ and $\Phi$ is a tensor functor, then
$$\tr(f)=\tr\left(\xi_X\circ \Phi(f)\right)\in\k$$
under the usual identifications $\End_\C(\1)=\End_\D(\1)=\k$.
\end{itemize}
\end{lemma}

The notion of monoidal functors preserving the pivotal structure is introduced as follows:

%\textcolor{red}{Replace notation $E$ by $\Phi$ or $\Psi$?}

\begin{definition}(\cite[Section 1]{NS07(a)})\label{def:pivotalfun}
Let $\C$ and $\D$ be rigid monoidal categories with pivotal structures $j$ and $j'$ respectively. A monoidal functor $\Phi:\C\rightarrow\D$ is said to \textit{preserve the pivotal structure}, if the diagram
\begin{equation}\label{eqn:pivotalfun}
\xymatrix{
\Phi(X) \ar[r]^{\Phi(j_X)} \ar[d]_{j'_{\Phi(X)}} & \Phi(X^\dast) \ar[dl]^{\xi_X}  \\
\Phi(X)^\dast
}
\end{equation}
commutes for each $X\in\C$.
% where $\xi$ is the monoidal natural isomorphism defined as
%\begin{equation}\label{eqn:chi}
%\xi_X:\Phi(X^\dast)\xrightarrow{\chi_{X^\ast}} \Phi(X^\ast)^\ast \xrightarrow{({\chi_X}^\ast)^{-1}} \Phi(X)^\dast,
%%\;\;\;\;\;\;\;\;(\forall X\in\C)
%\end{equation}
%and $\chi$ is the duality transformation (\ref{eqn:tau}).
\end{definition}

On the other hand, if the rigid monoidal category $\C$ is assumed to be pivotal, then a tensor equivalence $\Phi:\C\rightarrow\D$ would provide $\D$ a pivotal structure preserved by $\Phi$. Moreover, we have the following proposition involving spherical structures:

\begin{proposition}\label{prop:equivspherical}
Let $\C$ and $\D$ be rigid monoidal categories with pivotal structures $j$ and $j'$ respectively. Suppose a faithful monoidal functor $\Phi:\C\rightarrow\D$ preserves the pivotal structure. If $j'$ is spherical, then $j$ is also spherical.
\end{proposition}

\begin{proof}
Suppose $X\in\C$. In order to prove
\begin{equation}\label{eqn:Cspherical}
\tr(j_X)=\tr(j_{X^\ast})\in\End_\C(\1),%\;\;\;\;\;\;\;\;(\forall X\in\C),
\end{equation}
let us compare their images in $\End_\D(\Phi(\1))$ under the functor $\Phi$ which is faithful.
%Note that we might assume that $E(\1)=\1\in\D$ without the loss of generality (see \cite[Section 1]{NS07(a)} for details).

In fact according to Lemma \ref{lem:monoidalfunspresevetr}(1), one could find that
\begin{equation}\label{eqn:E(tr(jX))}
\Phi(\tr(j_X))=\varphi\circ\tr\left(\xi_X\circ \Phi(j_X)\right)\circ\varphi^{-1}
=\varphi\circ\tr\left(j'_{\Phi(X)}\right)\circ\varphi^{-1}\in\End_\D(\Phi(\1)),
\end{equation}
where $\xi_X$ is defined in (\ref{eqn:xi}) for $\Phi$, and the latter equality follows from
\begin{equation}\label{eqn:xiE(jX)}
\xi_X\circ \Phi(j_X)=j'_{\Phi(X)}\in\Hom_\D(\Phi(X),\Phi(X)^\dast)
\end{equation}
as $\Phi$ preserves the pivotal structure with the commutativity of (\ref{eqn:pivotalfun}).

Similarly, another equation is obtained as
\begin{equation}\label{eqn:E(tr(jX*))}
\Phi(\tr(j_{X^\ast}))
=\varphi\circ\tr\left(\xi_{X^\ast}\circ \Phi(j_{X^\ast})\right)\circ\varphi^{-1}
=\varphi\circ\tr\left(j'_{\Phi(X^\ast)}\right)\circ\varphi^{-1}\in\End_\D(\Phi(\1)).
\end{equation}
Furthermore, the fact that $j':\Id_\D\cong(-)^\dast$ is a natural isomorphism provides the following commutative diagram in $\D$:
$$\xymatrix{
\Phi(X^\ast) \ar[r]^{j'_{\Phi(X^\ast)}} \ar[d]_{\chi_X}
& \Phi(X^\ast)^\dast \ar[d]^{\chi_X^\dast}  \\
\Phi(X)^\ast \ar[r]_{j'_{\Phi(X)^\ast}}  &  \Phi(X)^{\ast\ast\ast},
}$$
where $\chi$ is the duality transformation of $\Phi$. This is equivalent to
$j'_{\Phi(X^\ast)}=(\chi_X^{-1})^\dast\circ j'_{\Phi(X)^\ast}\circ\chi_X$, and consequently
\begin{eqnarray}\label{eqn:E(tr(jX*))2}
\Phi(\tr(j_{X^\ast}))
&\overset{(\ref{eqn:E(tr(jX*))})}{=}&
\varphi\circ\tr\big(j'_{\Phi(X^\ast)}\big)\circ\varphi^{-1}
=\varphi\circ\tr\big((\chi_X^{-1})^\dast\circ j'_{\Phi(X)^\ast}\circ\chi_X\big)\circ\varphi^{-1}
\nonumber  \\
&=& \varphi\circ\tr\big(j'_{\Phi(X)^\ast}\circ\chi_X\circ\chi_X^{-1}\big)\circ\varphi^{-1}
=\varphi\circ\tr\big(j'_{\Phi(X)^\ast}\big)\circ\varphi^{-1}
\end{eqnarray}
with the usage of \cite[Proposition 4.7.3(4)]{EGNO15}.

However, it is clear that $\tr\big(j'_{\Phi(X)}\big)=\tr\big(j'_{\Phi(X)^\ast}\big)$
since $j'$ is a spherical structure of $\D$. Thus we conclude by Equations (\ref{eqn:E(tr(jX))}) and (\ref{eqn:E(tr(jX*))2}) that
$\Phi(\tr(j_X))=\Phi(\tr(j_{X^\ast}))$, which implies that $\tr(j_X)=\tr(j_{X^\ast})$ because $\Phi$ is faithful.
\end{proof}

We end this subsection by using Equation \ref{eqn:xiE(jX)} to note a direct consequence of Lemma \ref{lem:monoidalfunspresevetr} for monoidal functors preserving the pivotal structure:

\begin{corollary}\label{cor:monoidalfunspreseveptr}
Let $\C$ and $\D$ be pivotal rigid monoidal categories, and let $\Phi:\C\rightarrow\D$ be a monoidal functor with structures $J$ and $\varphi:\1\cong \Phi(\1)$, which preserves the pivotal structure. Suppose $f:X\rightarrow X^\dast$ is a morphism in $\C$. Then:
\begin{itemize}
\item[(1)]
$\Phi(\ptr(f))
=\varphi\circ\ptr\left(\Phi(f)\right)\circ\varphi^{-1}\in\End_\D(\Phi(\1))$;
\item[(2)]
If $\C$ and $\D$ are tensor categories over a field $\k$ and $\Phi$ is a tensor functor, then
\begin{equation}\label{eqn:monoidalfunspreseveptr}
\ptr(f)=\ptr\left(\Phi(f)\right)\in\k
\end{equation}
under the usual identifications $\End_\C(\1)=\End_\D(\1)=\k$.
\end{itemize}
\end{corollary}

\subsection{Braided monoidal categories with pivotal (or spherical) structures}

In this subsection, suppose $\C$ is a rigid monoidal category with braiding $\sigma$, which is a family of isomorphisms
$$\sigma^X_Y:X\otimes Y\cong Y\otimes X$$
natural in $X,Y\in\C$, satisfying the equations
$$(\id_Y\otimes\sigma^X_Z)\circ(\sigma^X_Y\otimes\id_Z)
=\sigma^X_{Y\otimes Z},\;\;\;\;
(\sigma^X_Z\otimes\id_Y)\circ(\id_X\otimes\sigma^Y_Z)
=\sigma^{X\otimes Y}_Z\;\;\;\;\;\;\;\;
(\forall X,Y,Z\in\C).$$
The detailed definition and basic properties of \textit{braided monoidal categories} is given in \cite[Section 2]{JS93}.

Moreover, when $\C$ is rigid, recall in \cite[Section 2.2]{BK01} that the \textit{Drinfeld morphism} $u$ of $\C$ and its inverse are defined as follows:
\begin{eqnarray}
u_X=(\ev_X\otimes\id_{X^\ast})\circ(\id_{X^\ast}\otimes(\sigma^X_{X^\dast})^{-1})
\circ(\coev_{X^\ast}\otimes\id_X) &:& X\rightarrow X^\dast,  \\
u_X^{-1}=(\id_X\otimes\ev_{X^\ast})
\circ(\sigma^{X^\dast}_X \otimes\id_{X^\ast})
\circ(\id_{X^\dast}\otimes\coev_X)\;\;\;\; &:& X^\dast\rightarrow X,  \label{eqn:Drinfeldelementinverse}
\end{eqnarray}
which are isomorphisms natural in $X\in\C$. Now let us study how a braided monoidal functor transfers the Drinfeld morphism:

\begin{lemma}\label{lem:preserveDrinfeldmor}
Let $\C$ and $\D$ be braided rigid monoidal categories. %with pivotal structures $j$ and $j'$ respectively.
Suppose $\Phi:\C\rightarrow\D$ is a braided monoidal functor.
%preserving the pivotal structure.
Denote by $u$ and $u'$ respectively the Drinfeld morphisms of $\C$ and $\D$. Then
$$u'_{\Phi(X)}=\xi_X\circ\Phi(u_X)$$
holds for each object $X\in\C$, where $\xi$ is the monoidal natural isomorphism defined in (\ref{eqn:xi}).
\end{lemma}

\begin{proof}
We denote the structures of $\Phi$ by $J:\Phi(-)\otimes\Phi(-)\cong\Phi(-\otimes-)$ and $\varphi:\1\cong \Phi(\1)$ as usual, and the definition of braided monoidal functor means that
\begin{equation}\label{eqn:braidedmonoidalfunctor}
J_{Y,X}\circ\sigma^{\Phi(X)}_{\Phi(Y)}=\Phi(\sigma^X_Y)\circ J_{X,Y}
\;\;\;\;(\forall X,Y\in\C),
\end{equation}
where the braiding on $\D$ is also denoted by $\sigma$ without confusions.

Suppose $X\in\C$. Let us explain at first the following diagram commutes in $\D$:
\begin{equation}\label{eqn:preserveDrinfeldmor1}\tiny
\xymatrix{
\Phi(X)^\dast \ar[rr]^{\id\otimes\coev_{\Phi(X)}} \ar[d]_{\xi_X^{-1}}
&& \Phi(X)^\dast\otimes\Phi(X)\otimes\Phi(X)^\ast
  \ar[d]^{\id\otimes\id\otimes\chi_X^{-1}}  \\
\Phi(X^\dast) \ar[r]^{\id\otimes\coev_{\Phi(X)}\;\;\;\;\;\;\;\;\;\;\;\;\;\;\;\;}
  \ar[ddd]_{\Phi(\id\otimes\coev_X)} \ar[dr]_{\id\otimes\varphi}
& \Phi(X^\dast)\otimes\Phi(X)\otimes\Phi(X)^\ast
  \ar[ur]^{\xi_X\otimes\id\otimes\id} \ar[dr]_{\id\otimes\id\otimes\chi_X^{-1}\;\;\;\;}
& \Phi(X)^\dast\otimes\Phi(X)\otimes\Phi(X^\ast)  \\
& \Phi(X^\dast)\otimes\Phi(\1) \ar[d]_{\id\otimes\Phi(\coev_X)}
& \Phi(X^\dast)\otimes\Phi(X)\otimes\Phi(X^\ast) \ar[dl]_{\id\otimes J_{X,X^\ast}}
  \ar[d]^{J_{X^\dast,X}\otimes\id} \ar[u]_{\xi_X\otimes\id\otimes\id}  \\
& \Phi(X^\dast)\otimes\Phi(X\otimes X^\ast) \ar[dl]_{J_{X^\dast,X\otimes X^\ast}}
& \Phi(X^\dast\otimes X)\otimes\Phi(X^\ast) \ar[dll]^{J_{X^\dast\otimes X,X^\ast}}\;.  \\
\Phi(X^\dast\otimes X\otimes X^\ast)
}
\end{equation}
The reasons for the commutativity of the cells in this diagram are stated as follows:
\begin{itemize}
\item The top and top-right quadrangle: $\otimes$ is a bifunctor on $\D$;
\item The middle pentagon: The definition (\ref{eqn:dualitytrans}) of the duality transformation $\chi$;
\item The bottom-left trapezoid: The equation $\id_{\Phi(X^\dast)}\otimes\varphi=J_{X^\dast,1}^{-1}$ due to \cite[Proposition 2.4.3]{EGNO15} for example, and the naturality of $J_{X^\dast,-}$;
\item The bottom-right quadrangle: The monoidal structure axiom for $J$.
\end{itemize}

On the other hand, we could obtain the commutativity of another diagram:
\begin{equation}\label{eqn:preserveDrinfeldmor2}\tiny
\xymatrix{
\Phi(X)\otimes\Phi(X)^\dast\otimes\Phi(X)^\ast
  \ar[d]_{\id\otimes\id\otimes\chi_X^{-1}}  \ar[drr]^{\id\otimes\ev_{\Phi(X)^\ast}} \\
\Phi(X)\otimes\Phi(X)^\dast\otimes\Phi(X^\ast)
  \ar[r]^{\id\otimes{(\chi_X)}^\ast\otimes\id}
& \Phi(X)\otimes\Phi(X^\ast)^\ast\otimes\Phi(X^\ast)
  \ar[r]_{\;\;\;\;\;\;\;\;\;\;\;\;\id\otimes\ev_{\Phi(X^\ast)}}
& \Phi(X)  \\
\Phi(X)\otimes\Phi(X^\dast)\otimes\Phi(X^\ast)
  \ar[ur]_{\;\;\id\otimes\chi_{X^\ast}\otimes\id} \ar[dr]^{J_{X^\dast,X^\ast}}
  \ar[d]_{J_{X,X^\dast}\otimes\id} \ar[u]^{\id\otimes\xi_X\otimes\id}
& \Phi(X)\otimes\Phi(\1) \ar[ur]_{\id\otimes\varphi^{-1}}  \\
\Phi(X\otimes X^\dast)\otimes\Phi(X^\ast) \ar[drr]_{J_{X\otimes X^\dast,X^\ast}}
& \Phi(X)\otimes\Phi(X^\dast\otimes\ X^\ast) \ar[u]_{\id\otimes\Phi(\ev_{X^\ast})}
  \ar[dr]^{\;\id\otimes J_{X,X^\dast\otimes X^\ast}}  \\
&& \Phi(X^\dast\otimes X\otimes X^\ast) \ar[uuu]_{\Phi(\id\otimes\ev_{X^\ast})}\;,
}
\end{equation}
where three cells commute due to similar reasons above, except:
\begin{itemize}
\item
The top quadrangle: The definition of the left dual morphism
$$({\chi_X})^\ast:=(\ev_{\Phi(X)^\ast}\otimes\id_{\Phi(X)^\ast})
\circ(\id_{\Phi(X)^\dast}\otimes\chi_X\otimes\id_{\Phi(X^\ast)^\ast})
\circ(\id_{\Phi(X)^\dast}\otimes\coev_{\Phi(X^\ast)})$$
and the equation
$$(\id_{\Phi(X^\ast)}\otimes\ev_{\Phi(X^\ast)})
\circ(\coev_{\Phi(X^\ast)}\otimes\id_{\Phi(X^\ast)})=\id_{\Phi(X^\ast)};$$
\item The left triangle: The definition (\ref{eqn:xi}) of $\xi$.
\end{itemize}

Moreover, since $\Phi$ is a braided monoidal functor with the property (\ref{eqn:braidedmonoidalfunctor}), the diagram
\begin{equation}\label{eqn:preserveDrinfeldmor3}\tiny
\xymatrix{
& \Phi(X)^\dast\otimes\Phi(X)\otimes\Phi(X)^\ast
  \ar[d]_{\id\otimes\id\otimes\chi_X^{-1}}
  \ar[r]^{\sigma^{\Phi(X)^\dast}_{\Phi(X)}\otimes\id}
& \Phi(X)\otimes\Phi(X)^\dast\otimes\Phi(X)^\ast
  \ar[d]^{\id\otimes\id\otimes\chi_X^{-1}}  \\
& \Phi(X)^\dast\otimes\Phi(X)\otimes\Phi(X^\ast)
  \ar[r]^{\sigma^{\Phi(X)^\dast}_{\Phi(X)}\otimes\id}
& \Phi(X)\otimes\Phi(X)^\dast\otimes\Phi(X^\ast)  \\
& \Phi(X^\dast)\otimes\Phi(X)\otimes\Phi(X^\ast)
  \ar[r]^{\sigma^{\Phi(X^\dast)}_{\Phi(X)}\otimes\id}
  \ar[d]_{J_{X^\dast,X}\otimes\id} \ar[u]^{\xi_X\otimes\id\otimes\id}
& \Phi(X)\otimes\Phi(X^\dast)\otimes\Phi(X^\ast)
  \ar[d]^{J_{X,X^\dast}\otimes\id} \ar[u]_{\id\otimes\xi_X\id\otimes\id}  \\
& \Phi(X^\dast\otimes X)\otimes\Phi(X^\ast) \ar[dl]_{J_{X^\dast\otimes X,X^\ast}}
  \ar[r]^{\Phi(\sigma^{X^\dast}_{X})\otimes\id}
& \Phi(X\otimes X^\dast)\otimes\Phi(X^\ast) \ar[dr]^{J_{X\otimes X^\dast,X^\ast}}  \\
\Phi(X^\dast\otimes X\otimes X^\ast) \ar[rrr]^{\Phi(\sigma^{X^\dast}_{X}\otimes\id)}
&&& \Phi(X^\dast\otimes X\otimes X^\ast)
}
\end{equation}
commutes as well.

As a conclusion of the commuting diagrams (\ref{eqn:preserveDrinfeldmor1}), (\ref{eqn:preserveDrinfeldmor2}) and (\ref{eqn:preserveDrinfeldmor3}), we find the equation
\begin{eqnarray*}
&& \big(\id_{\Phi(X)}\otimes\ev_{\Phi(X)^\ast}\big)
\circ\big(\sigma^{\Phi(X)^\dast}_{\Phi(X)}\otimes\id_{\Phi(X)^\ast}\big)
\circ\big(\id_{\Phi(X)^\dast}\otimes\coev_{\Phi(X)}\big)  \\
&=&
\Phi(\id_{X^\dast}\otimes\ev_{X^\ast})\circ\Phi(\sigma^{X^\dast}_{X}\otimes\id_{X^\ast})
  \circ\Phi(\id_{X^\dast}\otimes\coev_{X})\circ\xi_X^{-1}  \\
&=&
\Phi\Big((\id_{X^\dast}\otimes\ev_{X^\ast})\circ(\sigma^{X^\dast}_{X}\otimes\id_{X^\ast})
  \circ(\id_{X^\dast}\otimes\coev_{X})\Big)\circ\xi_X^{-1},
\end{eqnarray*}
which is equivalent to $(u'_{\Phi(X)})^{-1}=\Phi(u_X^{-1})\circ\xi_X^{-1}$ according to (\ref{eqn:Drinfeldelementinverse}) in the definition of the Drinfeld morphism. Thus our desired equation follows.
\end{proof}

In addition, suppose further that the braided rigid monoidal category $\C$ has a spherical structure $j$. Then the \textit{ribbon structure} $\theta$ on $\C$ is the natural automorphism defined as follows:
\begin{equation*}%\label{eqn:Z(C)ribbon}
\theta_X=u^{-1}_X\circ j_X:
X\cong X
\;\;\;\;(\forall X\in\C).
\end{equation*}
%and it satisfies
%\begin{equation}%\label{eqn:Z(C)ribbonproperty}
%(\theta_X)^\ast=\theta_{X^\ast}\;\;\;\;(\forall X\in\C)
%\end{equation}
%according to \cite[Proposition 8.10.12]{EGNO15}.

It is stated in \cite[Section 2]{NS07(b)} that a braided monoidal functor preserving the pivotal structure would also preserve the ribbon structure. We recall this fact as a consequence of Lemma \ref{lem:preserveDrinfeldmor}:

\begin{corollary}(cf. \cite[Section 2]{NS07(b)})\label{cor:preserveribbon}
Let $\C$ and $\D$ be braided monoidal categories, with spherical structures $j$ and $j'$ respectively. Suppose $\Phi:\C\rightarrow\D$ is a braided monoidal functor preserving the pivotal structure. Denote by $\theta$ and $\theta'$ respectively the ribbon structures of $\C$ and $\D$. Then
$$\theta'_{\Phi(X)}=\Phi(\theta_X)$$
holds for each object $X\in\C$.
\end{corollary}

\begin{proof}
One might find by the definition of the ribbon structures as well as Lemma \ref{lem:preserveDrinfeldmor} that
\begin{eqnarray*}
\theta'_{\Phi(X)} &=& {u'}^{-1}_{\Phi(X)}\circ j'_{\Phi(X)}
~=~ \Phi(u_X)^{-1}\circ\xi_X^{-1}\circ j'_{\Phi(X)}
\overset{(\ref{eqn:pivotalfun})}{=}
\Phi(u_X)^{-1}\circ\Phi(j_X)  \\
&=& \Phi(u_X^{-1}\circ j_X) ~=~ \Phi(\theta_X)
\end{eqnarray*}
holds for any $X\in\C$.
\end{proof}

\subsection{Module categories and dual tensor categories}

The notion of (left) module categories over monoidal categories could be found in \cite[Sections 7.1 to 7.10]{EGNO15},
and we always denote the module product bifunctor on a left $\C$-module category $\M$ by
$$\ogreaterthan:\C\times\M\rightarrow\M.$$
Also, associativity and unit constraints in $\M$ are abbreviated.

In this subsection, we recall some elementary properties of module categories over finite tensor categors, as well as the corresponding dual categories in the literature.

\begin{lemma}\label{lem:modcats}(\cite[Section 7.5]{EGNO15})
Let $\M$ be a module category over a finite tensor category $\C$.
%\begin{itemize}
%\item[(1)] (\cite[Lemma 3.4]{EO04})
%If $\M$ is %\textcolor{red}{}
%indecomposable and
%exact as a $\C$-module category, then $\M$ is finite.
%\item[(2)] (\cite[Section 7.5]{EGNO15})
If $\M$ is semisimple, then $\M$ is exact;
%\end{itemize}
\end{lemma}

\begin{remark}
Note in \cite[Section 3.1]{EO04} that an exact module category $\M$ over a finite tensor category $\C$ is defined to have only finitely many isoclasses of simple objects, and then $\M$ is finite if it is indecomposable according to \cite[Lemma 3.4]{EO04}.
\end{remark}

Suppose the left $\C$-module category $\M$ is finite and exact.
%and consequently all additive endofunctors on $\M$ are exact functors (\cite[Proposition 3.11]{EO04}).
The \textit{dual category} of $\C$ with respect to $\M$ is defined as
$$\C_\M^\ast:=\Rex_\C(\M)^\mathrm{rev},$$
the category of $\k$-linear right exact $\C$-module endofunctors on $\M$.
We should remark that the tensor product on $\C_\M^\ast$ in this paper is chosen to be opposite to the composition of $\C$-module endofunctors, which would differ with \cite{EO04} and \cite{EGNO15}, etc.
In this paper, each object in $\C_\M^\ast$ is denoted by a pair $\mathbf{F}=(F,s)$, where $s$ is a $\C$-module structure of $\k$-linear endofunctor $F:\M\rightarrow\M$ with notation
$$s_{X,M}:X\ogreaterthan F(M)\cong F(X\ogreaterthan M)
\;\;\;\;\;\;\;\;(X\in\C,\;M\in\M).$$
where $\ogreaterthan:\C\times\M\rightarrow\M$ denotes the module product functor.

\begin{lemma}(\cite[Section 3.3]{EO04})\label{lem:dualcattensor}
Suppose that $\M$ is a finite exact module category over a finite tensor category $\C$. Then $\C_\M^\ast$ is a finite multitensor category, which would be a tensor category if $\M$ is indecomposable as a $\C$-module category.
\end{lemma}

In particular, we would pay more attention to the case when $\C$ and $\M$ are both semisimple. When $\C$ is a fusion category, it is remarked in \cite[Example 3.3(iii)]{EO04} that a finite $\C$-module category $\M$ is exact if and only if it is semisimple. As for the dual category $\C_\M^\ast$:

\begin{lemma}(\cite[Theorem 2.15]{ENO05})\label{lem:dualcatfusion}
Let $\C$ be a fusion category over an algebraically closed field $\k$ of characteristic $0$. Suppose $\M$ is a finite indecomposable $\C$-module category which is semisimple (or exact). Then
%\begin{itemize}
%\item[(1)] (\cite[Section 3.3]{EO04})
%\;$\M$ is finite as a $\k$-linear abelian category;
%\item[(2)] (\cite[Theorem 2.15]{ENO05})\;
$\C_\M^\ast$ is a fusion category.
%\end{itemize}
\end{lemma}

\subsection{The centers as rigid monoidal categories and Schauenburg's equivalence}\label{subsectionOmega}

For any monoidal category $\C$, it is well-known that its (left) \textit{center} $\Z(\C)$ (e.g. \cite{JS91}) is the category consisting of objects $\mathbf{V}=(V,\sigma^V)$, where $V\in\C$ and the \textit{half-braiding} $\sigma^V$ is a family of isomorphisms
$$\sigma^V_X:V\otimes X\cong X\otimes V$$
natural in $X\in\C$, such that
$(\sigma^V_X\otimes\id_Y)\circ(\id_X\otimes\sigma^V_Y)
=\sigma^V_{X\otimes Y}$
holds for all $X,Y\in\C$.

Also, $\Z(\C)$ is a braided monoidal category with the braiding defined by
$$\sigma^V_W:V\otimes W\cong W\otimes V$$
for any $\mathbf{V}=(V,\sigma^V)$ and $\mathbf{W}=(W,\sigma^W)$ in $\Z(\C)$.
Besides, when $\C$ is rigid, its center $\Z(\C)$ would become rigid as well, and each object $\mathbf{V}=(V,\sigma^V)$ has left and right duals:
$$\mathbf{V}^\ast=\big(V^\ast,(\sigma^V_{{}^\ast(-)})^\ast\big)
\;\;\;\;\text{and}\;\;\;\;
{}^\ast\mathbf{V}=\big({}^\ast V,{}^\ast(\sigma^V_{(-)^\ast})\big).$$
Moreover, it is clear that the Drinfeld morphism $u$ of $\Z(\C)$ and its inverse are determined as follows:
\begin{eqnarray*}
u_\mathbf{V}=(\ev_V\otimes\id_{V^\ast})\circ(\id_{V^\ast}\otimes(\sigma^V_{V^\dast})^{-1})
\circ(\coev_{V^\ast}\otimes\id_V) &:& V\rightarrow V^\dast,  \\
(u_\mathbf{V})^{-1}=(\id_V\otimes\ev_{V^\ast})
\circ(\sigma^{V^\dast}_V \otimes\id_{V^\ast})
\circ(\id_{V^\dast}\otimes\coev_V) &:& V^\dast\rightarrow V,
\end{eqnarray*}
for any object $\mathbf{V}=(V,\sigma^V)\in\Z(\C)$.

Now suppose that $\C$ is a finite tensor category over a field $\k$, and $\M$ is a finite (left) $\C$-module category.
Then by \cite[Theorem 3.3]{Sch01}, there is an equivalence $\Omega:\Z(\C)\rightarrow\Z(\C_\M^\ast)$ of $\k$-linear braided monoidal categories, which is referred as \textit{Schauenburg's equivalence}.

Let us recall the details of the functor $\Omega$ according to \cite[Section 3]{Shi20}: For each object $\mathbf{V}=(V,\sigma^V)\in\Z(\C)$, we set $\Omega(\mathbf{V})=(V\ogreaterthan-,\varsigma^{V\ogreaterthan-})\in\Z(\C_\M^\ast)$, where the left $\C$-module structure of $V\ogreaterthan-$ is defined as the natural isomorphism
$$(\sigma^V_X)^{-1}\ogreaterthan\id_M: X\ogreaterthan V\ogreaterthan M
\cong V\ogreaterthan X\ogreaterthan M\;\;\;\;\;\;\;\;(X\in\C,\;M\in\M)$$
and the half-braiding $\varsigma^{V\ogreaterthan-}$ of $\Omega(\mathbf{V})$ is given as follows: For each $\mathbf{F}=(F,s)\in\C_\M^\ast$, define
\begin{equation}\label{eqn:Zequiv-halfbraid}
(\varsigma^{V\ogreaterthan-}_F)_M=
s_{V,M}^{-1}:(\mathbf{F}\circ\Omega(\mathbf{V}))(M)=F(V\ogreaterthan M)
  \xrightarrow{\cong} V\ogreaterthan F(M)=(\Omega(\mathbf{V})\circ\mathbf{F})(M)
\end{equation}
for $M\in\M$, and hence we might write $\varsigma^{V\ogreaterthan-}_F=s_{V,-}^{-1}$.
Also, the monoidal structure of $\Omega$ is given by:
\begin{equation*}
\sigma^W_V\ogreaterthan\id_{(-)}:\;\Omega(\mathbf{W})\circ\Omega(\mathbf{V})
=W\ogreaterthan V\ogreaterthan(-)
\;\xrightarrow{\cong}\; V\ogreaterthan W\ogreaterthan(-)
=\Omega(\mathbf{V}\otimes\mathbf{W})
\end{equation*}
for all objects $\mathbf{V}=(V,\sigma^V)$ and $\mathbf{W}=(W,\sigma^W)$ in $\Z(\C)$.

However in this paper, we would focus on the case when the $\C$-module category $\M$ is finite, indecomposable and exact, and consequently $\C_\M^\ast$ is a finite tensor category by Lemma \ref{lem:dualcattensor}. In this case, Schauenburg's equivalence $\Omega$ might be noted again as:

\begin{lemma}(cf. \cite[Theorem 3.13]{Shi20})
Let $\C$ be a finite tensor category, and let $\M$ be a finite indecomposable left $\C$-module category with is exact. Then $\Omega:\Z(\C)\rightarrow\Z(\C_\M^\ast)$ is an equivalence of braided finite tensor categories.
\end{lemma}

For the remaining of this subsection, we describe the spherical structure and ribbon structure on $\Z(\C)$ when $\C$ is spherical.

Denote the pivotal structure on a finite tensor category $\C$ by $j$, and then it follows that the center $\Z(\C)$ has a pivotal structure (e.g. \cite[Proposition 3.9]{Mug03}) defined as:
\begin{equation}\label{eqn:Z(C)pivotal}
j_\mathbf{V}:\mathbf{V}=(V,\sigma)\overset{j_V}{\cong}
  \left(V^\dast,(\sigma_{{}^\dast(-)})^\dast\right)=\mathbf{V}^\dast
\end{equation}
%where
%$$(\sigma_{{}^\dast X})^\dast=``graph''$$
for each $\mathbf{V}\in\Z(\C)$.
Moreover, $\Z(\C)$ is spherical if and only if $\C$ is spherical.

Besides,
%when $\C$ is a spherical \textit{fusion} category,
the ribbon structure $\theta$ of a braided spherical tensor category $\Z(\C)$ is the family of automorphism
$\theta_\mathbf{V}=u^{-1}_\mathbf{V}\circ j_\mathbf{V}$
natural in $\mathbf{V}\in\Z(\C)$,
where $u$ is the Drinfeld morphism of $\Z(\C)$.
It also satisfies
\begin{equation}\label{eqn:Z(C)ribbonproperty}
(\theta_{\mathbf{V}})^\ast=\theta_{\mathbf{V}^\ast}\;\;\;\;(\forall \mathbf{V}\in\Z(\C))
\end{equation}
according to \cite[Proposition 8.10.12]{EGNO15}.

We end this section by applying Corollary \ref{cor:preserveribbon} to Schauenburg's equivalence $\Omega:\Z(\C)\approx\Z(\C_\M^\ast)$:

\begin{corollary}\label{cor:dual-ribbon}
Let $\C$ be a finite tensor category with spherical structure $j$, and let $\M$ be a finite indecomposable $\C$-module category which is exact. Denote the ribbon structure of $\Z(\C)$ by $\theta=u^{-1}j$.
%Then the ribbon structure $\Theta$ of $\Z(\C_\M^\ast)$ with respect to the spherical structure determined by $\Omega(j)$ satisfies
If $\C_\M^\ast$ is pivotal and $\Omega$ preserves the pivotal structure, then $\C_\M^\ast$ is also spherical, and
\begin{equation}\label{eqn:Omegaribbon}
\Theta_{\Omega(\mathbf{V})} = \Omega(\theta_{\mathbf{V}})
= \theta_{\mathbf{V}}\ogreaterthan\id_{(-)}
\end{equation}
holds for each $\mathbf{V}\in\Z(\C)$, where $\Theta$ is the ribbon structure of $\Z(\C_\M^\ast)$.
\end{corollary}

\begin{proof}
This is a direct consequence of Proposition \ref{prop:equivspherical} and Corollary \ref{cor:preserveribbon}.
\end{proof}

\section{Frobenius-Schur indicators of the dual category to spherical fusion categories}\label{section3}

\subsection{Definition and compatibility}

Let $\C$ be a linear pivotal monoidal category with pivotal structure $j$ over a field. The the Frobenius-Schur indicators of objects in $\C$ is defined in \cite[Definition 3.1]{NS07(a)} as follows: Suppose $X\in\C$ and $n$ is a positive integer. Define the linear transformation
$$\begin{array}{ccrcl}
E_X^{(n)} &:&
\Hom_\C(\1,X^{\otimes n}) &\rightarrow &\Hom_\C(\1,X^{\otimes n}),  \\
&& f &\mapsto& (\ev_X\otimes\id_X^{\otimes(n-1)}\otimes j_X^{-1})
           \circ(\id_{X^\ast}\otimes f\otimes \id_{X^\dast})\circ\coev_{X^\ast},
\end{array}.$$
Then the $n$-th \textit{Frobenius-Schur indicator} (or \textit{indicator} for short) of $V$ is the scalar
\begin{equation}\label{eqn:indsdef}
\nu_n(X):=\mathrm{Tr}\left(E_X^{(n)}\right),
%\;\;\;\;(n\geq1),
\end{equation}
which denotes the trace of the linear transformation $E_X^{(n)}$ on the finite-dimensional vector space $\Hom_\C(\1,X^{\otimes n})$.

For the remaining of this subsection, we always assume that $\C$ is furthermore spherical and fusion over the complex field $\mathbb{C}$, and regard the (right) pivotal traces of endomorphisms in $\C$ as complex numbers via the identification $\End_\C(\1)=\mathbb{C}$.

Besides, denote by $\mathcal{O}(\C)$ the set of isoclasses of simple objects in $\C$.
The \textit{categorical dimension} of $\C$ is defined as
\begin{equation}\label{eqn:catdim}
\dim(\C)=\sum_{X\in\mathcal{O}(\C)} \ptr(\id_X) \ptr(\id_{X^\ast}),
\end{equation}
which is a real scalar not less than $1$ due to \cite[Theorem 2.3]{ENO05}.
It was proved in \cite[Theorem 4.1]{NS07(b)} that indicators could be computed through the ribbon structure $\theta$ of the center $\Z(\C)$:

\begin{lemma}(\cite[Theorem 4.1]{NS07(b)})\label{lem:indsorigin}
Let $\C$ be a fusion category over $\mathbb{C}$ with spherical structure $j$. Suppose $u$ is the Drinfeld morphism of $\Z(\C)$. Then for any $X\in\C$,
\begin{equation}\label{eqn:indsformula}
\nu_n(X)=\frac{1}{\dim(\C)} \ptr\left(\theta_{I(X)}^n\right)
\;\;\;\;\;\;\;\;(\forall n\geq1),
\end{equation}
where $\theta=u^{-1}j$ is the ribbon structure of $\Z(\C)$, and $I$ is a two-sided adjoint to the forgetful functor $\Z(\C)\rightarrow\C$.
\end{lemma}

%\begin{remark}
As noted in \cite[Section 4]{NS07(b)}, the adjoint $I:\C\rightarrow\Z(\C)$ to the forgetful functor could be chosen by a result of \cite[Proposition 8.1]{Mug03} such that
$$I(X)\cong\bigoplus_{(V,\sigma^V)\in\mathcal{O}(\Z(\C))} \dim(\Hom_\C(V,X))\cdot(V,\sigma^V),$$
where $\mathcal{O}(\Z(\C))$ denotes the set of isoclasses of simple objects in the fusion category $\Z(\C)$.
%We would explain in the next paragraph that $I$ could also be replaced by an arbitrary (left or right) adjoint to the forgetful functor.
%\end{remark}
However, it is known in \cite[Proposition 8.1]{Mug03} the forgetful functor $\mathbf{U}:\Z(\C)\rightarrow\C$ has isomorphic left and right adjoints.
Thus we conclude by Lemma \ref{lem:independadj} that the functor $I$ in (\ref{eqn:indsformula}) could be chosen as an arbitrary adjoint to $\mathbf{U}$.

\begin{remark}\label{rmk:adjsiso}
In fact as a fusion category, $\C$ must be unimodular due to \cite[Corollary 6.4]{ENO04}, which means that the distinguished invertible object is the unit object $\1$. Thus according to \cite[Theorem 4.10]{Shi17(a)}, the forgetful functor $\mathbf{U}:\Z(\C)\rightarrow\C$ is Frobenius since the base field $\mathbb{C}$ is perfect, and hence we could as well find that the left and right adjoints to $\mathbf{U}$ are naturally isomorphic.
\end{remark}

%\begin{lemma}\label{lem:independadj}
%Let $\C$ and $\theta$ be as in Lemma \ref{lem:indsorigin}. Suppose $I$ and $I'$ are left or right adjoints to the forgetful functor $\Z(\C)\rightarrow\C$. Then
%$$\ptr\left(\theta_{I(V)}^n\right)=\ptr\left(\theta_{I'(V)}^n\right)$$
%holds for each $V\in\C$.
%\end{lemma}
%
%\begin{proof}
%We have mentioned that there must be a natural isomorphism $\phi$ from $I$ to $I'$ in the cases. The naturalities of the ribbon structure $\theta$ and spherical structure $j$ imply respectively that
%$$\phi_V\circ\theta_{I(V)}=\theta_{I'(V)}\circ\phi_V
%\;\;\;\;\text{and}\;\;\;\;
%j_V\circ \phi_V=\phi_V^\dast\circ j_V$$
%for each $V\in\C$. Thus for any positive integer $n$,
%\begin{eqnarray*}
%\ptr\left(\theta_{I'(V)}^n\right)
%&=& \ptr\left(\phi_V\circ\theta_{I(V)}^n\circ\phi_V^{-1}\right)
%~=~ \tr\left(j_V\circ\phi_V\circ\theta_{I(V)}^n\circ\phi_V^{-1}\right)  \\
%&=& \tr\left(\phi_V^\dast\circ j_V\circ\theta_{I(V)}^n\circ\phi_V^{-1}\right)
%~=~ \tr\left(j_V\circ\theta_{I(V)}^n\circ\phi_V^{-1}\circ\phi_V\right)  \\
%&=& \tr\left(j_V\circ\theta_{I(V)}^n\right)
%~=~ \ptr\left(\theta_{I(V)}^n\right)
%\end{eqnarray*}
%holds, where the forth equality is due to \cite[Proposition 4.7.3(4)]{EGNO15}.
%\end{proof}

In order to study the indicators of objects in the dual category $\C_\M^\ast$, we attempt to introduce an analogous formula, which is supposed to be provided as a definition. This is because we do not know whether $\C_\M^\ast$ is canonically pivotal, even though $\C$ is a spherical fusion category.

%\begin{question}
%What kind of (indecomposable) module category $\M$ over a pivotal (or spherical) fusion category $\C$ would make the dual category $\C_\M^\ast$ pivotal again?
%\end{question}
%
%At least we are not supposed to say that the answer is always positive for every $\C$ and $\M$. Otherwise, we might use the equivalence $(\C_\M^\ast)_\M^\ast\cong\C$ (\cite[Theorem 3.27]{EO04}) to believe the well-known conjecture that all fusion categories are pivotal is true.
%
%The indicators in the dual category $\C_\M^\ast$ under certain assumptions are defined as follows:

\begin{definition}\label{def:dual-inds}
Let $\C$ be a fusion category over $\mathbb{C}$ with spherical structure $j$, and let
$\M$ be a finite indecomposable left $\C$-module category which is semisimple.
Denote by $\mathbf{U}_\M:\Z(\C_\M^\ast)\rightarrow\C_\M^\ast$ the forgetful functor,
and suppose $u$ is the Drinfeld morphism of $\Z(\C)$. For any $\mathbf{F}\in\C_\M^\ast$, %\textcolor{red}{(F)}
the $n$-th Frobenius-Schur indicator (or \textit{indicator} for short) of $\mathbf{F}$ is the defined as the scalar
\begin{equation}\label{eqn:dual-inds}
\nu_n(\mathbf{F})=\frac{1}{\dim(\C)} \ptr\left(\theta_{K(\mathbf{F})}^n\right)
\in\mathbb{C}\;\;\;\;\;\;\;\;(\forall n\geq1),
\end{equation}
where $\theta=u^{-1}j$ is the ribbon structure of $\Z(\C)$, and
$K$ is an arbitrary (left or right) adjoint to $\mathbf{U}_\M\circ\Omega:\Z(\C)\rightarrow\C_\M^\ast$.
\end{definition}

Similarly, we should explain the scalar $\ptr\left(\theta_{K(\mathbf{F})}^n\right)$ is independent on the choice of the functor $K:\C_\M^\ast\rightarrow\Z(\C)$ as a left or right adjoint to $\mathbf{U}_\M\circ\Omega$:

\begin{lemma}\label{lem:independadj2}
Let $\C$, $\M$ and $\theta$ be as in Definition \ref{def:dual-inds}. Suppose $K$ and $K'$ are left or right adjoints to the functor $\mathbf{U}_\M\circ\Omega:\Z(\C)\rightarrow\C_\M^\ast$. Then
$$\ptr\big(\theta_{K(\mathbf{F})}^n\big)=\ptr\big(\theta_{K'(\mathbf{F})}^n\big)
\;\;\;\;\;\;\;\;(\forall n\geq1)$$
holds for each $\mathbf{F}\in\C_\M^\ast$.
\end{lemma}

\begin{proof}
At first $\C_\M^\ast$ is a fusion category by \cite[Theorem 2.15]{ENO05}, as the finite $\C$-module category $\M$ is assumed to be indecomposable and semisimple.

Besides, since $\Omega:\Z(\C)\rightarrow\Z(\C_\M^\ast)$ is a tensor equivalence, $\Omega\circ K$ and $\Omega\circ K'$ are both (left or right) adjoint functors to the forgetful functor $\mathbf{U}_\M:\Z(\C_\M^\ast)\rightarrow\C_\M^\ast$, and hence they are naturally isomorphic according to \cite[Theorem 4.10]{Shi17(a)} as well. Consequently, we could know that $K(\mathbf{F})\cong K'(\mathbf{F})$ holds for each $\mathbf{F}\in\C_\M^\ast$. Our desired equality could be obtained with the help of Lemma \ref{lem:independadj} as well.
\end{proof}

For later use, we usually choose the adjoint functor $K$ in the view of Lemma \ref{lem:adjpair}(1) introduced by \cite{Shi20} as
\begin{equation}\label{eqn:radjKM}
\Z(\rho^\mathrm{ra}):\C_\M^\ast\rightarrow\Z(\C),\;\;
  \mathbf{F}=(F,s)\mapsto\left(\int_{M\in\M}\intHom(M,F(M)),\sigma^\mathbf{F}\right),
\end{equation}
where the detailed constructions would be recalled in Subsections \ref{subsection:4.1} and \ref{subsection:4.2}.
%half-braiding $\sigma^\mathbf{F}$ is defined through the commuting diagram (\ref{eqn:sigma^F}).
We also note the following fact, which is also a direct consequence of Lemma \ref{lem:independadj}:
\begin{remark}\label{rmk:isoobjshavesameinds}
For any $n\geq1$, isomorphic objects in $\C_\M^\ast$ always have the same $n$-th indicators in the sense of Definition \ref{def:dual-inds}.
\end{remark}

Now recall again from \cite[Theorem 2.15]{ENO05} that if $\M$ is a finite indecomposable and semisimple module category over a complex fusion category $\C$, then
\begin{equation}\label{eqn:dualcatdims}
\dim(\C_\M^\ast)=\dim(\C)
\end{equation}
holds.
Thus we might formulate the indicators of objects in $\C_\M^\ast$ with the ribbon structure $\Theta$ of its center $\Z(\C_\M^\ast)$, %similarly to (\ref{eqn:indsformula}),
as long as $\C_\M^\ast$ already has a pivotal %\textcolor{red}{(spherical?)}
structure such that Schauenburg's equivalence $\Omega$ preserves the pivotal structure:

%Now on the contrary, let us consider the case that $\C_\M^\ast$ already has a pivotal structure $j'$ such that $j'_\mathbf{F}=\mathbf{U}_\M(\mathfrak{j}_{(\mathbf{F},\varsigma)})$ holds for any $(\mathbf{F},\varsigma)\in\Z(\C_\M^\ast)$. Then by Lemma \ref{lem:indsorigin}, the Frobenius-Schur indicators of objects in the spherical fusion category $\C_\M^\ast$ should be computed according to (\ref{eqn:dual-inds2}).
%It is straightforward from Lemma \ref{lem:dual-inds} that

\begin{proposition}\label{prop:indscoincide}
Let $\C$ be a fusion category over $\mathbb{C}$ with spherical structure $j$, and let
$\M$ be a finite indecomposable left $\C$-module category which is semisimple.
%Suppose $u$ is the Drinfeld morphism of $\Z(\C)$.
Suppose $\C_\M^\ast$ is also pivotal %\textcolor{red}{(pivotal?)} %with structure $j'$
and $\Omega:\Z(\C)\approx\Z(\C_\M^\ast)$ preserves the pivotal structure. Then
for any $\mathbf{F}=(F,s)\in\C_\M^\ast$, its $n$-th indicator in the sense of Definition \ref{def:dual-inds} is equal to
\begin{equation}\label{eqn:indscoincide}
\nu_n(\mathbf{F})=\frac{1}{\dim(\C_\M^\ast)} \ptr\left(\Theta_{I_\M(\mathbf{F})}^n\right),
\end{equation}
where $\Theta$ is the ribbon structure of $\Z(\C_\M^\ast)$, and
$I_\M$ is an adjoint to the forgetful functor $\mathbf{U}_\M:\Z(\C_\M^\ast)\rightarrow\C_\M^\ast$.
In other words, the indicators of any $\mathbf{F}\in\C_\M^\ast$ in the senses of (\ref{eqn:indsformula}) and (\ref{eqn:dual-inds}) coincide.
\end{proposition}

\begin{proof}
%By Lemma \ref{lem:equivpreservespivot} and Corollary \ref{cor:equivspherical}, $\Z(\C_\M^\ast)$ is regarded to have the spherical structure $\mathfrak{j}$ determined by $\Omega(j)$, and $\Omega$ becomes a tensor equivalence that preserves the pivotal structure.

Since $\Omega:\Z(\C)\approx\Z(\C_\M^\ast)$ preserves the pivotal structure,
$\Z(\C_\M^\ast)$ (and hence $\C_\M^\ast$) is also spherical with structure denoted by $j'$ by Corollary \ref{cor:dual-ribbon}. Besides, we could obtain the equation
\begin{equation}\label{eqn:Omegapreservespivotal}
\xi_\mathbf{V}\circ\Omega(j_{\mathbf{V}})=j'_{\Omega(\mathbf{V})}
\;\;\;\;\;\;\;\;(\forall \mathbf{V}\in\Z(\C))
\end{equation}
according to Definition \ref{def:pivotalfun}, where
$\xi:\Omega((-)^\dast)\cong\Omega(-)^\dast$ is determined by the duality transformation of $\Omega$ as (\ref{eqn:xi}).
%Another formula we need is followed by Lemma \ref{lem:Drinfeldmor} that
%\begin{equation}\label{eqn:OmegapDrinfeldmor}
%\Omega(u_\mathbf{V})^{-1}=\mathfrak{u}_{\Omega(\mathbf{V})}^{-1}\circ\xi_\mathbf{V}
%\;\;\;\;\;\;\;\;(\forall \mathbf{V}\in\Z(\C)),
%\end{equation}
%where $u$ and $\mathfrak{u}$ are Drinfeld morphisms of $\Z(\C)$ and $\Z(\C_\M^\ast)$ respectively.

Now suppose $K:\C_\M^\ast\rightarrow\Z(\C)$ is a (two-sided) adjoint of $\mathbf{U}_\M\circ\Omega$.
With the usage of Corollary \ref{cor:monoidalfunspreseveptr}(2), we compute for any $n\ge1$ that
$$\ptr\big(\theta_{K(\mathbf{F})}^n\big)
\overset{(\ref{eqn:monoidalfunspreseveptr})}{=}
\ptr\left(\Omega(\theta_{K(\mathbf{F})}^n)\right)
=\ptr\left(\Omega(\theta_{K(\mathbf{F})})^n\right)
\overset{(\ref{eqn:Omegaribbon})}{=}
\ptr\left(\Theta_{\Omega(K(\mathbf{F}))}^n\right).$$
%\begin{eqnarray*}
%&& \Omega\left(\ptr\big(\theta_{K(\mathbf{F})}^n\big)\right)
%=
%\Omega\left(\tr\big(j_{K(\mathbf{F})}\circ\theta_{K(\mathbf{F})}^n\big)\right)  \\
%&=&
%\tr\left(\xi_{K(\mathbf{F})}\circ\Omega(j_{K(\mathbf{F})}\circ\theta_{ K(\mathbf{F})}^n)\right)
%=
%\tr\left(\xi_{K(\mathbf{F})}\circ\Omega(j_{K(\mathbf{F})})
%  \circ\Omega(\theta_{ K(\mathbf{F})})^n\right)  \\
%&\overset{(\ref{eqn:Omegapreservespivotal})}{=}&
%\tr\left(j'_{\Omega(K(\mathbf{F}))}
%  \circ\Omega(\theta_{ K(\mathbf{F})})^n\right)
%=
%\ptr\left(\Omega(\theta_{ K(\mathbf{F})})^n\right)  \\
%&\overset{(\ref{eqn:Omegaribbon})}{=}&
%\ptr\left(\Theta_{\Omega(K(\mathbf{F}))}^n\right).
%\end{eqnarray*}
However, it is clear by \cite[Theorem 1 in Section IV.8]{MacL98} that $\Omega\circ K$ is also an adjoint to the forgetful functor $\mathbf{U}_\M:\Z(\C_\M^\ast)\rightarrow\C_\M^\ast$, as $\Omega$ is a tensor equivalence. The same argument with the proof of Lemma \ref{lem:independadj2} would imply that $\Omega\circ K$ and $I_\M$ are naturally isomorphic, and hence
$$\ptr\left(\Theta_{\Omega(K(\mathbf{F}))}^n\right)
=\ptr\left(\Theta_{I_\M(\mathbf{F})}^n\right)$$
holds due to Lemma \ref{lem:independadj} as well.
As a conclusion, we find by Equation (\ref{eqn:dualcatdims}) that
$$\frac{1}{\dim(\C)}\ptr\big(\theta_{K(\mathbf{F})}^n\big)
=\frac{1}{\dim(\C_\M^\ast)}\ptr\left(\Theta_{I_\M(\mathbf{F})}^n\right)$$
as complex numbers for each $n\geq1$.
\end{proof}

\subsection{Invariance of adjoints in the dual category}

In this subsection, let us check that the indicators in the sense of Definition \ref{def:dual-inds} are invariant under the dualities for objects in $\C_\M^\ast$.

%we recall how the right adjoint to the forgetful functor $\mathbf{U}_\M:\Z(\C_\M^\ast)\rightarrow\C_\M^\ast$ is described with the notion of ends.
Before that, we recall some descriptions in \cite{Shi20} on adjoint pairs between dual categories and centers:
Assume that $\M$ is a finite left module category over a finite tensor category $\C$. Define the \textit{action functor}
$$\rho:\C\rightarrow\Rex(\M),\;\;X\mapsto X\otimes-$$
to the category $\Rex(\M)$ of linear right exact endofunctors on $\M$,
and note that $\rho$ is an exact monoidal functor. Thus it has both left adjoint $\rho^\la$ and right adjoint $\rho^\ra$. Moreover, the adjunction could be ``lifted'' to adjoint pairs between the centers:

\begin{lemma}(\cite[Section 3.6]{Shi20})\label{lem:adjpair0}
Suppose that $\M$ is exact as a $\C$-module category. Then there are are adjoint pairs
\begin{equation}\label{eqn:adjpair1}
\Big(\Z(\rho^\la):\Rex_\C(\M)\rightarrow\Z(\C),\;\;\;\;
   \Z(\rho):\Z(\C)\rightarrow\Rex_\C(\M)\Big)
\end{equation}
and
\begin{equation}\label{eqn:adjpair2}
\Big(\Z(\rho):\Z(\C)\rightarrow\Rex_\C(\M),\;\;\;\;
   \Z(\rho^\ra):\Rex_\C(\M)\rightarrow\Z(\C)\Big),
\end{equation}
where $\Z(\rho)$ is a (strong) monoidal functor, and $\Z(\rho^\ra)$ is a right adjoint to $\mathbf{U}_\M\circ\Omega$.
\end{lemma}

\begin{proof}
We omit here the detailed constructions of functors $\Z(\rho)$, $\Z(\rho^\la)$ and $\Z(\rho^\ra)$, but note that the latter adjunction desired is \cite[Theorems 3.11 and 3.14]{Shi20}, while the former one could also be obtained by applying the 2-functor $\Z$ to $\rho$ and its left adjoint $\rho^\la$. Details are founded in \cite[Sections 3.6 and 3.7]{Shi20}.
\end{proof}

Note again that in this paper, the tensor product bifunctor on the dual tensor category $\C_\M^\ast:=\Rex_\C(\M)^\mathrm{rev}$ is chosen as
$$(\mathbf{F},\mathbf{G})\mapsto\mathbf{G}\circ\mathbf{F},$$
which is opposite to the composition of $\C$-module endofunctors.
Thus
the left and right duals of every object $\mathbf{F}\in\C_\M^\ast$ would be its left and right adjoint functors, denoted by $\mathbf{F}^\la$ and $\mathbf{F}^\ra$, respectively.
Now we show that their indicators are equal as long as well-defined:

\begin{proposition}
Let $\C$ be a spherical fusion category over $\mathbb{C}$, and let
$\M$ be a finite indecomposable left $\C$-module category which is semisimple. Then for any $\mathbf{F}\in\C_\M^\ast$,
$$\nu_n(\mathbf{F}^\la)=\nu_n(\mathbf{F})=\nu_n(\mathbf{F}^\ra)
\;\;\;\;\;\;\;\;(\forall n\geq1).$$
\end{proposition}

\begin{proof}
At first by a fact stated in \cite[Section 2.2]{Ost03}, our assumptions imply that the left dual $\mathbf{F}^\la$ and right dual $\mathbf{F}^\ra$ are isomorphic as objects in the fusion category $\C_\M^\ast$. Thus $\nu_n(\mathbf{F}^\la)=\nu_n(\mathbf{F}^\ra)$ holds for all $n\geq1$ due to Remark \ref{rmk:isoobjshavesameinds}
and it suffices to show $\nu_n(\mathbf{F}^\la)=\nu_n(\mathbf{F})$.

%Now recall in Lemma \ref{lem:adjpair0} that there are adjoint pairs \begin{equation}\label{eqn:adjpair1}
%\Big(\Z(\rho^\la):\Rex_\C(\M)\rightarrow\Z(\C),\;\;\;\;
%   \Z(\rho):\Z(\C)\rightarrow\Rex_\C(\M)\Big)
%\end{equation}
%and
%\begin{equation}\label{eqn:adjpair2}
%\Big(\Z(\rho):\Z(\C)\rightarrow\Rex_\C(\M),\;\;\;\;
%   \Z(\rho^\ra):\Rex_\C(\M)\rightarrow\Z(\C)\Big).
%\end{equation}

%Now with the same argument as the proof of \cite[Equation (3.2)]{Shi20}, we try to show that
Now since $\Z(\rho)$ is a (strong) monoidal functor, we could apply \cite[Lemma 3.5]{BV12} to the adjoint pairs $(\Z(\rho^\la),\Z(\rho))$ and $(\Z(\rho),\Z(\rho^\ra))$, and obtain an isomorphism
\begin{equation}\label{eqn:adjinv}
{}^\ast\Z(\rho^\la)(\mathbf{F}^\la)\cong\Z(\rho^\ra)(\mathbf{F})
\end{equation}
in $\Z(\C)$, where $\Z(\rho)$, $\Z(\rho^\la)$ and $\Z(\rho^\ra)$ are denoted in Lemma \ref{lem:adjpair0}.
%Indeed, the duality adjunctions follow the isomorphism
%\begin{eqnarray*}
%\Hom_{\Z(\C)}\left(\mathbf{V},{}^\ast\Z(\rho^\la)(\mathbf{F}^\la)\right)
%&\cong& \Hom_{\Z(\C)}\left(\Z(\rho^\la)(\mathbf{F}^\la),\mathbf{V}^\ast\right)  \\
%&\overset{(\ref{eqn:adjpair1})}{\cong}& \Nat_{\Rex_\C(\M)}\left(\mathbf{F}^\la,\Z(\rho)(\mathbf{V}^\ast)\right),
%\end{eqnarray*}
%natural in $\mathbf{V}\in\Z(\C)$.
%Moreover, one could verify that the functor $\Z(\rho)$ is (strong) monoidal, since $\rho$ is so. Therefore we obtain isomorphisms $\Z(\rho)(\mathbf{V}^\ast)\cong\Z(\rho)(\mathbf{V})^\la$ and
%\begin{eqnarray*}
%&& \Nat_{\Rex_\C(\M)}(\mathbf{F}^\la,\Z(\rho)(\mathbf{V}^\ast))
%~\cong ~
%\Nat_{\Rex_\C(\M)}(\mathbf{F}^\la,\Z(\rho)(\mathbf{V})^\la)  \\
%&\cong&
%\Nat_{\Rex_\C(\M)}(\Z(\rho)(\mathbf{V}),\mathbf{F})
%~\overset{(\ref{eqn:adjpair2})}{\cong}~
%\Hom_{\Z(\C)}(\mathbf{V},\Z(\rho^\ra)(\mathbf{F})),
%\end{eqnarray*}
%which are natural in $\mathbf{V}\in\Z(\C)$.
%Thus according to Yoneda Lemma, the natural isomorphism
%$\Hom_{\Z(\C)}\left(\mathbf{V},{}^\ast\Z(\rho^\la)(\mathbf{F}^\la)\right)
%\cong\Hom_{\Z(\C)}(\mathbf{V},\Z(\rho^\ra)(\mathbf{F}))$
%concluded above implies (\ref{eqn:adjinv}), which implies also that
This implies also that
\begin{equation}\label{eqn:adjinv2}
\Z(\rho^\la)(\mathbf{F}^\la)\cong\Z(\rho^\ra)(\mathbf{F})^\ast
\end{equation}

On the other hand, more adjunctions could be found by the adjoint pairs
$$\left(\Z(\rho^\la),\Z(\rho)\right),\;\;\;\;
\left(\Z(\rho),\Z(\rho^\ra)\right)\;\;\;\;
\text{and}\;\;\;\;
\left(\mathbf{U}_\M\circ\Omega,\Z(\rho^\ra)\right),$$
according to \cite[Theorem 1 in Section IV.8]{MacL98}:
Both $\Z(\rho)\circ\Omega^{-1}$ and the forgetful functor $\mathbf{U}_\M$ are both left adjoints of $\Omega\circ\Z(\rho^\ra)$, and hence they would be naturally isomorphic.
Therefore, as the left and right adjoints of the forgetful functor $\mathbf{U}_\M$ are naturally isomorphic in this case (by the argument in Remark \ref{rmk:adjsiso} due to \cite[Theorem 4.10]{Shi17(a)}), we conclude that $\Omega\circ\Z(\rho^\la)$ and $\Omega\circ\Z(\rho^\ra)$ are both adjoints to $\mathbf{U}_\M$. It follows that
$$\Z(\rho^\la)\;\;\;\;\text{and}\;\;\;\;\Z(\rho^\ra)$$
are both (left or right) adjoints to
$\mathbf{U}_\M\circ\Omega:\Z(\C)\rightarrow\C_\M^\ast$.

Finally, since the definition of ribbon structure provides Equation (\ref{eqn:Z(C)ribbonproperty}) that
$\theta_{\mathbf{V}^\ast}=(\theta_\mathbf{V})^\ast$
for any $\mathbf{V}\in\Z(\C)$, we find that
\begin{equation}\label{eqn:ptr(thetaV*)=ptr(thetaV)}
\ptr\left(\theta_{\mathbf{V}^\ast}^n\right)
=\ptr\left((\theta_\mathbf{V}^n)^\ast\right)
\overset{(\ref{eqn:sphericaltraces})}{=}
\ptr\left(\theta_\mathbf{V}^n\right)
\;\;\;\;(\forall n\geq1)
\end{equation}
always holds in the spherical braided category $\Z(\C)$.
Thus for any $n\geq1$, we could compute by Lemma \ref{lem:independadj} that
$$
\ptr\left(\theta_{\Z(\rho^\la)(\mathbf{F}^\la)}^n\right)
\overset{(\ref{eqn:adjinv2})}{=}
\ptr\left(\theta_{\Z(\rho^\ra)(\mathbf{F})^\ast}^n\right)
\overset{(\ref{eqn:ptr(thetaV*)=ptr(thetaV)})}{=}
\ptr\left(\theta_{\Z(\rho^\ra)(\mathbf{F})}^n\right).
$$
Consequently, the desired invariance property $\nu_n(\mathbf{F}^\la)=\nu_n(\mathbf{F})$ follows due to Definition \ref{def:dual-inds}, which is independent of the choices of adjoint functors $\Z(\rho^\la)$ or $\Z(\rho^\ra)$.
\end{proof}

\subsection{Invariance under tensor equivalences}

In this subsection, we aim to show the invariance of indicators in $\C_\M^\ast$ under certain tensor equivalences between $\C$ and other spherical fusion categories. Suppose that $\Phi:\C\rightarrow\D$ is a tensor functor, and $\M$ is an exact $\D$-module category with module product bifunctor $\ogreaterthan$. Recall in \cite[Section 3.5]{EO04} that $\M$ becomes naturally a $\C$-module category via $\Phi$ as follows:
\begin{itemize}
\item
The module product bifunctor is
$\C\times\M\rightarrow \M,\;\;(X,M)\mapsto \Phi(X)\ogreaterthan M$;
\item
The module associativity constraint is $J^{-1}\ogreaterthan\id$, where $J$ is the monoidal structure of $\Phi$. Specifically,
\begin{equation}\label{eqn:Edeterminemodasso}
J_{X,Y}^{-1}\ogreaterthan\id_M:
\Phi(X\otimes Y)\ogreaterthan M\cong \Phi(X)\ogreaterthan \Phi(Y)\ogreaterthan M.
\end{equation}
Here we abbreviate the module associativity of the $\D$-module category $\M$.
\end{itemize}

%Note that the morphism
%$$f\otimes\id_M:X\otimes M\rightarrow Y\otimes M$$
%is defined as $E(f)\otimes\id_M$ for $f:X\rightarrow Y$ in $\C$.

If the $\C$-module category $\M$ determined as above is also exact (in this case $(\Phi,\M)$ is called an \textit{exact pair} by \cite[Definition 3.44]{EO04}), then we have the \textit{dual tensor functor}
\begin{equation}\label{eqn:E*}
\Phi^\ast:\D_\M^\ast\rightarrow\C_\M^\ast,
\end{equation}
which sends $\mathbf{F}=(F,s)\in\D_\M^\ast$ to the $\C$-module functor $(F,s_{\Phi(-),-})\in\C_\M^\ast$.

\begin{proposition}
Let $\C$ and $\D$ be spherical fusion categories over $\mathbb{C}$, and let $\M$ be a finite indecomposable left $\D$-module category which is semisimple. Suppose that $\Phi:\C\rightarrow\D$ is a tensor equivalence preserving the pivotal structure. Then for any $\mathbf{F}=(F,s)\in\D_\M^\ast$,
$$\nu_n(\Phi^\ast(\mathbf{F}))=\nu_n(\mathbf{F})
\;\;\;\;\;\;\;\;(\forall n\geq1).$$
\end{proposition}

\begin{proof}
Evidently, %since $E$ is an equivalence of fusion categories,
recall in \cite[Example 3.3(iii)]{EO04} that the semisimple category $\M$ is also exact as a $\C$-module category. Thus the indicators of objects $\Phi^\ast(\mathbf{F})\in\C_\M^\ast$ are well-defined for all $\mathbf{F}\in\D_\M^\ast$.

At first, it is mentioned in \cite[Section 2]{NS07(b)} that $\Phi$ induces a braided tensor equivalence between centers $\Z(\C)$ and $\Z(\D)$.
Specifically, suppose $\psi:\Id_\D\cong \Phi\Phi^{-1}$ is a monoidal natural isomorphism, where $\Phi^{-1}$ denotes a quasi-inverse of $\Phi$.
%A direct computation shows that
%$$E^\ast\circ \Omega'\circ E(j)=j\otimes\id=\Omega(j)$$
%holds in $\Z(\C_\M^\ast)$ for the pivotal structure $j$ of $\C$. It follows that the tensor equivalence $$E^\ast:\Z(\D_\M^\ast)\rightarrow\Z(\C_\M^\ast)$$ preserves pivotal structure, if their pivotal structures are chosen respectively as $\Omega'\circ E(j)$ and $\Omega(j)$. Consequently, we could find the ribbon structure
%$$\Theta=\Omega(\theta)???$$
Then for each $\mathbf{V}=(V,\sigma^V)\in\Z(\C)$, we define $\Phi(\mathbf{V})=(\Phi(V),\sigma^{\Phi(V)})\in\Z(\D)$, where the half-braiding $\sigma^{\Phi(V)}$ is defined through the commuting diagram for all $X'\in\D$:
\begin{equation}\label{eqn:Z(C)equivtoZ(D)}
\xymatrix{
\Phi(V)\otimes X' \ar[rr]^{\id_{\Phi(V)}\otimes\psi_{X'}} \ar@{-->}[d]_{\sigma^{\Phi(V)}_{X'}}
&& \Phi(V)\otimes \Phi\Phi^{-1}(X') \ar[rr]^{J_{V,\Phi^{-1}(X')}}
&& \Phi(V\otimes \Phi^{-1}(X')) \ar[d]^{\Phi(\sigma^V_{\Phi^{-1}(X')})}  \\
X'\otimes \Phi(V) \ar[rr]_{\psi_{X'}\otimes\id_{\Phi(V)}}
&& \Phi\Phi^{-1}(X')\otimes \Phi(V) \ar[rr]_{J_{\Phi^{-1}(X'),V}}
&& \Phi(\Phi^{-1}(X')\otimes V) \;.
}
\end{equation}
Morphisms in $\Z(\C)$ are mapped to $\Z(\D)$ in the same way as $\Phi:\C\rightarrow\D$.
Indeed, one could find that the diagram
\begin{equation}\label{eqn:Z(C)equivtoZ(D)2}
\xymatrix{
\Phi(V)\otimes \Phi(X)  \ar[r]^{J_{V,X}} \ar[d]_{\sigma^{\Phi(V)}_{\Phi(X)}}
& \Phi(V\otimes X) \ar[d]^{\Phi(\sigma^V_X)}  \\
\Phi(X)\otimes \Phi(V) \ar[r]_{J_{X,V}}
& \Phi(X\otimes V)
}
\end{equation}
commutes for all $X\in\C$ as a consequence of the commuting diagram (\ref{eqn:Z(C)equivtoZ(D)}), and hence $\Phi:\Z(\C)\rightarrow\Z(\D)$ is an equivalence between braided tensor categories.

Now we conclude relevant fusion categories and tensor functors with the following diagram of categories and functors:
\begin{equation}\label{eqn:equivinvdiagram}
\xymatrix{
\Z(\C) \ar[r]^{\Omega} \ar[d]_{\Phi}
& \Z(\C_\M^\ast) \ar[r]^{\mathbf{U}_\M} & \C_\M^\ast  \\
\Z(\D) \ar[r]_{\Omega'}
& \Z(\D_\M^\ast)  \ar[r]_{\mathbf{U}'_\M} %\ar@{-->}[u]^{E^\ast}
& \D_\M^\ast \ar[u]_{\Phi^\ast}
},
\end{equation}
where $\Omega'$ and $\mathbf{U}'_\M$ denote respectively Schauenburg's equivalence and the forgetful functor in the same ways.
Here we try to verify that the diagram (\ref{eqn:equivinvdiagram}) is commutative.

In fact according to Subsection \ref{subsectionOmega}, the composition $\mathbf{U}_\M\circ\Omega$ sends each object
$\mathbf{V}=(V,\sigma^V)\in\Z(\C)$ to the endofunctor $\Phi(V)\ogreaterthan -$ on $\M$ with $\C$-module structure determined as the left vertical arrow of the following commuting diagram:
\begin{equation}\label{eqn:Votimes-modstru}
\xymatrix{
\Phi(X)\ogreaterthan \Phi(V)\ogreaterthan M
  \ar[rr]^{J_{X,V}\ogreaterthan\id_M} \ar@{-->}[d]_{\cong}
&& \Phi(X\otimes V)\ogreaterthan M
  \ar[d]^{\Phi((\sigma^V_X)^{-1})\ogreaterthan\id_M}  \\
\Phi(V)\ogreaterthan \Phi(X)\ogreaterthan M  \ar[rr]_{J_{V,X}\ogreaterthan\id_M}
&& \Phi(V\otimes X)\ogreaterthan M
}
\end{equation}
for all $X\in\C$ and $\M\in M$. This is because horizontal arrows are module associativity constraints (\ref{eqn:Edeterminemodasso}) of the $\C$-module category $\M$. However, according to the commuting diagram (\ref{eqn:Z(C)equivtoZ(D)2}), the left vertical arrow of (\ref{eqn:Votimes-modstru}) is $(\sigma^{\Phi(V)}_{\Phi(X)})^{-1}\otimes\id_M$. As a conclusion:
\begin{equation}\label{eqn:UOgema(V)}
\mathbf{U}_\M\circ\Omega\left(V,\sigma^V\right)
=\left(\Phi(V)\ogreaterthan -,
    (\sigma^{\Phi(V)}_{\Phi(-)})^{-1}\ogreaterthan\id_{-}\right)
\in\C_\M^\ast.
\end{equation}

On the other hand, we obtain similarly that $\mathbf{U}'_\M\circ\Omega'\circ \Phi$ sends $\mathbf{V}=(V,\sigma^V)\in\Z(\C)$ to the endofunctor $\Phi(V)\otimes-$ on $\M$ with $\D$-module structure $(\sigma^{\Phi(V)}_{-})^{-1}\otimes\id_{-}$. Afterwards, it follows from the definition of the dual tensor equivalence $\Phi^\ast$ (\ref{eqn:E*}) that
$$\Phi^\ast\left(\Phi(V)\ogreaterthan-,
    (\sigma^{\Phi(V)}_{-})^{-1}\ogreaterthan\id_{-}\right)
=\left(\Phi(V)\ogreaterthan -,
    (\sigma^{\Phi(V)}_{\Phi(-)})^{-1}\ogreaterthan\id_{-}\right)
\in\C_\M^\ast,$$
which coincides with (\ref{eqn:UOgema(V)}). Therefore, we have shown that (\ref{eqn:equivinvdiagram}) is commutative.

%$$(J_{-,V}^{-1}\otimes\id_{-})\circ(E(\sigma^V_{-})\otimes\id_{-})
%\circ(J_{V,-}\otimes\id_{-}):
%E(V)\otimes E(-)\otimes-\cong E(-)\otimes E(V)\otimes-.$$

In order to compute the indicators desired,
suppose that $K'$ is a (right) adjoint to $\mathbf{U}'_\M\circ\Omega'$.
Then we apply \cite[Theorem 1 in Section IV.8]{MacL98} to find that
$$\left(\Phi^\ast\circ\mathbf{U}'_\M\circ\Omega'\circ \Phi,
\;\;\Phi^{-1}\circ K'\circ (\Phi^\ast)^{-1}\right),$$
or identically (by the commutativity of (\ref{eqn:equivinvdiagram})),
$$\left(\mathbf{U}_\M\circ\Omega,
\;\;\Phi^{-1}\circ K'\circ (\Phi^\ast)^{-1}\right)$$
would be an adjoint pair. Denote $K:=\Phi^{-1}\circ K'\circ (\Phi^\ast)^{-1}$ for simplicity, and hence there is a natural isomorphism
\begin{equation}\label{eqn:K'_Miso}
K'\cong \Phi\circ K\circ \Phi^\ast.
\end{equation}

Moreover, it is clear that the equivalence $\Phi:\Z(\C)\rightarrow\Z(\D)$ also preserves the pivotal structure, and hence we know by \cite[Section 2]{NS07(b)} or Corollary \ref{cor:preserveribbon} that
\begin{equation}\label{eqn:Epreserveribbon}
\theta'_{\Phi(\mathbf{V})}=\Phi(\theta_\mathbf{V})
\;\;\;\;\;\;\;\;(\forall \mathbf{V}\in\Z(\C)),
\end{equation}
where $\theta$ and $\theta'$ are the ribbon structures of $\Z(\C)$ and $\Z(\D)$ respectively. Therefore, we find by Lemma \ref{lem:independadj} that
\begin{eqnarray*}
\ptr\left({\theta'}_{K'(\mathbf{F})}^n\right)
&\overset{(\ref{eqn:K'_Miso})}{=}&
\ptr\left({\theta'}_{\Phi\circ K\circ \Phi^\ast(\mathbf{F})}^n\right)
\overset{(\ref{eqn:Epreserveribbon})}{=}
\ptr\left(\Phi\left(\theta_{K\circ \Phi^\ast(\mathbf{F})}\right)^n\right) \nonumber \\
&=&
\ptr\left(\Phi\left(\theta_{K\circ \Phi^\ast(\mathbf{F})}^n\right)\right)
\overset{(\ref{eqn:monoidalfunspreseveptr})}{=}
\ptr\left(\theta_{K\circ \Phi^\ast(\mathbf{F})}^n\right)
\end{eqnarray*}
holds for each $\mathbf{F}\in\D_\M^\ast$.
%, where the last equality is followed by Corollary \ref{cor:monoidalfunspreseveptr}(2).

Finally, since $\Phi:\C\rightarrow\D$ is a tensor equivalence between fusion categories preserving the pivotal structure,
\begin{eqnarray*}
\dim(\C)
&\overset{(\ref{eqn:catdim})}{=}&
\sum_{X\in\mathcal{O}(\C)}\ptr(\id_X)\ptr(\id_{X^\ast})
\overset{(\ref{eqn:monoidalfunspreseveptr})}{=}
\sum_{X\in\mathcal{O}(\C)}\ptr(\Phi(\id_X))\ptr(\Phi(\id_{X^\ast}))  \\
&=&
\sum_{X\in\mathcal{O}(\C)}\ptr(\id_{\Phi(X)})\ptr(\id_{\Phi(X^\ast)})
=
\sum_{X'\in\mathcal{O}(\D)}\ptr(\id_{X'})\ptr(\id_{{X'}^\ast})
\overset{(\ref{eqn:catdim})}{=}
\dim(\D).
\end{eqnarray*}
As a conclusion, for each $n\geq1$,
$$\nu_n(\Phi^\ast(\mathbf{F}))
=\frac{1}{\dim(\C)}\ptr\left({\theta}_{K'(\mathbf{F})}^n\right)
=\frac{1}{\dim(\D)}\ptr\left({\theta'}_{K\circ \Phi^\ast(\mathbf{F})}^n\right)
=\nu_n(\mathbf{F}).$$
\end{proof}

\section{Invariance for certain objects under the duality of categories}\label{section4}

\subsection{Ends and the central Hopf comonad}\label{subsection:4.1}

The notions of ends play an essential role in the construction of (right) adjoint to the forgetful functor from the center of a finite tensor category. We start by recalling relevant definitions, which could be found in \cite[Chapter IX]{MacL98}.

Suppose that $\C$ and $\D$ are categories, and $S,T:\C^\vee\times\C\rightarrow\D$ are functors, where $\C^\vee$ denotes the category having same objects but reversed directions of morphisms with $\C$.
A \textit{dinatural transformation} from $S$ to $T$ is a family
$$\varphi=\{\varphi_X:S(X,X)\rightarrow T(X,X)\}_{X\in\C}$$
of morphisms in $\D$ satisfying that
$$T(\id_X,f)\circ\varphi_X\circ S(f,\id_X)=T(f,\id_Y)\circ\varphi_Y\circ S(\id_Y,f)$$
for all morphisms $f:X\rightarrow Y$ in $\C$.

An \textit{end} of $S$ is an object $E\in\D$ equipped with a dinatural transformation $\pi$ from $E$ (as a constant functor) to $S$ such that:
For every dinatural transformation $\pi'$ from a constant functor $E'$ to $S$, there is a unique morphism $g:E'\rightarrow E$ satisfying $\pi'_X=\pi_X\circ g$ for all $X\in\C$. An end of $S$ is denoted by $\int_{X\in\C}S(X,X)$, which is unique up to a unique isomorphism,
and $\pi$ is called its \textit{universal dinatural transformation}.

%The definition of a \textit{coend} $\int^{X\in\C}S(X,X)$ of $S$ is completely dual to an end. We omit the details, since the coend construction seldom in this paper. See \cite[Chapter IX]{MacL98} for the details.

Now suppose $\C$ is a tensor category.
For any $V\in\C$, consider the functor
\begin{equation}\label{eqn:comonadfunctor}
\C^\vee\times\C\rightarrow \C,\;\;(X,Y)\mapsto Y\otimes V\otimes X^\ast.
\end{equation}
We would denote the end of (\ref{eqn:comonadfunctor}) by $\mathsf{Z}(V)=\int_{X\in\C} X\otimes V\otimes X^\ast$,
which makes $\mathsf{Z}$ an endofunctor called the \textit{central Hopf comonad} on $\C$ (see \cite[Definition 3.2]{Shi17(b)}) if such ends exist for all $V\in\C$. In this case, the universal dinatural transformations would be denoted by
$$\pi^{\mathsf{Z}(V)}_X:\mathsf{Z}(V)\rightarrow X\otimes V\otimes X^\ast\;\;\;\;(X\in\C).$$
Furthermore, the comonad structure of $\mathsf{Z}$ means that there are morphisms
$$\delta_V:\mathsf{Z}(V)\rightarrow\mathsf{Z}^2(V)\;\;\;\;\text{and}\;\;\;\;
  \epsilon_V:\mathsf{Z}(V)\rightarrow V\;\;\;\;\;\;\;\;(V\in\C)$$
uniquely determined by the properties that $\epsilon_V=\pi^{\mathsf{Z}(\1)}_V$ and
$$\left(\id_X\otimes\pi^{\mathsf{Z}(V)}_Y\otimes\id_{X^\ast}\right)
\circ \pi^{\mathsf{Z}^2(V)}_X\circ \delta_V=\pi^{\mathsf{Z}(V)}_{X\otimes Y}$$
hold for all objects $V,X,Y\in\C$ (see also \cite[Section 3.2]{Shi17(b)} for details).

\begin{lemma}\label{lem:centralHopfcomonadZ(C)}
(\cite[Section 3]{Shi17(b)} or \cite[Section 3.3]{Shi19})\label{lem:Z(C)radj}
Let $\C$ be a finite tensor category. Then the forgetful functor
$\mathbf{U}:\Z(\C)\rightarrow\C$
has a right adjoint:
$$I:\C\rightarrow\Z(\C),
\;\;V\mapsto \left(\mathsf{Z}(V),\sigma^{\mathsf{Z}(V)}\right),$$
where the half-braiding $\sigma^{\mathsf{Z}(V)}$ is defined as the composition
$$\sigma^{\mathsf{Z}(V)}_X:\mathsf{Z}(V)\otimes X
  \xrightarrow{\delta_V\otimes\id_X}
\mathsf{Z}^2(V)\otimes X
  \xrightarrow{\pi^{\mathsf{Z}^2(V)}_X\otimes\id_X}
X\otimes \mathsf{Z}(V)\otimes X^\ast\otimes X
  \xrightarrow{\id_X\otimes\id_{\mathsf{Z}(V)}\otimes\ev_X} X\otimes\mathsf{Z}(V)$$
natural in $X\in\C$.
\end{lemma}

The half-braiding $\sigma^{\mathsf{Z}(V)}$ determined in Lemma \ref{lem:Z(C)radj} is referred to as the \textit{canonical half-braiding} for $\mathsf{Z}(V)$. Here we provide another description for later use:

\begin{lemma}
Let $\C$ be a finite tensor category. Suppose $V\in\C$, and $\sigma^{\mathsf{Z}(V)}$ is the canonical half-braiding for $\mathsf{Z}(V)$. Then
\begin{equation}\label{eqn:sigma^Z(V)2}
\left(\id_X\otimes\pi^{\mathsf{Z}(V)}_Y\right)\circ\sigma^{\mathsf{Z}(V)}_X
=\Big(\id_X\otimes\id_Y\otimes\id_V\otimes\id_{Y^\ast}\otimes\ev_X\Big)
\circ\left(\pi^{\mathsf{Z}(V)}_{X\otimes Y}\otimes\id_X\right)
\end{equation}
holds for all $X,Y\in\C$.
\end{lemma}

\begin{proof}
Note that the canonical half-braiding for $\mathsf{Z}(V)$ is defined in Lemma \ref{lem:Z(C)radj} as
$$\sigma^{\mathsf{Z}(V)}_X
=(\id_X\otimes\id_{\mathsf{Z}(V)}\otimes\ev_X)
\circ(\pi^{\mathsf{Z}^2(V)}_X\otimes\id_X)
\circ(\delta_V\otimes\id_X).$$
Thus our goal is to explain the commutativity of the diagram
%\begin{equation*}
%\xymatrix{
%X\otimes Y\otimes V\otimes Y^\ast\otimes X^\ast\otimes X
%\ar[ddd]_{\id_X\otimes\id_Y\otimes\id_V\otimes\id_{Y^\ast}\otimes\ev_X}
%&&& \mathsf{Z}(V)\otimes X  \ar[d]^{\delta_V\otimes\id_X}
%  \ar[lll]_{\pi^{\mathsf{Z}(V)}_{X\otimes Y}\otimes\id_X}  \\
%&&& \mathsf{Z}^2(V)\otimes X  \ar[d]^{\pi^{\mathsf{Z}^2(V)}_X\otimes\id_X}  \\
%&&& X\otimes\mathsf{Z}(V)\otimes X^\ast\otimes X
%  \ar[d]^{\id_X\otimes\id_{\mathsf{Z}(V)}\otimes\ev_X}
%  \ar[uulll]^{\id_X\otimes\pi^{\mathsf{Z}(V)}_Y\otimes\id_{X^\ast}\otimes\id_X} \\
%X\otimes Y\otimes V\otimes Y^\ast
%&&& X\otimes\mathsf{Z}(V)  \ar[lll]^{\id_X\otimes\pi^{\mathsf{Z}(V)}_Y}
%}
%\end{equation*}
\begin{equation*}
\xymatrix{
\mathsf{Z}(V)\otimes X  \ar[d]_{\delta_V\otimes\id_X}
  \ar[rrr]^{\pi^{\mathsf{Z}(V)}_{X\otimes Y}\otimes\id_X}
&&& X\otimes Y\otimes V\otimes Y^\ast\otimes X^\ast\otimes X
  \ar[ddd]^{\id_X\otimes\id_Y\otimes\id_V\otimes\id_{Y^\ast}\otimes\ev_X}  \\
\mathsf{Z}^2(V)\otimes X  \ar[d]_{\pi^{\mathsf{Z}^2(V)}_X\otimes\id_X}  \\
X\otimes\mathsf{Z}(V)\otimes X^\ast\otimes X
  \ar[d]_{\id_X\otimes\id_{\mathsf{Z}(V)}\otimes\ev_X}
  \ar[uurrr]_{\;\;\;\;\;\;\;\;\id_X\otimes\pi^{\mathsf{Z}(V)}_Y\otimes\id_{X^\ast}\otimes\id_X} \\
X\otimes\mathsf{Z}(V)  \ar[rrr]_{\id_X\otimes\pi^{\mathsf{Z}(V)}_Y}
&&& X\otimes Y\otimes V\otimes Y^\ast
}
\end{equation*}
for all $X,Y\in\C$:
The top-left quadrangle commutes due to \cite[Equation (3.8)]{Shi17(b)}, while the bottom quadrangle is commutative since $\otimes$ is a bifunctor on $\C$.

Consequently,
Equation (\ref{eqn:sigma^Z(V)2}) is a result of the commutativity of the outside hexagon of this diagram.
\end{proof}

%\begin{proof}
%As explained in \cite[Section 4.1]{Shi17(b)}, the coend $\int^{Y\in\C} X^\ast\otimes Y\otimes X$ exists for all $X\in\C$ if $\C$ is a finite tensor category.
%Thus the claim holds according to \cite[Section 3]{Shi17(b)}.
%\end{proof}

\begin{remark}(\cite[Proposition 5.4]{ENO05})
If $\C$ is furthermore a fusion category, then
$$\mathbf{U}\circ I(X)\cong\bigoplus_{Y\in\mathcal{O}(\C)} Y\otimes X\otimes Y^\ast$$
holds for each $X\in\C$, where $I$ is a (right) adjoint functor to the forgetful functor $\mathbf{U}:\Z(\C)\rightarrow\C$.
\end{remark}

\subsection{The internal Hom functor and adjoints to the forgetful functor}\label{subsection:4.2}

Suppose $\M$ is a finite exact and indecomposable left module category over a finite tensor category $\C$. Then $\M$ is known to be \textit{closed}, which means that the functor
$$\C\rightarrow\M,\;\;X\mapsto X\ogreaterthan M$$
has a right adjoint for every $M\in\M$. Consequently the \textit{internal Hom functor}
$$\intHom(-,-):\M^\vee\times\M\rightarrow \C$$
could be defined (\cite[Section 3.2]{EO04}). It is furthermore described in \cite{Shi20} etc. that we could choose $\C$-module structures
\begin{eqnarray*}
\a_{X,M,N} &:& X\otimes\intHom(M,N)\xrightarrow{\cong}\intHom(M,X\ogreaterthan N)
\;\;\;\;\;\;\;\;(X\in\C),  \\
\b_{M,N,Y} &:& \intHom(M,N)\otimes Y^\ast\xrightarrow{\cong}\intHom(Y\ogreaterthan M,N)
\;\;\;\;\;\;\;\;(Y\in\C),  \\
\c_{X,M,N,Y} &:&
  X\otimes\intHom(M,N)\otimes Y^\ast
  \xrightarrow{\cong}\intHom(Y\ogreaterthan M,X\ogreaterthan N)
\;\;\;\;\;\;\;\;(X,Y\in\C)
\end{eqnarray*}
with relations listed in the following lemma,
and they are all dinatural in $(M,N)\in\M^\vee\times\M$. Note that $\c_{X,M,N,Y}$ is also dinatural in $(Y,X)\in\C^\vee\times\C$. Here our notations are the same with
\cite[Section 2.3]{Shi24} as a reference for a number of formulas on them.

\begin{lemma}\label{lem:abc}
There exist natural isomorphisms $\a$, $\b$ and $\c$ with notations above, such that:
\begin{itemize}
\item[(1)]
For all $X,Y\in\C$ and $M,N\in\M$,
$$\c_{X,M,N,Y}=\a_{X,Y\ogreaterthan M,N}\circ(\id_X\otimes \b_{M,N,Y})
=\b_{M,X\ogreaterthan N,Y}\circ(\a_{X,M,N}\otimes\id_{Y^\ast});$$
\item[(2)]
For all $X_1,X_2,Y_1,Y_2\in\C$ and $M,N\in\M$,
\begin{eqnarray}
\a_{X_1\otimes X_2,M,N}
&=&
\a_{X_1,M,X_2\ogreaterthan N}\circ(\id_{X_1}\otimes\a_{X_2,M,N}), \label{eqn:a=aa}  \\
\b_{M,N,Y_1\otimes Y_2}
&=&
\b_{Y_2\ogreaterthan M,N,Y_1}\circ(\b_{M,N,Y_2}\otimes\id_{Y_1^\ast}) \label{eqn:b=bb}
\end{eqnarray}
and
\begin{equation}\label{eqn:c=cc}
\c_{X_1\otimes X_2,M,N,Y_2^\ast\otimes Y_1^\ast}
=\c_{X_1,Y_2\ogreaterthan M,X_2\ogreaterthan N,Y_1}
\circ(\id_{X_1}\otimes\c_{X_2,M,N,Y_2}\otimes\id_{Y_1^\ast}).
\end{equation}
\end{itemize}
\end{lemma}

\begin{proof}
\begin{itemize}
\item[(1)]
The latter equality could be found as \cite[Equation (2.10)]{Shi24}, while the former one is in fact the definition of $\c$ (See \cite[Equation (2.18)]{Shi20} for example).

\item[(2)]
Equation (\ref{eqn:a=aa}) is introduced in \cite[Equation (2.10)]{Shi20} as the axiom for $\a$ being a left $\C$-module structure of the functor $\intHom(M,-)$. As for Equation (\ref{eqn:b=bb}), it is essentially proved in \cite[Lemmas 2.3 and A.1]{Shi20}, where the notation of $\b$ appearing is defined to be the inverse of ours in this paper and \cite{Shi24}.

In order to show Equation (\ref{eqn:c=cc}), let us verify the commutativity of the following diagram:
\begin{equation*}
\xymatrix{
X_1\otimes X_2\otimes\intHom(M,N)\otimes Y_2^\ast\otimes Y_1^\ast
\ar[rr]^{\id\otimes\id\otimes\b_{M,N,Y_2}\otimes\id}
\ar[d]_{\id\otimes\id\otimes\b_{M,N,Y_1\otimes Y_2}}
&& X_1\otimes X_2\otimes\intHom(Y_2\ogreaterthan M,N)\otimes Y_1^\ast
\ar[dll]^{\id\otimes\id\otimes\b_{Y_2\ogreaterthan M,N,Y_1}}
\ar[d]^{\id\otimes\a_{X_2,Y_2\ogreaterthan M,N}\otimes\id}  \\
X_1\otimes X_2\otimes \intHom(Y_1\ogreaterthan Y_2\ogreaterthan M,N)
\ar[d]_{\a_{X_1\otimes X_2,Y_1\ogreaterthan Y_2\ogreaterthan M,N}}
\ar[drr]^{\id\otimes\a_{X_2,Y_1\ogreaterthan Y_2\ogreaterthan M,N}}
&& X_1\otimes \intHom(Y_2\ogreaterthan M,X_2\ogreaterthan N)\otimes Y_1^\ast
\ar[d]^{\id\otimes\b_{Y_2\ogreaterthan M,X_2\ogreaterthan N,Y_1}}  \\
\intHom(Y_1\ogreaterthan Y_2\ogreaterthan M,X_1\ogreaterthan X_2\ogreaterthan  N)
&& X_1\otimes \intHom(Y_1\ogreaterthan Y_2\ogreaterthan M,X_2\ogreaterthan N)\;.
\ar[ll]_{\a_{X_1,Y_1\ogreaterthan Y_2\ogreaterthan M,X_2\ogreaterthan N}}
}
\end{equation*}
In fact, the commutativity of the two triangles on the left are due to Equations (\ref{eqn:a=aa}) and (\ref{eqn:b=bb}), and the right-sided quadrangle commutes since both compositions equal to $\id_{X_1}\otimes\c_{X_2,Y_2\ogreaterthan M,N,Y_1}$ according to (1).

Finally by (1) again, the compositions of the outside hexagon are respectively
$\c_{X_1\otimes X_2,M,N,Y_2^\ast\otimes Y_1^\ast}$ and
$\c_{X_1,Y_2\ogreaterthan M,X_2\ogreaterthan N,Y_1}
\circ(\id_{X_1}\otimes\c_{X_2,M,N,Y_2}\otimes\id_{Y_1^\ast})$.
Thus Equation (\ref{eqn:c=cc}) is implied by the commutativity of the diagram above.
\end{itemize}
\end{proof}

With the languages of ends, the internal Hom functor and the module structures above, we
introduce the construction of
$\Z(\rho^\ra)$ appearing in Lemma \ref{lem:adjpair0} according to \cite[Section 3.7]{Shi20}. Recall that for an object $\mathbf{F}=(F,s)\in\Rex_\C(\M)$, we have
$\Z(\rho^\mathrm{ra})(\mathbf{F})=(\rho^\mathrm{ra}(F),\sigma^\mathbf{F})\in\Z(\C),$
where
\begin{equation}\label{eqn:rho^ra(F)}
\rho^\mathrm{ra}(F)=\int_{M\in\M}\intHom(M,F(M))\;\in\C
\end{equation}
is an end of $\intHom(-,F(-))$,
and the half-braiding $\sigma^\mathbf{F}$ is the unique natural transformation such that the following diagram commutes for all objects $X\in\C$ and $M\in\M$:
\begin{equation}\label{eqn:sigma^F}
\xymatrix{
\rho^\mathrm{ra}(F)\otimes X
   \ar[rrr]^{\pi_{X\ogreaterthan M}\otimes\id_X\hspace{40pt}}
   \ar[dddd]_{\sigma^\mathbf{F}_X}
&&& \intHom(X\ogreaterthan M,F(X\ogreaterthan M))\otimes X
   \ar[d]^{\mathfrak{b}_{M,F(X\ogreaterthan M),X}^{-1}\otimes\id_X}  \\
&&& \intHom(M,F(X\ogreaterthan M))\otimes X^\ast\otimes X
   \ar[d]^{\id_{\intHom(M,F(X\ogreaterthan M))}\otimes\ev_X}  \\
&&& \intHom(M,F(X\ogreaterthan M))
   \ar[d]^{\intHom(\id_M,s^{-1}_{X,M})}  \\
&&& \intHom(M,X\ogreaterthan F(M))
   \ar[d]^{\mathfrak{a}^{-1}_{X,M,F(M)}}  \\
X\otimes \rho^\mathrm{ra}(F)
   \ar[rrr]^{\id_X\otimes \pi_{M}}
&&& X\otimes\intHom(M,F(M))\;\;,
}
\end{equation}
where $\pi$ is the universal dinatural transformation of the end $\rho^\mathrm{ra}(F)$.
We remark that the commutativity of (\ref{eqn:sigma^F}) is due to \cite[Lemma 3.12]{Shi20}.

At the final of this subsection, we recall the following lemma in order to provide our usual choice of an adjoint functor $K$ in Definition \ref{def:dual-inds}:
\begin{lemma}(\cite[Theorem 3.14]{Shi20} and \cite[Corollary 3.15]{Shi20})\label{lem:adjpair}
Let $\mathbf{U}_\M:\Z(\C_\M^\ast)\rightarrow\C_\M^\ast$ be the forgetful functor. Then:
\begin{itemize}
\item[(1)]
\;$\Big(\mathbf{U}_\M\circ\Omega:\Z(\C)\rightarrow\C_\M^\ast,\;\;\;\;
  \Z(\rho^\mathrm{ra}):\C_\M^\ast\rightarrow\Z(\C)\Big)$
is an adjoint pair;
\item[(2)]
\;$\Big(\mathbf{U}_\M:\Z(\C_\M^\ast)\rightarrow\C_\M^\ast,\;\;\;\;
   \Omega\circ\Z(\rho^\mathrm{ra}):\C_\M^\ast\rightarrow\Z(\C_\M^\ast)\Big)$
is an adjoint pair.
\end{itemize}
\end{lemma}

\subsection{An equality between indicators of certain objects in $\C_\M^\ast$ and $\C$}\label{subsection:mainthm}
We still suppose $\C$ is a finite tensor category. It is known in \cite[Corollary 7.10.5]{EGNO15} that if $\M$ is a finite left $\C$-module category, we could assume that $\M=\C_A$, the category of right $A$-modules in $\C$, for some algebra $A\in\C$. In this case, there is a tensor equivalence
\begin{equation}\label{eqn:bimodcat}
{}_A \C_A\rightarrow \C_\M^\ast,\;\;\;\;P\mapsto (-)\otimes_A P,
\end{equation}
according to the proof of \cite[Proposition 3.23]{EO04}.

This subsection is devoted to studying indicators of certain objects in the dual category $\C_{\C_A}^\ast$ under the equivalence (\ref{eqn:bimodcat}). Specifically, for each object $M\in\C_A$, there is an $A$-bimodule structure on $A\otimes M\in\C$ via the regular left $A$-module and the right $A$-module structure on $M$. This would be sent by (\ref{eqn:bimodcat}) to the functor $(-)\otimes_A(A\otimes M)\in\C_{\C_A}^\ast$, which is canonically isomorphic to $-\otimes M$ with $\C$-module structure as the associativity constraint of $\C$ (and hence abbreviated).

\begin{lemma}
Let $\C$ be a finite tensor category. Suppose that $A$ is an algebra in $\C$ such that $\M=\C_A$ is an indecomposable $\C$-module category which is exact. For any $M\in \C_A$,
denote
$$\mathsf{E}:=\rho^\mathrm{ra}(-\otimes M)
=\int_{N\in\C_A}\intHom(N,N\otimes M)\;\in\C$$
as in (\ref{eqn:rho^ra(F)}) with half-braiding $\sigma^{-\otimes M}$ defined such that Diagram (\ref{eqn:sigma^F}) commutes.
Then the following diagram commutes for each $X\in\C$:
\begin{equation}\label{eqn:sigma^E}
\xymatrix{
\mathsf{E}\otimes X
   \ar[rrr]^{\pi^\mathsf{E}_{X\otimes N}\otimes\id_X\hspace{40pt}}
   \ar[dd]_{\sigma^{-\otimes M}_X}
&&& \intHom(X\otimes N,X\otimes N\otimes M)\otimes X
   \ar[d]^{\c_{X,N,N\otimes M,X}^{-1}\otimes\id_X}  \\
&&& X\otimes\intHom(N,N\otimes M)\otimes X^\ast\otimes X
   \ar[d]^{\id_X\otimes\id_{\intHom(N,N\otimes M)}\otimes\ev_X}  \\
X\otimes \mathsf{E}
   \ar[rrr]^{\id_X\otimes \pi^\mathsf{E}_{X}\hspace{40pt}}
&&& X\otimes\intHom(N,N\otimes M)\;.
}
\end{equation}
\end{lemma}

\begin{proof}
By comparison to Diagram (\ref{eqn:sigma^F}), it suffices to show that the following diagram commutes for all $X\in\C$ and $N\in\C_A$:
\begin{equation}\label{eqn:sigma^-otimesM}
\xymatrix{
\intHom(X\otimes N,X\otimes N\otimes M)\otimes X
  \ar[d]_{\b_{X\otimes N,X\otimes N\otimes M,X}^{-1}\otimes \id_X}
  \ar[drrr]^{\c_{X,N,N\otimes M,X}^{-1}\otimes\id_X}  \\
\intHom(N,X\otimes N\otimes M)\otimes X^\ast\otimes X
  \ar[rrr]_{\a_{X,N,N\otimes M}^{-1}\otimes\id_{X^\ast}\otimes\id_X}
  \ar[d]_{\id_{\intHom(N,X\otimes N\otimes M)}\otimes\ev_X}
&&& X\otimes\intHom(N,N\otimes M)\otimes X^\ast\otimes X \;,
  \ar[ddlll]^{\;\;\;\;\;\;\;\;\;\;\;\;\;\;\;\;\id_X\otimes\id_{\intHom(N,N\otimes M)}\otimes\ev_X}  \\
\intHom(N,X\otimes N\otimes M)
  \ar[d]_{\a_{X,N,N\otimes M}^{-1}}  \\
X\otimes\intHom(N,N\otimes M)
}
\end{equation}
where the $\C$-module structure of $-\otimes M$ is abbreviated.
Specifically, the top triangle is commutative due to Lemma \ref{lem:abc}(1), and the bottom quadrangle commutes as $\otimes$ is a bifunctor on $\C$.
\end{proof}

\begin{remark}
In fact, for an arbitrary $\mathbf{F}\in\Rex_\C(\M)$, a diagram analogous to (\ref{eqn:sigma^-otimesM}) for $\sigma^{\mathbf{F}}$ commutes as well, which follows from the commutativity of (\ref{eqn:sigma^F}).
\end{remark}

\begin{proposition}\label{prop:endsiso}
Let $\C$ be a finite tensor category. Suppose that $A$ is an algebra in $\C$ such that $\M=\C_A$ is an indecomposable $\C$-module category which is exact. Then for any $M\in \C_A$,
$$\left(\int_{N\in\C_A} \intHom(N,N\otimes M),\sigma^{-\otimes M}\right)
\cong
\left(\mathsf{Z}(\intHom(A,M)),\sigma^{\mathsf{Z}(\intHom(A,M))}\right)$$
as objects in $\Z(\C)$.
\end{proposition}

\begin{proof}
Recall from \cite[Lemma 3.2]{Ost03} that there is an adjoint pair $(L,R)$ defined as
\begin{equation}\label{eqn:adjpair(L,R)}
\left\{
\begin{array}{lll}
L :\C\rightarrow\C_A, & X\mapsto X\otimes A; & \\
R :\C_A\rightarrow\C, & N\mapsto N & \text{(the forgetful functor)},
\end{array}
\right.
\end{equation}
with the unit and counit
\begin{equation}\label{eqn:adjunitcounit(L,R)}
\eta:\Id_\C\rightarrow -\otimes A=RL,\;\;\;\;\mu:-\otimes A=LR\rightarrow\Id_{\C_A}.
\end{equation}
Note that $\mu$ could be chosen such that $\mu_N$ is the right $A$-module structure on every $N\in\C_A$ according to the proof of \cite[Lemma 7.8.12]{EGNO15}, and thus
\begin{equation}\label{eqn:adjcounit=mod}
\id_X\otimes\mu_N=\mu_{X\otimes N}\;\;\;\;\;\;\;\;
(\forall X\in\C,\;\;\forall N\in\C_A)
\end{equation}
as the right $A$-module structure of $X\otimes N$.

Suppose $M\in\C_A$. We apply \cite[Lemma 2.2]{Shi20} (which is the dual form of \cite[Lemma 3.9]{BV12}) to the adjunction (\ref{eqn:adjpair(L,R)})
and the functor
$$\C\times(\C_A)^\vee\rightarrow\C,
\;\;\;(X,N)\mapsto\intHom(N,X\otimes M),$$
and then obtain an isomorphism $f$:
\begin{equation}\label{eqn:endsiso-f}
\mathsf{E}=\int_{N\in\C_A} \intHom(N,N\otimes M)
\overset{f}{\cong} \int_{X\in\C} \intHom(X\otimes A,X\otimes M)=\mathsf{E}'.
\end{equation}
We remark that these ends do exist according to \cite[Theorem 3.4]{Shi20}.

However, recall that there is an isomorphism
$$\mathfrak{c}_{Y,A,M,X}:Y\otimes \intHom(A,M)\otimes X^\ast
\cong\intHom(X\otimes A,Y\otimes M),$$
introduced in Subsection \ref{subsection:4.2}, which is dinatural in $(X,Y)\in\C^\vee\times\C$.
%\textcolor{red}{such that the diagram
%\begin{equation}
%\xymatrix{
%\Hom_\C(Z,\,\intHom(X\otimes A,Y\otimes M))  \ar[rr]^{\phi_{Z,X\otimes A,Y\otimes M}}
%  \ar@{-->}[d]_{\Hom_\C(\id_Z,\omega_{X,Y})}
%&& \Hom_{\C_A}(Z\otimes X\otimes A,Y\otimes M)
%  \ar[d]^{\lambda:\;\text{adjunction\;(\ref{eqn:adjpair(L,R)})}}  \\
%\Hom_\C(Z\otimes X, Y\otimes M)
%&& \Hom_\C(Z,Y\otimes M\otimes X^\ast)  \ar[ll]^{\text{duality}}
%}
%\end{equation}
%commutes for all $X,Y,Z\in\C$.}
Consequently,
another isomorphism $g$:
\begin{equation}\label{eqn:endsiso-g}
\mathsf{E}'=\int_{X\in\C} \intHom(X\otimes A,X\otimes M)
\overset{g}{\cong}
  \int_{X\in\C} X\otimes \intHom(A,M)\otimes X^\ast
  =\mathsf{Z}(\intHom(A,M))=\mathsf{C}
\end{equation}
between ends could be obtained.

Furthermore, let us explain that the isomorphisms $f$ (\ref{eqn:endsiso-f}) and $g$ (\ref{eqn:endsiso-g}) between ends could be chosen
such that the following diagram commutes for all objects $N\in\C_A$ and $X\in\C$:
\begin{equation}\label{eqn:endsisodiagram}
\xymatrix{
\intHom(N,N\otimes M) \ar[d]_{\intHom(\mu_N,\id)}
& \int_{N\in\C_A}\intHom(N,N\otimes M) \ar@{-->}[d]_{f}^{\cong}
  \ar[l]_{\pi^\mathsf{E}_N} \ar[r]^{\pi^\mathsf{E}_{X\otimes A}}
& \intHom(X\otimes A,X\otimes A\otimes M)  \\
\intHom(N\otimes A,N\otimes M)
  \ar[d]_{\c_{N,A,M,N}^{-1}}
& \int_{X\in\C} \intHom(X\otimes A,X\otimes M) \ar@{-->}[d]_{g}^{\cong}
  \ar[l]_{\pi^{\mathsf{E}'}_N} \ar[r]^{\pi^{\mathsf{E}'}_X}
& \intHom(X\otimes A,X\otimes M) \ar[u]_{\intHom(\id,\eta_X\otimes\id_M)}  \\
N\otimes \intHom(A,M)\otimes N^\ast
& \int_{X\in\C} X\otimes\intHom(A,M)\otimes X^\ast \ar[l]^{\pi^{\mathsf{C}}_N} \ar[r]_{\pi^{\mathsf{C}}_X}
& X\otimes \intHom(A,M)\otimes X^\ast\;, \ar[u]_{\c_{X,A,M,X}}
}
\end{equation}
where
$$\pi^\mathsf{E}:\mathsf{E}\rightarrow\intHom(-,-\otimes M),
\;\;\pi^{\mathsf{E}'}:\mathsf{E}'\rightarrow\intHom(-\otimes A,-\otimes M)
\;\;\text{and}\;\;
\pi^{\mathsf{C}}:
\mathsf{C}%=\mathsf{Z}(\intHom(A,M))
\rightarrow(-)\otimes \intHom(A,M)\otimes (-)^\ast$$
denote the corresponding universal dinatural transformations for the ends respectively:
%and
%\begin{equation}\label{eqn:adjunitcounit(L,R)}
%\eta:\Id_\C\rightarrow -\otimes A=RL,\;\;\;\;\e:-\otimes A=LR\rightarrow\Id_{\C_A}
%\end{equation}
%are the unit and counit of the adjoint pair $(L,R)$ defined in (\ref{eqn:adjpair(L,R)}):
More explicitly, the commutativity of the top rectangles of (\ref{eqn:endsisodiagram}) could be satisfied due to the paragraph after \cite[Lemma 2.2]{Shi20}, while the bottom rectangles of (\ref{eqn:endsisodiagram}) are commutative by \cite[Proposition 1 in Section IX.7]{MacL98}.

We point out the left rectangles of the commuting diagram (\ref{eqn:endsisodiagram}) as follows, where the isomorphism $g\circ f$ is denoted by $h$ for simplicity:
\begin{equation}\label{eqn:endsisodiagram2}
\xymatrix{
\mathsf{E} \ar[d]_{\pi^\mathsf{E}_N}  \ar[rr]^{h}
&& \mathsf{C}   \ar[d]^{\pi^{\mathsf{C}}_N} \\
\intHom(N,N\otimes M)  \ar[r]_{\intHom(\mu_N,\id)}
& \intHom(N\otimes A,N\otimes M)
& N\otimes\intHom(A,M)\otimes N^\ast \;.
\ar[l]^{\c_{N,A,M,N}}
}
\end{equation}
On the other hand, one might apply \cite[Theorem 1 in Section IV.1]{MacL98} to the adjunction (\ref{eqn:adjpair(L,R)}) to obtain that
\begin{equation}\label{eqn:adjcounitunit=id}
\mu_N\circ\eta_N=\id_N,
\;\;\;\;\text{and hence}\;\;\;\;
\intHom(\eta_N,\id)\circ\intHom(\mu_N,\id)=\id
\;\;\;\;\;\;\;\;(\forall N\in \C_A).
\end{equation}

%\textcolor{red}{Then for every $N\in \C_A$, the commutativity
%$$\intHom(\e_N,\id)\circ\pi^\mathsf{E}_N
%=\omega_{N,N}^{-1}\circ\pi^{\mathsf{Z}(M)}_N\circ h$$
%of the left outside rectangle of (\ref{eqn:endsisodiagram}) provides the following commuting diagram in $\C$:
%\begin{equation}\label{eqn:endsisodiagram2}
%\xymatrix{
%\mathsf{E} \ar[d]_{\pi^\mathsf{E}_N}
%&&& \mathsf{C}  \ar[lll]_{h^{-1}} \ar[d]^{\pi^{\mathsf{C}}_N} \\
%\intHom(N,N\otimes M)
%&& \intHom(N\otimes A,N\otimes M) \ar[ll]_{\intHom(\eta_N,\id)}
%& N\otimes\intHom(A,M)\otimes N^\ast \;.
%\ar[l]_{\c_{N,A,M,N}}
%}
%\end{equation}
%}

Now we focus on the following diagram (\ref{eqn:endsisodiagram3}) and try to prove its commutativity for all $X\in\C$ and $N\in\C_A$.
Please note that due to Equations (\ref{eqn:adjcounitunit=id}), if the arrow $\id\otimes\intHom(\mu_N,\id)$ in Diagram (\ref{eqn:endsisodiagram3}) is replaced by the dotted arrow $\id\otimes\intHom(\eta_N,\id)$ with inverse direction, then the new diagram will be also commutative as a consequence.
\begin{equation}\label{eqn:endsisodiagram3}\tiny
\xymatrix{
\mathsf{E}\otimes X
  \ar[rr]^{h\otimes\id}  \ar[d]_{\pi^\mathsf{E}_{X\otimes N}\otimes\id}
&& \mathsf{C}\otimes X
  \ar[d]^{\pi^\mathsf{C}_{X\otimes N}\otimes\id}  \\
\intHom(X\otimes N,X\otimes N\otimes M)\otimes X
  \ar[r]^{\intHom(\mu_{X\otimes N},\id)\otimes\id}
  \ar[d]_{\c_{X,N,N\otimes M,X}^{-1}\otimes\id}
& \intHom(X\otimes N\otimes A,X\otimes N\otimes M)\otimes X
  \ar[d]_{\c_{X,N\otimes A,N\otimes M,X}^{-1}\otimes\id}
& X\otimes N\otimes\intHom(A,M)\otimes N^\ast\otimes X^\ast\otimes X
  \ar[l]_{\mathfrak{c}_{X\otimes N,A,M,X\otimes N}}
  \ar[dd]^{\id\otimes\id\otimes\id\otimes\id\otimes\ev_X}
  \ar[dl]^{\id\otimes\c_{N,A,M,N}\otimes\id\otimes\id}  \\
X\otimes\intHom(N,N\otimes M)\otimes X^\ast\otimes X
  \ar[r]^{\id\otimes\intHom(\mu_{N},\id)\otimes\id\otimes\id}
  \ar[d]_{\id\otimes\id\otimes\ev_X}
& X\otimes\intHom(N\otimes A,N\otimes M)\otimes X^\ast\otimes X
  \ar[d]_{\id\otimes\id\otimes\ev_X}  \\
X\otimes\intHom(N,N\otimes M)
  \ar[r]^{\id\otimes\intHom(\mu_N,\id)}
& X\otimes\intHom(N\otimes A,N\otimes M)
  \ar@<1.5ex>@{-->}[l]^{\id\otimes\intHom(\eta_N,\id)}
& X\otimes N\otimes\intHom(A,M)\otimes N^\ast
  \ar[l]_{\id\otimes\c_{N,A,M,N}}  \\
X\otimes\mathsf{E}
  \ar[u]^{\id\otimes\pi^\mathsf{E}_N}  \ar[rr]_{\id\otimes h}
&& X\otimes\mathsf{C} \;.
  \ar[u]_{\id\otimes\pi^\mathsf{C}_N}
}
\end{equation}
Our reasons for the commutativity of the cells in this diagram (with original arrow $\id\otimes\intHom(\mu_N,\id)$) are stated respectively as follows:
\begin{itemize}
\item
The two (top and bottom) pentagons: Diagram (\ref{eqn:endsisodiagram2}) commutes;
\item
The upper rectangle:
This is
obtained by the naturality of $\c_{X,-,N,N\otimes M,X}$ for the morphism $\mu_N$ in $\C$, where Equation (\ref{eqn:adjcounit=mod}) will be used;
\item
The triangle: Equation (\ref{eqn:c=cc});
\item
The remaining two quadrangles (the right trapezoid and the lower rectangle): $\otimes$ is a bifunctor on $\C$.
\end{itemize}
Then consider the following two commuting diagrams, which are (\ref{eqn:sigma^E}) as well as (\ref{eqn:sigma^Z(V)2}) where $\mathsf{Z}(\intHom(A,M))=\mathsf{C}$:
$$\tiny
\xymatrix{
\mathsf{E}\otimes X
   \ar@/_35ex/[dddd]_{\sigma^{\mathsf{E}}_X}
   \ar[d]^{\pi^\mathsf{E}_{X\otimes N}\otimes\id\hspace{40pt}}
&& \mathsf{C}\otimes X
   \ar[d]_{\pi^\mathsf{C}_{X\otimes N}\otimes\id}
   \ar@/^35ex/[dddd]^{\sigma^{\mathsf{E}}_X}
\\
 \intHom(X\otimes N,X\otimes N\otimes M)\otimes X
   \ar[d]^{\c_{X,N,N\otimes M,X}^{-1}\otimes\id}
&& X\otimes N\otimes\intHom(A,M)\otimes N^\ast\otimes X^\ast\otimes X
   \ar[dd]_{\id\otimes\id\otimes\id\otimes\id\otimes\ev_X}
\\
 X\otimes\intHom(N,N\otimes M)\otimes X^\ast\otimes X
   \ar[d]^{\id\otimes\id_{\intHom(N,N\otimes M)}\otimes\ev_X}
&  \text{and}
\\
 X\otimes\intHom(N,N\otimes M)
&& X\otimes N\otimes\intHom(A,M)\otimes N^\ast
\\
X\otimes \mathsf{E}
   \ar[u]_{\id\otimes \pi^\mathsf{E}_{X}}
&& X\otimes\mathsf{C}  \;,
   \ar[u]^{\id\otimes\pi^\mathsf{C}_N}
}
$$
If we put them at the left and right sides of (\ref{eqn:endsisodiagram}) respectively, then we could conclude by the commutativity (with the dotted arrow $\id\otimes\intHom(\eta_N,\id)$ instead) that
$$(\id_X\otimes\pi^\mathsf{E}_N)\circ\sigma^\mathsf{E}_X
=(\id_X\otimes\pi^\mathsf{E}_N)
\circ\left[(\id_X\otimes h)\circ\sigma^\mathsf{C}_X\circ(h^{-1}\otimes\id_X)\right]
\;\;\;\;(\forall X\in\C,\;\forall N\in\C_A).$$
It follows from %Fubini Theorem and
the universal property for ends that
$$\sigma^\mathsf{E}_X\circ(h\otimes\id_X)
=(\id_X\otimes h)\circ\sigma^\mathsf{C}_X
\;\;\;\;\;\;\;\;(\forall X\in\C).$$
Finally as a result, $(\mathsf{E},\sigma^\mathsf{E})$ and $(\mathsf{C},\sigma^\mathsf{C})$ are isomorphic objects in $\Z(\C)$ as desired.
\end{proof}

%\textcolor{red}{Should this equiv induce pivotal stru of ${}_A \C_A$????}

The following corollary could be a generalization of the last isomorphism in \cite[Remark 6]{FS23}:

\begin{corollary}\label{cor:endsiso2}
Let $\C$ be a finite tensor category. Suppose that $A$ is an algebra in $\C$ such that $\M=\C_A$ is an indecomposable $\C$-module category which is exact. Then for any $M\in \C_A$,
\begin{equation*}%\label{eqn:endsiso2}
\Z(\rho^\mathrm{ra})(-\otimes M)
=\left(\int_{N\in\C_A} \intHom(N,N\otimes M),\sigma^{-\otimes M}\right)
\cong(\mathsf{Z}(M),\sigma^{\mathsf{Z}(M)}),
\end{equation*}
where $\Z(\rho^\mathrm{ra})$ is defined by (\ref{eqn:rho^ra(F)}) and (\ref{eqn:sigma^F}) as a right adjoint of $\mathbf{U}_{\C_A}\circ\Omega:\Z(\C)\rightarrow\C_{\C_A}^\ast$.
\end{corollary}

\begin{proof}
As the central Hopf comonad $\mathsf{Z}$ is an endofunctor on $\C$, the isomorphism $\sigma^{\mathsf{Z}(V)}_X$ defined in Lemma \ref{lem:centralHopfcomonadZ(C)} will be natural in $V,X\in\C$. Thus by Proposition \ref{prop:endsiso}, we only need to show that $M\cong\intHom(A,M)$ as objects in $\C$. In fact due to Yoneda Lemma, this is known by the following isomorphism
$$\Hom_\C(X,M)\overset{(\ref{eqn:adjpair(L,R)})}{\cong}\Hom_{\C_A}(X\otimes A,M)
\cong\Hom_\C(X,\intHom(A,M))$$
which is natural in $X\in\C$.
\end{proof}

\begin{theorem}\label{thm:bimod-inds}
Let $\C$ be a spherical fusion category over $\mathbb{C}$. Suppose that $A$ is an algebra in $\C$ such that $\M=\C_A$ is an indecomposable $\C$-module category which is semisimple. Then for any $M\in \C_A$, the $n$-th indicator of
$-\otimes M\in\C_{\C_A}^\ast$ is
$$\nu_n(-\otimes M)=\nu_n(M)\;\;\;\;\;\;\;\;(\forall n\geq1),$$
where:
\begin{itemize}
\item
The $\C$-module functor $-\otimes M$ is canonically isomorphic to $(-)\otimes_A(A\otimes M)$ with structure as the associativity constraint of $\C$;

\item
The object $M$ in the right side is regarded in $\C$ (as the image under the forgetful functor from $\C_A$ to $\C$).
\end{itemize}
\end{theorem}

%\begin{remark}
%\textcolor{red}{
%Evidently, if we replace $-\otimes M$ by a suitable $\C$-module functor $F\in\C_\M^\ast$, the conclusion in Theorem \ref{thm:bimod-inds} would be
%\begin{equation}\label{eqn:functor-inds}
%\nu_n(F)=\nu_n(F(A)).
%\end{equation}
%}
%\end{remark}

\begin{proof}
It is clear that $\C_A$ is finite according to \cite[Exercise 7.8.16]{EGNO15}, and we
compute the indicators in the left side of this equation
by Definition \ref{def:dual-inds}. Choose %the adjoint functor
$K_{\C_A}=\Z(\rho_{\C_A}^\mathrm{ra})$ as in Corollary \ref{cor:endsiso2},
which is a two-sided adjoint of $\mathbf{U}_\M\circ\Omega$.
Thus for any $n\geq1$,
\begin{equation}\label{eqn:nu_n(-otimesM)}
\nu_n(-\otimes M)
=\frac{1}{\dim(\C)} \ptr\left(\theta_{\Z(\rho^\mathrm{ra})(-\otimes M)}^n\right).
\end{equation}

On the other hand,
the $n$-th indicator of $M\in\C$ is computed according to Lemma \ref{lem:indsorigin}
as
\begin{equation}\label{eqn:nu_n(M)}
\nu_n(M)=\frac{1}{\dim(\C)}
\ptr\left(\theta_{(\mathsf{Z}(M),\sigma^{\mathsf{Z}(M)})}^n\right),
\end{equation}
because we know by Lemma \ref{lem:centralHopfcomonadZ(C)} that
$I:\C\rightarrow\Z(\C),\;\;V\mapsto (\mathsf{Z}(V),\sigma^{\mathsf{Z}(V)})$ is a right adjoint to the forgetful functor $\mathbf{U}:\Z(\C)\rightarrow\C$.

Finally, it follows from Lemma \ref{lem:independadj} that the right sides of Equations (\ref{eqn:nu_n(-otimesM)}) and (\ref{eqn:nu_n(M)}) coincide, since
$\Z(\rho^\mathrm{ra})(-\otimes M)\cong (\mathsf{Z}(M),\sigma^{\mathsf{Z}(M)})$ in $\Z(\C)$ due to Corollary \ref{cor:endsiso2}.
\end{proof}

%\begin{corollary}
%\textcolor{red}{$\nu_n(\mathbf{F})=?(\mathbf{F}(\1))$?? for what kind of $\mathbf{F}$?}
%\end{corollary}

%\subsection{(Compatibility of the definition)}\label{subsection:compatible}
We end this subsection by considering
the particular case when $\M=\C$ as the regular left $\C$-module category. It is known that there is a tensor equivalence
$$\C\rightarrow \C_\C^\ast,\;\;X\mapsto -\otimes X,$$
according to \cite[Example 7.12.3]{EGNO15} for example. Under this equivalence, we
remark that indicators of objects in $\C_\C^\ast$ coincides with those in $\C$ originally.

\begin{corollary}
Let $\C$ be a spherical fusion category over $\mathbb{C}$. Then for each $X\in\C$,
\begin{equation}\label{eqn:indscompatible}
\nu_n(-\otimes X)=\nu_n(X)\;\;\;\;\;\;\;\;(\forall n\geq1),
\end{equation}
where $-\otimes X$ is an object in $\C_\C^\ast$ with $\C$-module structure as the inverse of the associativity constraint in $\C$.
\end{corollary}

\begin{proof}
%\begin{itemize}
%\item[(0)]
We show this with the help of Theorem \ref{thm:bimod-inds}. Since the left regular $\C$-module category could be identified with the category $\C_\1$ of right $\1$-modules in $\C$, where the unit object $\1\in\C$ is trivially an algebra. Moreover, the $\C$-module functor $-\otimes X$ is clearly isomorphic to $-\otimes_\1(\1\otimes X)$ for each $X\in\C$.
\end{proof}

\subsection{Frobenius-Schur exponent of the dual fusion cateogry $\C_\M^\ast$}
%\textcolor{red}{(DELETE THIS SUBSECTION?)}

The Frobenius-Schur exponents of a pivotal monoidal category and its objects are defined in \cite[Definition 5.1]{NS07(b)}. In order to recall these definitions, let us use convention $\min \varnothing=\infty$ in this paper.

For any object $X$ in a pivotal monoidal category $\C$ with pivotal structure $j$, the \textit{Frobenius-Schur exponent of $X$} is defined as
\begin{equation}\label{eqn:objFSexp}
\FSexp(X):=\min\{n\geq1\mid\nu_n(X)=\tr(j_X)\},
\end{equation}
where $\tr$ denotes the (left) categorical trace. Moreover, the \textit{Frobenius-Schur exponent} of $\C$ is defined to be
\begin{equation}\label{eqn:FSexp}
\FSexp(\C):=\min\{n\geq1\mid \forall X\in\C,\;\;\nu_n(X)=\tr(j_X)\}.
\end{equation}

Since $\C_\M^\ast$ is not known to be consequently pivotal, the Frobenius-Schur exponents of its objects might not be defined frequently.
However, when the dual fusion category $\C_\M^\ast$ is assumed to be pivotal such that Schauenburg's equivalence $\Omega$ preserves the pivotal structure, we could compare the Frobenius-Schur exponents of $\C$ and $\C_\M^\ast$ by a combination of \cite[Corollary 4.4]{NS07(a)} and \cite[Corollary 7.8]{NS07(b)}:

\begin{corollary}\label{cor:FSexpscoincide}
Let $\C$ be any spherical fusion category over $\mathbb{C}$ with spherical structure $j$, and let
$\M$ be a finite indecomposable left $\C$-module category which is semisimple.
%Suppose $u$ is the Drinfeld morphism of $\Z(\C)$.
Suppose $\C_\M^\ast$ is also spherical %\textcolor{red}{(pivotal?)} %with structure $j'$
and $\Omega:\Z(\C)\approx\Z(\C_\M^\ast)$ preserves the pivotal structure. Then
%\textcolor{red}{(TRUE?)}
%\begin{itemize}
%\item[(1)]
%For any object $\mathbf{F}=(F,s)\in\C_\M^\ast$, its Frobenius-Schur exponent $\FSexp(\mathbf{F})$ in the senses of (\ref{eqn:dual-objFSexp}) and (\ref{eqn:objFSexp}) coincide;
%
%\item[(2)]
%The Frobenius-Schur exponent $\FSexp(\C_\M^\ast)$ of $\C_\M^\ast$ in the senses of (\ref{eqn:dual-FSexp}) and (\ref{eqn:FSexp}) coincide.
%\end{itemize}
\begin{equation}\label{eqn:FSexpscoincide}
\FSexp(\C_\M^\ast)=\FSexp(\C),
\end{equation}
where $\FSexp(\C_\M^\ast)$ is the Frobenius-Schur exponent of the spherical fusion category $\C_\M^\ast$ in the sense of (\ref{eqn:FSexp}).
\end{corollary}

\begin{proof}
%We still denote $N:=\FSexp(\C)$, which is a positive integer according to \cite[Theorem 5.5]{NS07(b)}.
%\begin{itemize}
%\item[(1)]
%Suppose $\mathbf{F}\in\C_\M^\ast$, and
%note by Proposition \ref{prop:indscoincide} that its indicators $\nu_n(\mathbf{F})\;(n\geq1)$ in the senses of (\ref{eqn:indsdef}) and (\ref{eqn:dual-inds}) coincide under the assumptions.
%
%On the other hand, since $\Omega$ is an equivalence, there exists an object $X$ in $\C$ such that $\mathbf{F}\cong\Omega(X)$ \textcolor{red}{(Wrong!)}. Then it is followed by Lemma \ref{lem:independadj} and Corollary \ref{cor:monoidalfunspreseveptr}(2) that
%$$\ptr(\id_\mathbf{F})=\ptr(\id_{\Omega(X)})
%=\ptr(\Omega(\id_X))=\ptr(\id_X).$$
%\item[(2)]
%\end{itemize}

We know according to \cite[Corollary 4.4]{NS07(a)} that $\Omega:\Z(\C)\approx\Z(\C_\M^\ast)$ would preserves the indicators of all the corresponding objects, as it is assumed to preserve the pivotal structure. Consequently, the definition of the Frobenius-Schur exponent (\ref{eqn:FSexp}) implies that
$$\FSexp(\Z(\C_\M^\ast))=\FSexp(\Z(\C)).$$
On the other hand,
it is clear by \cite[Corollary 7.8]{NS07(b)} that
$$\FSexp(\C_\M^\ast)=\FSexp(\Z(\C_\M^\ast))
\;\;\;\;\;\;\;\;\text{and}\;\;\;\;\;\;\;\;\FSexp(\C)=\FSexp(\Z(\C))$$
both hold. Therefore we obtain Equation (\ref{eqn:FSexpscoincide}) as a conclusion.
\end{proof}

%The regular object $R$ of $\C$ is usually defined with Frobenius-Perron dimensions. In order to make $R$ become an object of $\C$, we should assume that $\C$ is an integral spherical fusion category, which would be equivalent to the category of representations of a semisimple quasi-Hopf algebra (\cite[Theorem 8.33]{ENO05} or \cite[Propsition 2.6]{EO04}). In this case, Frobenius-Perron dimensions and categorical dimensions are equal for all objects in $\C$, and
%$$R=\bigoplus_{X\in\mathcal{O}(\C)} \FPdim(X)X=\bigoplus_{X\in\mathcal{O}(\C)} \dim(X)X.$$
%
%\begin{conjecture}
%Let $\C$ be an integral spherical fusion category, and let $\M$ be an indecomposable $\C$-module category which is semisimple. Then
%$$\nu_n(\mathbf{R})=\nu_n(R),$$
%where $\mathbf{R}$ and $R$ are regular elements of $\C_\M^\ast$ and $\C$ respectively.
%\end{conjecture}

\section{Applications: Indicators and exponents of semisimple Hopf algebras and its left partial dualization}\label{section5}

%\section{Specific cases for the module categories}

In this section, our results are applied to the case when $\C=\Rep(H)$ is the category of finite-dimensional representations of a semisimple Hopf algebra $H$ over $\mathbb{C}$. The comultiplication and counit of $H$ are denoted by $\Delta$ and $\varepsilon$ respectively, and the Sweedler notation is frequently used for coproducts in $H$.

\subsection{Canonical pivotal structures of integral fusion categories, and indicators for semisimple quasi-Hopf algebras}

Recall in \cite[Definition 9.6.1]{EGNO15} that a fusion category  is said to be \textit{integral} if the Frobenius-Perron dimensions of all the objects in $\C$ are integers. An integral fusion category is evidently \textit{weakly integral}. Moreover, one could combine \cite[Propositions 8.24 and 8.23]{ENO05} to know that an integral fusion category over $\mathbb{C}$ is \textit{pseudo-unitary} (\cite[Section 8.4]{ENO05}) and hence admits a unique spherical structure satisfying
\begin{equation}\label{eqn:canonicalptr}
\ptr(\id_X)=\FPdim(X)%\textcolor{red}{\text{ What is FPdim? When = dim?}}
\;\;\;\;\;\;\;\;(\forall X\in\C),
\end{equation}
which is referred to as the \textit{canonical pivotal structure} of $\C$ by \cite[Section 2]{NS08} (it is also called the canonical spherical structure in \cite[Section 9.5]{EGNO15}).

\begin{lemma}\label{lem:intfusioncatcenter}
Let $\C$ be an integral %\textcolor{red}{(pseudo-unitary?)}
fusion category over $\mathbb{C}$ with canonical pivotal structure $j$.
Then:
\begin{itemize}
\item[(1)]
Its center $\Z(\C)$ is integral as well;
\item[(2)]
The spherical structure on $\Z(\C)$ induced by $j$ via (\ref{eqn:Z(C)pivotal}) is also the canonical pivotal structure.
\end{itemize}
\end{lemma}

\begin{proof}
%\textcolor{red}{(2)? Def (\ref{eqn:Z(C)pivotal}) and ptr; The forgetful $\Z(\C)\rightarrow\C$ preserves FPdims.}
\begin{itemize}
\item[(1)]
This is mentioned in the proof of \cite[Proposition 9.6.11]{EGNO15}, as the forgetful functor $\Z(\C)\rightarrow\C$ preserves Frobenius-Perron dimensions.
\item[(2)]
Recall in (\ref{eqn:Z(C)pivotal}) that the spherical structure on $\Z(\C)$ induced is the same as $j$ at the objects in $\C$. Thus for each object $\mathbf{V}=(V,\sigma^V)\in\Z(\C)$, we find
$$\ptr(\id_\mathbf{V})
= \tr(j_{\mathbf{V}}) = \tr(j_V) =\FPdim(V)=\FPdim(\mathbf{V}),$$
where the last equality holds since the forgetful functor $\Z(\C)\rightarrow\C$ preserves Frobenius-Perron dimensions as well.
\end{itemize}
\end{proof}

\begin{proposition}\label{prop:intindscoincide}
Let $\C$ be an integral fusion category over $\mathbb{C}$ with canonical pivotal structure $j$. Suppose $\M$ is a finite indecomposable left $\C$-module category which is semisimple. Then the indicators of any $\mathbf{F}\in\C_\M^\ast$ in the senses of Lemma \ref{lem:indsorigin} and Definition \ref{def:dual-inds} coincide.
\end{proposition}

\begin{proof}
Since $\C$ is an integral fusion category, we conclude by \cite[Theorem 8.35 and Proposition 8.24]{ENO05} that its dual fusion category $\C_\M^\ast$ will be also integral and hence pseudo-unitary with canonical structure $j'$ (which is in fact spherical).

According to Lemma \ref{lem:intfusioncatcenter}, the centers $\Z(\C)$ and $\Z(\C_\M^\ast)$ are both pseudo-unitary with canonical pivotal structures $j$ and $j'$ respectively. It follows from \cite[Corollary 6.2]{NS07(a)} that Schauenburg's equivalence $\Omega:\Z(\C)\approx\Z(\C_\M^\ast)$ will preserve their canonical pivotal structures. Thus the claim desired holds according to Proposition \ref{prop:indscoincide}.
\end{proof}

Now we consider the category $\Rep(K)$ of finite-dimensional representations of a finite-dimensional semisimple quasi-Hopf algebra $K$ over $\mathbb{C}$. It is known in \cite[Theorem 8.33]{ENO05} that $\Rep(K)$ is an integral fusion category (with Frobenius-Perron dimensions of objects being their dimensions as vector spaces), and hence $\Rep(K)$ admits the canonical pivotal structure. Details could be found in \cite{NS08}, where the notion of higher Frobenius-Schur indicators for a semisimple quasi-Hopf algebra over $\mathbb{C}$ is then defined as a consequence:

%. Specifically:
%
%With the language of the canonical pivotal structure,
%the notion of higher Frobenius-Schur indicators for a semisimple quasi-Hopf algebra over $\mathbb{C}$ is defined in \cite{NS08}. Specifically:

\begin{definition}\label{def:inds-quasiHopfalg}
Let $K$ be a finite-dimensional semisimple quasi-Hopf algebra over $\mathbb{C}$.
Regard $\Rep(K)$ as the integral fusion category with the canonical pivotal structure.
\begin{itemize}
\item[(1)] (\cite[Definition 3.1]{NS08})
For any finite-dimensional left $K$-module $V$, its $n$-th Frobenius-Schur indicator is defined to be $\nu_n(V)$ in the sense of Equation (\ref{eqn:indsdef}), where $V$ is regarded as an object in $\Rep(K)$;
\item[(2)] (\cite[Proposition 5.3]{NS07(b)})
The Frobenius-Schur exponent $\FSexp(K)$ of $K$ is defined in the sense of (\ref{eqn:objFSexp}) for the regular module $K\in\Rep(K)$, which coincides with $\FSexp(\Rep(K))$ in the sense of (\ref{eqn:FSexp}) as well.
\end{itemize}
\end{definition}

\begin{remark}\label{rmk:inds-Hopfalgcoincide}
Suppose the semisimple complex quasi-Hopf algebra $K$ is in fact a Hopf algebra. Then
%\begin{itemize}
%\item[(1)]
it is shown in
\cite[Remark 3.4]{NS08} that the $n$-th indicator of $V\in\Rep(K)$ defined above coincides with \cite[Definition 2.3]{KSZ06} for semisimple Hopf algebras:
\begin{equation}\label{eqn:inds-Hopfalg}
\nu_n(V)=\langle\chi,\Lambda^{[n]}\rangle\;\;\;\;\;\;\;\;(\forall n\geq1),
\end{equation}
where $\chi_V$ is the character of the $K$-module $V$, and $\Lambda$ is the normalized integral of $K$ with
$n$-th Sweedler power $\Lambda^{[n]}=\sum\Lambda_{(1)}\Lambda_{(2)}\cdots\Lambda_{(n)}$.
%\item[(2)]
%\textcolor{red}{Since $\exp(K)$ defined in \cite{EG99} is known to be the period of the sequence $\{\nu_n(K)\}_{n\geq1}$ of indicators for the regular representation, one might find that $\exp(K)=\FSexp(K)$?}
%\end{itemize}
\end{remark}

At final of this subsection, we list some basic formulas on indicators of semisimple Hopf algebras for later use.
In what follows, the indicator $\nu_n(V)$ (\ref{eqn:inds-Hopfalg}) would be also denoted by $\nu_n^K(V)$ in order to distinguish the Hopf algebra $K$.

\begin{lemma}\label{lem:indicatorformula}
Let $H$ and $K$ be semisimple Hopf algebras. Then:
\begin{itemize}
\item[(1)]
For any $V\in\Rep(H)$ and $W\in\Rep(K)$, the indicators of $V\otimes W\in\Rep(H\otimes K)$ satisfy \begin{equation}\label{eqn:indicatorformula1}
\nu_n(V\otimes W)=\nu_n^H(V)\nu_n^K(W)\;\;\;\;\;\;\;\;(\forall n\geq1);
\end{equation}

\item[(2)]
For any $V_1,V_2\in\Rep(H)$,
\begin{equation}\label{eqn:indicatorformula2}
\nu_n(V_1\oplus V_2)=\nu_n(V_1)+\nu_n(V_2)\;\;\;\;\;\;\;\;(\forall n\geq1);
\end{equation}

\item[(3)]
For any $W\in\Rep(K)$,
\begin{equation}\label{eqn:indicatorformula3}
\nu_n^{K^\op}(W^\ast)=\nu_n^{K^\cop}(W)\;\;\;\;\;\;\;\;(\forall n\geq1),
\end{equation}
where $(W^\ast,\cdot)\in\Rep(K^\op)$ is defined by the contravariant Hom functor $\Hom_\mathbb{C}(-,\mathbb{C})$, meaning that
$k\cdot w^\ast:=\langle w^\ast,k(-)\rangle$ for any $k\in K^\op$ and $w^\ast\in W^\ast$.
\end{itemize}
\end{lemma}

\begin{proof}
(1) appears in the proof of \cite[Proposition 5.11]{NS07(b)}, and (2) is a consequence of \cite[Corollary 7.8]{NS07(a)}.

In order to show (3), denote by $\Lambda$ the normalized integral of $K$.
Suppose $\{w_i\}$ is a linear basis of $W$ with dual basis $\{w^\ast_i\}$ of $W^\ast$. Then we could compute for any $n\geq1$ that
\begin{eqnarray*}
\nu_n^{K^\op}(W^\ast)
&=& \sum_i\big\langle \Lambda_{(n)}\Lambda_{(n-1)}\cdots \Lambda_{(1)}\cdot w^\ast_i,w_i\big\rangle  \\
&=& \sum_i\big\langle w^\ast_i,\Lambda_{(n)}\Lambda_{(n-1)}\cdots \Lambda_{(1)}w_i\big\rangle
~=~ \nu_n^{K^\cop}(W).
\end{eqnarray*}
\end{proof}

\subsection{Left partially dualized quasi-Hopf algebras, and some situations arising from Hopf algebra extensions}\label{subsection:5.2}

We work over a general field $\k$ in this subsection.
Let $H$ be a finite-dimensional Hopf algebra over $\k$, and let $\Rep(H)$ be the category of finite-dimensional representations of $H$. A well-known result \cite[Proposition 1.19]{AM07} states that each indecomposable exact $\Rep(H)$-module category has form $\Rep(B)$ for some indecomposable exact left $H$-comodule algebra $B$. When $B$ is in particular a left coideal subalgebra of $H$, it is reconstructed in \cite{Li23} a quasi-Hopf algebra from the dual tensor category $\Rep(H)_{\Rep(B)}^\ast$.

In this subsection, we recall some of the relevant notions and results. The first one is \cite[Definition 2.6]{Li23}:
%where only the case $C=H/B^+H$ is necessary in this paper (see \cite[Remark 2.7]{Li23} for details):

\begin{definition}\label{def:PAMS}
Let $H$ be a finite-dimensional Hopf algebra. Suppose that

\begin{itemize}
\item[(1)]
$\iota:B\rightarrowtail H$ is an injection of left $H$-comodule algebras, and
$\pi:H\twoheadrightarrow C$ is a surjection of right $H$-module coalgebras;
\item[(2)]
The image of $\iota$ equals the space of the coinvariants of the right $C$-comodule $H$ with structure $(\id_H\otimes\pi)\circ\Delta$.
\end{itemize}
Then the pair of $\k$-linear diagrams
\begin{equation}\label{eqn:admissiblemapsys}
\begin{array}{ccc}
\xymatrix{
B \ar@<.5ex>[r]^{\iota} & H \ar@<.5ex>@{-->}[l]^{\zeta} \ar@<.5ex>[r]^{\pi}
& C \ar@<.5ex>@{-->}[l]^{\gamma}  }
&\;\;\text{and}\;\;&
\xymatrix{
C^\ast \ar@<.5ex>[r]^{\pi^\ast}
& H^\ast \ar@<.5ex>@{-->}[l]^{\gamma^\ast} \ar@<.5ex>[r]^{\iota^\ast}
& B^\ast \ar@<.5ex>@{-->}[l]^{\zeta^\ast}  },
\end{array}
\end{equation}
is said to be a {\pams} for $\iota$, denoted by $(\zeta,\gamma^\ast)$ for simplicity, if all the conditions
\begin{itemize}
\item[(3)]
$\zeta$ and $\gamma$ have convolution inverses $\overline{\zeta}$ and $\overline{\gamma}$ respectively;
\item[(4)]
$\zeta$ preserves left $B$-actions, and $\gamma$ preserves right $C$-coactions;
\item[(5)]
$\zeta$ and $\gamma$ preserve both the units and counits, meaning that
$$\zeta(1_H)=1_B,\;\;\;\;\e_H\mid_B\circ\zeta=\e_H,\;\;\;\;
\gamma[\pi(1_H)]=1_H\;\;\;\;\text{and}\;\;\;\;\e_H\circ\gamma=\e_C;$$
%(where the counit of $B$ and unit of $H/B^+H$ are induced by those of $H$ via $\iota$ and $\pi$ respectively);
\item[(6)]
$(\iota\circ\zeta)\ast(\gamma\circ\pi)=\id_H$,
\end{itemize}
and the dual forms of (3) to (6) hold equivalently.
\end{definition}

\begin{remark}
Under the assumptions in (1), it is clear that $B$ could be identified with a left coideal subalgebra of $H$, and (2) is equivalent to say that $C\cong H/B^+H$ holds
as mentioned in \cite[Remark 2.7]{Li23}.
Moreover, one could find additional formulas such as
\begin{equation}\label{eqn:piiotazetagammatrivial}
\pi[\iota(b)]=\langle\e_H,\iota(b)\rangle\pi(1_H)\;\;\;\;\text{and}\;\;\;\;
\zeta[\gamma(x)]=\langle\e_C,x\rangle 1_B
\end{equation}
and
\begin{equation}\label{eqn:zetaiotapigammaidentity}
\zeta[\iota(b)]=b\;\;\;\;\text{and}\;\;\;\;
\pi[\gamma(x)]=x
\end{equation}
for all $b\in B$ and $x\in C$ according to \cite[Proposition 2.9]{Li23}.
\end{remark}

For a pair of maps $\iota:B\hookrightarrow H$ and $\pi:H\twoheadrightarrow C$ satisfying (1) and (2) in Definition \ref{def:PAMS},
its {\pams}s always exist but are not unique. However, each system $(\zeta,\gamma^\ast)$ would determine a \textit{left partially dualized quasi-Hopf algebra} denoted by $C^\ast\#B$. We collect necessary descriptions of it as the definition below, which is essentially the same as a combination of some contents in \cite[Theorem 3.1, Definition 3.3 and Remark 3.4]{Li23}.

For the purpose, we could regard $B$ as a left coideal subalgebra of $H$, and regard
$C^\ast$ as a right coideal subalgebra of $H^\ast$ via $\pi^\ast$. Consequently, the following notations would be used for $b\in B$ and $f\in C^\ast$ that
\begin{equation}\label{eqn:b(1)b(2)f(1)f(2)}
\sum b_{(1)}\otimes b_{(2)}\in H\otimes B
\;\;\;\;\;\;\text{and}\;\;\;\;\;\;\sum f_{(1)}\otimes f_{(2)}\in C^\ast\otimes H^\ast
\end{equation}
to represent the structures of the left and right coideals (or comodules) respectively. Then we could also write equations
\begin{equation}\label{eqn:iotapi*}
\sum\iota(b)_{(1)}\otimes\iota(b)_{(2)}=\sum b_{(1)}\otimes \iota(b_{(2)})
\;\;\text{and}\;\;
\sum\pi^\ast(f)_{(1)}\otimes\pi^\ast(f)_{(2)}=\sum \pi^\ast(f_{(1)})\otimes f_{(2)}.
\end{equation}

\begin{definition}
Let $H$ be a finite-dimensional Hopf algebra with a {\pams} (\ref{eqn:admissiblemapsys}).
The left partially dualized quasi-Hopf algebra (or left partial dual) $C^\ast\#B$ determined by a $(\zeta,\gamma^\ast)$ is defined with the following structures: %For a linear basis $\{b_i\}$ of $B$ with dual basis $\{b_i^\ast\}$ of $B^\ast$,
\begin{itemize}
\item[(1)]
As an algebra, $C^\ast\#B$ is the smash product algebra with underlying vector space $C^\ast\otimes B$: The multiplication is given by
\begin{equation}\label{eqn:smashprod}
(f\#b)(g\#c)
:=\sum f(b_{(1)}\rightharpoonup g)\#b_{(2)}c
\;\;\;\;\;\;\;\;(\forall f,g\in C^\ast,\;\;\forall b,c\in B),
\end{equation}
and the unit element is $\e\#1$;

\item[(2)]
The ``comultiplication''
$\pd{\Delta}$ and ``counit'' are described in
\cite[Remark 3.4(2) and Definition 3.1(2)]{Li23};

\item[(3)]
The associator $\pd{\phi}$ is the inverse of the element
\begin{equation}\label{eqn:phi^-1}
\pd{\phi}^{-1}=\sum_{i,j}\left(\e\#\zeta[\gamma(x_i)\gamma(x_j)_{(1)}]\right)
\otimes\left(x_i^\ast\#\zeta[\gamma(x_j)_{(2)}]\right)
\otimes(x_j^\ast\#1),
\end{equation}
where $\{x_i\}$ is a linear basis of $C$ with dual basis $\{x_i^\ast\}$ of $C^\ast$;

\item[(4)]
The antipodes are described in \cite[Definition 3.1(4)]{Li23}.
\end{itemize}
If the associator $\pd{\phi}$ is trivial, then the left partial dual $C^\ast\#B$ is called a left partial dualized Hopf algebra.
\end{definition}

\begin{remark}\label{rmk:partialdualdim}
It is known by \cite[Theorem 2.1(6)]{Mas92} that $\dim(C^\ast\#B)=\dim(B)\dim(C)=\dim(H)$ holds.
\end{remark}

The \textit{reconstruction theorem for left partial duals} should be remarked as \cite[Theorem 4.22]{Li23}. Specifically, it reconstructs a left partial dualized quasi-Hopf algebra $C^\ast\#B$ from the dual tensor category $\Rep(H)_{\Rep(B)}^\ast$, which implies that the finite tensor categories $\Rep(C^\ast\#B)$ and $\Rep(H)$ are \textit{categorically Morita equivalent}.
On the other hand, since $\Rep(B)$ would be an indecomposable and exact $\Rep(H)$-module category, one could know by Lemma \ref{lem:dualcatfusion} that $\Rep(H)_{\Rep(B)}^\ast$ would be fusion as long as $\Rep(H)$ is fusion in characteristic $0$. In other words:

\begin{corollary}\label{cor:leftpdsemisimple}
Let $H$ be a semisimple Hopf algebra over an algebraically closed field $\k$ of characteristic $0$.
Then every left partial dualized quasi-Hopf algebra of $H$ is semisimple.
\end{corollary}

It is known that the quasi-Hopf algebra $C^\ast\#B$ would become a Hopf algebra when its
associator $\pd{\phi}$ (or its inverse $\pd{\phi}^{-1}$) is trivial. In this case, we also say that $C^\ast\#B$ is a \textit{left partially dualized Hopf algebra}.
Some examples are collected in Subsection \ref{subsection:5.4} in order to study their indicators, including \textit{bismash products of matched pair of groups} (\cite{Tak81}) and \textit{(generalized) quantum doubles} (\cite{DT94}).

Finally, let us focus on left partial duals determined by {\pams}s arising from extensions of Hopf algebras (\cite{Mas94,AD95}).
Recall in \cite[Definitions 1.3 and 2.4]{Mas94} that if a diagram
\begin{equation}\label{eqn:Hopfalgext}
\xymatrix{B \ar[r]^{\iota} & H \ar[r]^{\pi} & C }
\end{equation}
of Hopf algebras satisfies that
\begin{itemize}
\item[(1)]
$\iota$ is an injection, and $\pi$ is a projection;
\item[(2)]
$\mathrm{Im}(\iota)=\{h\in H\mid \sum h_{(1)}\otimes\pi(h_{(2)})=h\otimes\pi(1_H)\}$
(or other equivalent conditions stated in \cite[Lemma 1.2]{Mas94}) holds,\
\end{itemize}
then (\ref{eqn:Hopfalgext}) is called an \textit{extension of Hopf algebras}.
In particular, the extension called \textit{abelian} if $B$ is commutative and $C$ is cocommutative (cf. \cite{Mas02}).

It is clear that for an extension (\ref{eqn:Hopfalgext}) of Hopf algebras, there always exists {\pams} of form
\begin{equation*}
\begin{array}{ccc}
\xymatrix{
B \ar@<.5ex>[r]^{\iota} & H \ar@<.5ex>@{-->}[l]^{\zeta} \ar@<.5ex>[r]^{\pi}
& C \ar@<.5ex>@{-->}[l]^{\gamma}  }
&\;\;\text{and}\;\;&
\xymatrix{
C^\ast \ar@<.5ex>[r]^{\pi^\ast}
& H^\ast \ar@<.5ex>@{-->}[l]^{\gamma^\ast} \ar@<.5ex>[r]^{\iota^\ast}
& B^\ast \ar@<.5ex>@{-->}[l]^{\zeta^\ast}  },
\end{array}
\end{equation*}
for $\iota$.
In this situation, the algebra structure of the left partial dual determined by $(\zeta,\gamma^\ast)$ is shown to be special:

\begin{lemma}\label{lem:abelextleftPD}
Let $H$ be a finite-dimensional Hopf algebra fitting into an extension $B\xrightarrow{\iota}H\xrightarrow{\pi}C$
of Hopf algebras.
Then every left partial dualized quasi-Hopf algebra
$$C^\ast\#B=C^\ast\otimes B$$
%$(H/B^+H)^\ast\#B$ for $B\subseteq H$ is the tensor product $(H/B^+H)^\ast\otimes B$
as an algebra. In particular, if the extension is abelian, then $C^\ast\#B$ is commutative.
\end{lemma}

\begin{proof}
Note here that $\iota:B\rightarrowtail H$ and $\pi:H\twoheadrightarrow C$ are Hopf algebra maps, and hence the right $H^\ast$-comodule structure of $C^\ast$ and the left $H$-comodule structure of $B$ are respectively:
$$\begin{array}{ccc}
\begin{array}{ccc}
C^\ast &\rightarrow& C^\ast\otimes H^\ast  \\
f &\mapsto& \sum f_{(1)}\otimes \pi^\ast(f_{(2)})
\end{array}
&\;\;\;\;\text{and}\;\;\;\;&
\begin{array}{ccc}
B &\rightarrow& H\otimes B  \\
b &\mapsto& \sum \iota(b_{(1)})\otimes b_{(2)},
\end{array}
\end{array}$$
where the Sweedler notations stand for the comultiplications of respective Hopf algebras $C^\ast$ and $B$.
Then we calculate in the left partially dualized quasi-Hopf algebra $C^\ast\#B$ with the formula (\ref{eqn:smashprod}) that: For all $f,g\in C^\ast$ and $b,c\in B$,
\begin{eqnarray*}
(f\#b)(g\#c)
&=&
\sum f(\iota(b_{(1)})\rightharpoonup g)\# b_{(2)}c
~=~
\sum fg_{(1)}\langle\pi^\ast(g_{(2)}),\iota(b_{(1)})\rangle\# b_{(2)}c  \\
&=&
\sum fg_{(1)}\langle g_{(2)},\pi[\iota(b_{(1)})]\rangle\# b_{(2)}c
\overset{(\ref{eqn:piiotazetagammatrivial})}{=}
\sum fg_{(1)}\langle g_{(2)},\e(b_{(1)})1\rangle\# b_{(2)}c  \\
&=& fg\#bc.
\end{eqnarray*}

Furthermore, if $B$ is commutative and $C$ is cocommutative, then the tensor product $C^\ast\otimes B$ of algebras becomes commutative, since $C^\ast$ is commutative as well.
\end{proof}

\begin{remark}
Suppose $H$ is a semisimple Hopf algebra over $\mathbb{C}$ fitting into an abelian extension of Hopf algebras. Then we could infer by Lemma \ref{lem:abelextleftPD} that the fusion category $\Rep(H)$ is categorically Morita equivalent to a pointed one, and hence $\Rep(H)$ is \textit{group-theoretical} in the sense of \cite[Definition 8.40]{ENO05}. This fact was shown by Natale in \cite[Theorem 1.3]{Nat03}.
\end{remark}

Another particular case for extension of Hopf algebras is the split condition:
\begin{definition}(\cite[Definition 6.5.2]{Sch02})\label{def:splitHopfalgext}
An extension
$B\xrightarrow{\iota} H\xrightarrow{\pi} C$
of Hopf algebras is said to be split, if there exist a left $B$-module coalgebra map $\zeta:H\rightarrow B$ and a right $C$-comodule algebra map $\gamma:C\rightarrow H$, such that $(\iota\circ\zeta)\ast(\gamma\circ\pi)=\id_H$ holds.
\end{definition}

Under the conditions in Definition \ref{def:splitHopfalgext},
it is clear that $\zeta$ preserves the counits, and $\gamma$ preserves the units. Moreover, due to the arguments in the proof of \cite[Theorem 9]{DT86} and
\cite[Lemma 2.15]{Mas92}, we could assume that $\zeta$ and $\gamma$ both preserve the units and the counits since $\zeta(1)$ must be a group-like element in this case. Details are also found in the paragraphs before \cite[Lemma 2.5]{Li23}.

Consequently, a split extension (\ref{eqn:Hopfalgext}) of finite-dimensional Hopf algebras implies a {\pams} $(\zeta,\gamma^\ast)$ of $\iota$, such that $\zeta$ is a left $B$-module coalgebra map and $\gamma$ is a right $C$-comodule algebra map. Howerver in this situation, partially dualized Hopf algebras should be determined according to the following lemma:

\begin{lemma}\label{lem:leftPDsplitext}
Let $H$ be a finite-dimensional Hopf algebra fitting into a split extension $B\xrightarrow{\iota}H\xrightarrow{\pi}C$
of Hopf algebras.
Suppose that $(\zeta,\gamma^\ast)$ is a {\pams} for $\iota$, where $\zeta$ is a left $B$-module coalgebra map and $\gamma$ is a right $C$-comodule algebra map.
Then
$(\pi,\iota^\ast)$ is a {\pams}
\begin{equation}\label{eqn:admissiblemapsyssplit}
\begin{array}{ccc}
\xymatrix{
C^\biop \ar@<.5ex>[r]^{\gamma} & H^\biop \ar@<.5ex>[l]^{\pi} \ar@<.5ex>[r]^{\zeta}
& B^\biop \ar@<.5ex>[l]^{\iota}  }
&\text{and}&
\xymatrix{
B^{\ast\biop} \ar@<.5ex>[r]^{\zeta^\ast}
& H^{\ast\biop} \ar@<.5ex>[l]^{\iota^\ast} \ar@<.5ex>[r]^{\gamma^\ast}
& C^{\ast\biop} \ar@<.5ex>[l]^{\pi^\ast}  },
\end{array}
\end{equation}
for $\gamma:C^\biop\rightarrowtail H^\biop$, and it
determines a left partially dualized Hopf algebra $B^{\ast\biop}\#C^\biop$ of $H^\biop$, such that
$$B^{\ast\biop}\#C^\biop=B^{\ast\biop}\otimes C^\biop$$
as a coalgebra.
\end{lemma}

\begin{proof}
Firstly, we aim to show that $\mathrm{Im}(\gamma)$ is a right coideal of $H$, which implies that $\gamma:C\rightarrow H$ is a right $H$-comodule injection. In fact, since $\gamma$ is a right $C$-comodule map and $\zeta$ is a coalgebra map, we could calculate for any $x\in C$ that:
\begin{eqnarray*}
\sum\zeta[\gamma(x)_{(1)}]\otimes\gamma(x)_{(2)}
&\overset{\text{Definition}\;\ref{def:PAMS}(6)}=&
\sum\zeta[\gamma(x)_{(1)}]\otimes
  \iota\left(\zeta[\gamma(x)_{(2)}]\right)\gamma\left(\pi[\gamma(x)_{(3)}]\right)  \\
&=&
\sum\zeta[\gamma(x_{(1)})_{(1)}]\otimes
  \iota\left(\zeta[\gamma(x_{(1)})_{(2)}]\right)\gamma(x_{(2)})  \\
&=&
\sum\zeta[\gamma(x_{(1)})]_{(1)}\otimes
  \iota\left(\zeta[\gamma(x_{(1)})]_{(2)}\right)\gamma(x_{(2)})  \\
&\overset{(\ref{eqn:piiotazetagammatrivial})}=&
1_B\otimes \gamma(x),
\end{eqnarray*}
and hence
\begin{eqnarray*}
\sum\gamma(x)_{(1)}\otimes\gamma(x)_{(2)}
&\overset{\text{Definition}\;\ref{def:PAMS}(6)}=&
\sum\iota\left(\zeta[\gamma(x)_{(1)}]\right)\gamma\left(\pi[\gamma(x)_{(2)}]\right)
  \otimes\gamma(x)_{(3)}  \\
&=&
\sum\iota(1_B)\gamma\left(\pi[\gamma(x)_{(1)}]\right)\otimes\gamma(x)_{(2)}  \\
&=&
\sum\gamma\left(\pi[\gamma(x)_{(1)}]\right)\otimes\gamma(x)_{(2)}
\;\;\in \mathrm{Im}(\gamma)\otimes H.
\end{eqnarray*}

Dually, we could also find that $\mathrm{Ker}(\zeta)$ is a left ideal of $H$ in an analogous way. Indeed,
since $\pi$ and $\gamma$ are algebra maps
which imply that $\overline{\gamma}=\gamma\circ S_C$ holds,
we could calculate for any $h\in H$ and $k\in\mathrm{Ker}(\zeta)$ that
\begin{eqnarray*}
\iota[\zeta(hk)]
&\overset{\text{Definition}\;\ref{def:PAMS}(6)}=&
\sum h_{(1)}k_{(1)}\overline{\gamma}[\pi(h_{(2)}k_{(2)})]
~=~
\sum h_{(1)}k_{(1)}\gamma\big(S_C[\pi(h_{(2)}k_{(2)})]\big)   \\
&=&
\sum h_{(1)}k_{(1)}\gamma\big(S_C[\pi(k_{(2)})]\big)\gamma\big(S_C[\pi(h_{(2)})]\big)  \\
&\overset{\text{Definition}\;\ref{def:PAMS}(6)}=&
\sum h_{(1)}\iota[\zeta(k)]\gamma\big(S_C[\pi(h_{(2)})]\big)
~=~ 0.
\end{eqnarray*}
and thus $\zeta(hk)=0$ as $\iota$ is an injection.
Therefore, it follows that $\zeta:H\rightarrow B$ is a left $H$-module surjection.

As a conclusion, $\zeta:C^\biop\rightarrow H^\biop$ is an injection of left $H^\biop$-comodule algebras, and $\gamma:H^\biop\rightarrow B^\biop$ is a surjection of right $H^\biop$-module coalgebras. Furthermore, it follows from Equation (\ref{eqn:piiotazetagammatrivial}) and $\dim(H)=\dim(B)\dim(C)$ by Remark \ref{rmk:partialdualdim} that
$$\mathrm{Im}(\gamma)=\{h\in H\mid \sum \pi(h_{(1)})\otimes h_{(2)}=\pi(1)\otimes h\}.$$

Then in order to explain that $\pi:H^\biop\rightarrow C^\biop$ preserves left $C^\biop$-actions, it suffices to note by Equation (\ref{eqn:zetaiotapigammaidentity}) that
$$\pi(h)x=\pi(h)\pi[\gamma(x)]=\pi[h\gamma(x)]\;\;\;\;\;\;\;\;(\forall x\in C,\;\forall h\in H)$$
holds in $C$.
It follows from similar calculations that $\iota:B^\biop\rightarrow H^\biop$ would preserve right $B^\biop$-coactions. Besides,
the requirement in Definition \ref{def:PAMS}(6) is the same as the equation $(\iota\circ\zeta)\ast(\gamma\circ\pi)=\id_H$.
Therefore,
(\ref{eqn:admissiblemapsyssplit}) is a {\pams} for $\gamma:C^\biop\rightarrowtail H^\biop$ as desired.

At final, we might apply \cite[Lemma 2.12]{HKL25} to the {\pams} $(\pi,\iota^\ast)$ to know that $B^{\ast\biop}\#C^\biop$ of $H^\biop$ is a partially dualized Hopf algebra, whose coalgebra structure is
the tensor product $B^{\ast\biop}\otimes C^\biop$.
\end{proof}

\begin{corollary}(cf. \cite{Sch02})\label{cor:splitabelext}
Suppose $H$ is a finite-dimensional Hopf algebra fitting into a split abelian extension of Hopf algebras.
Then $\Rep(H)$ is categorically Morita equivalent to $\Rep(K)$ for some finite-dimensional cocommutative Hopf algebra $K$.
\end{corollary}

\begin{proof}
Suppose $B\rightarrow H\rightarrow C$ is a split abelian extension. Then we know by Lemma \ref{lem:leftPDsplitext} that $H^\biop$ has a left partially dualized Hopf algebra denoted by $K=B^{\ast\,\biop}\#C^\biop$, whose coalgebra structure is the tensor product $B^{\ast\,\biop}\otimes C^\biop$.

However, the Hopf algebra $H^\biop$ is isomorphic to $H$ via the antipode.
Consequently, the finite tensor categories $\Rep(H)$ and $\Rep(H^\biop)$ are categorically Morita equivalent to $\Rep(K)$.
On the other hand, since the extension is abelian, the Hopf algebras $B^{\ast\,\biop}$ and $C^\biop$ are both cocommutative, and hence $K$ is a cocommutative Hopf algebra.
\end{proof}

\subsection{Relations between the indicators of a semisimple Hopf algebra and its left partial duals}\label{subsection:mainthmPD}

In order to study left partial duals by using formulas in \cite[Section 4]{Li23} frequently, let us consider the case when $B\subseteq H$ is a left coideal subalgebra and $C=H/B^+H$ in the {\pams} (\ref{eqn:admissiblemapsys}), and the notations in (\ref{eqn:b(1)b(2)f(1)f(2)}) become
\begin{equation}\label{eqn:b(1)b(2)f(1)f(2)new}
\sum b_{(1)}\otimes b_{(2)}\in H\otimes B
\;\;\;\;\;\;\text{and}\;\;\;\;\;\;\sum f_{(1)}\otimes f_{(2)}\in (H/B^+H)^\ast\otimes H^\ast
\end{equation}
for all $b\in B$ and $f\in(H/B^+H)^\ast$.

This subsection is devoted to applying Theorem \ref{thm:bimod-inds} to the reconstruction $\Rep(H)_{\Rep(B)}^\ast\approx\Rep((H/B^+H)^\ast\#B)$ mentioned in the previous subsection, where $H$ is a semisimple Hopf algebra over $\mathbb{C}$ with left coideal subalgebra $B$. Our result would be an equation on indicators of two related modules over $(H/B^+H)^\ast\#B$ and $H$ respectively. At first we introduce the structures desired in the following, some of which are known in \cite{Mas94} for example:

Suppose $L$ is a finite-dimensional left $B$-module, or $L\in\Rep(B)$ for simplicity. We denote by
$(H/B^+H)^\ast\#L$ the left $(H/B^+H)^\ast\#B$-module whose underlying vector space is $(H/B^+H)^\ast\otimes L$, and its module structure is defined as
\begin{eqnarray}\label{eqn:smashLmodstru}
[(H/B^+H)^\ast\#B]\otimes[(H/B^+H)^\ast\#L] &\rightarrow& (H/B^+H)^\ast\#L  \\
(f\#b)\otimes(g\#l) &\mapsto& \sum f(b_{(1)}\rightharpoonup g)\#b_{(2)}l  \nonumber
\end{eqnarray}
with notations in Subsection \ref{subsection:5.2}. Furthermore, this induces a restricted action of the subalgebra $B\cong\e\#B$ on $(H/B^+H)^\ast\#L$, which induces then a right $B^\ast$-comodule structure
$$[(H/B^+H)^\ast\#L]\rightarrow [(H/B^+H)^\ast\#L]\otimes B^\ast,\;\;
f\#l\mapsto\sum (f\#l)_{\langle0\rangle}\otimes (f\#l)_{\langle1\rangle}$$
satisfies that the equation
$\sum (f\#l)_{\langle0\rangle}\langle (f\#l)_{\langle1\rangle},b\rangle=(\e\#b)(f\#l)\
=\sum(b_{(1)}\rightharpoonup f)\#b_{(2)}l$
holds for each $b\in B$.

On the other hand, $L\in\Rep(B)$ becomes similarly a right $B^\ast$-comodule, such that the coaction
\begin{equation}\label{eqn:Lrightcomodstru}
L\rightarrow L\otimes B^\ast,\;\;l\mapsto\sum l_{\langle0\rangle}\otimes l_{\langle1\rangle}
\end{equation}
satisfying that the equation
$\sum l_{\langle0\rangle}\langle l_{\langle1\rangle},b\rangle=bl$
holds for each $b\in B$. Meanwhile, if we regard $H^\ast$ as a left $B^\ast$-comodule induced by the coalgebra map $\iota^\ast:H^\ast\rightarrow B^\ast$, then the \textit{cotensor product} (e.g. \cite[Section 0]{Tak77}) of $L$ and $H^\ast$ could be defined as
\begin{equation}\label{eqn:cotensorprod0}
L\square_{B^\ast}H^\ast
:=\Big\{\sum_i l_i\otimes h^\ast_i\in L\otimes H^\ast\mid
\sum_i {l_i}_{\langle0\rangle}\otimes {l_i}_{\langle1\rangle}\otimes h^\ast_i
=\sum_i l_i\otimes \iota^\ast({h^\ast_i}_{(1)})\otimes{h^\ast_i}_{(2)}\Big\}.
\end{equation}
Then $L\square_{B^\ast}H^\ast$ is a right $H^\ast$-comodule via the comultiplication on the latter (co)tensorand $H^\ast$ (similarly to \cite[Equation (4.14)]{Li23}), which induces naturally a left $H$-action:
\begin{equation}\label{eqn:cotensorH*modstru}
H\otimes(L\square_{B^\ast}H^\ast)\rightarrow L\square_{B^\ast}H^\ast,\;\;
h\otimes\Big(\sum_i l_i\otimes h^\ast_i\Big)\mapsto \sum_i l_i\otimes (h\rightharpoonup h^\ast_i).
\end{equation}
Besides, $L\square_{B^\ast}H^\ast$ is a right $(H/B^+H)^\ast$-module (also as in \cite[Equation (4.14)]{Li23}) with structure
\begin{equation}\label{eqn:cotensorH*modstru2}
(L\square_{B^\ast}H^\ast)\otimes (H/B^+H)^\ast\rightarrow L\square_{B^\ast}H^\ast,\;\;
\Big(\sum_i l_i\otimes h^\ast_i\Big)\otimes f\mapsto \sum_i l_i\otimes h^\ast_i\pi^\ast(f^\ast).
\end{equation}

\begin{lemma}
With notations above, there is an isomorphism
\begin{eqnarray}\label{eqn:psi}
\psi~:~(H/B^+H)^\ast\otimes(L\square_{B^\ast}H^\ast)
&\cong& [(H/B^+H)^\ast\# L]\,\square_{B^\ast}H^\ast,  \\
f\otimes\Big(\sum_i l_i\otimes h^\ast_i\Big)
&\mapsto& \sum_i (f_{(1)}\#l_i)\otimes f_{(2)}h^\ast_i  \nonumber
\end{eqnarray}
in ${}_{(H/B^+H)^\ast}\Rep(H)_{(H/B^+H)^\ast}$, where:
\begin{itemize}
\item
$(H/B^+H)^\ast\otimes(L\square_{B^\ast}H^\ast)\in\Rep(H)$ via the diagonal $H$-actions, and it is similar to (\ref{eqn:cotensorH*modstru}) that
$[(H/B^+H)^\ast\# L]\,\square_{B^\ast}H^\ast\in\Rep(H)$ via the hit $H$-action on the last (co)tensorand $H^\ast$;
\item
The left and right $(H/B^+H)^\ast$-module structures on $(H/B^+H)^\ast\otimes(L\square_{B^\ast}H^\ast)$ are determined on respective tensorands. However as for $[(H/B^+H)^\ast\# L]\,\square_{B^\ast}H^\ast$, its right $(H/B^+H)^\ast$-action is similar to (\ref{eqn:cotensorH*modstru2}), while the left $(H/B^+H)^\ast$-action is diagonal at the cotensor product (\cite[Equation (4.13)]{Li23}).
\end{itemize}
\end{lemma}

\begin{proof}
Firstly, one might prove directly by the definition of the cotensor product (\ref{eqn:cotensorprod0}) that
\begin{equation}\label{eqn:cotensorprod}
\sum_i l_i\otimes h^\ast_i\in L\square_{B^\ast}H^\ast
\;\;\;\;\Longleftrightarrow\;\;\;\;
\forall b\in B,\;\;\sum_i bl_i\otimes h^\ast_i
  =\sum_i l_i\otimes\left(h^\ast_i\leftharpoonup\iota(b)\right)
\end{equation}
for $l_i\in L$ and $h^\ast_i\in H^\ast$.
Thus
in order to show that $\psi$ (\ref{eqn:psi}) is well-defined, suppose $f\in (H/B^+H)^\ast$ and $\sum_i l_i\otimes h^\ast_i\in L\square_{B^\ast}H^\ast$. We compute for any $b\in B$ that
\begin{eqnarray*}
\sum_i (\e\#b)(f_{(1)}\#l_i)\otimes f_{(2)}h^\ast_i
&\overset{(\ref{eqn:smashLmodstru})}{=}&
\sum_i \left[(b_{(1)}\rightharpoonup f_{(1)})\#b_{(2)}l_i\right]
  \otimes f_{(2)}h^\ast_i  \\
&\overset{(\ref{eqn:cotensorprod})}{=}&
\sum_i \left[(b_{(1)}\rightharpoonup f_{(1)})\#l_i\right]
  \otimes f_{(2)}(h^\ast_i\leftharpoonup \iota(b_{(2)}))  \\
&=&
\sum_i (f_{(1)}\langle f_{(2)},b_{(1)}\rangle\#l_i)
  \otimes f_{(3)}(h^\ast_i\leftharpoonup \iota(b_{(2)}))  \\
&=&
\sum_i (f_{(1)}\#l_i)
  \otimes (f_{(2)}\leftharpoonup b_{(1)})(h^\ast_i\leftharpoonup \iota(b_{(2)}))  \\
&\overset{(\ref{eqn:iotapi*})}{=}&
\sum_i (f_{(1)}\#l_i)\otimes
  (f_{(2)}\leftharpoonup \iota(b)_{(1)})(h^\ast_i\leftharpoonup \iota(b)_{(2)})  \\
&=&
\sum_i (f_{(1)}\#l_i)
  \otimes (f_{(2)}h^\ast_i\leftharpoonup \iota(b)),
\end{eqnarray*}
which means that $\sum_i (f_{(1)}\#l_i)\otimes f_{(2)}h^\ast_i\in L\square_{B^\ast}H^\ast$ according to (\ref{eqn:cotensorprod}).

Now let us verify that $\psi$ is a morphism in ${}_{(H/B^+H)^\ast}\Rep(H)_{(H/B^+H)^\ast}$:
Suppose $f\in (H/B^+H)^\ast$ and $\sum_i l_i\otimes h^\ast_i\in L\square_{B^\ast}H^\ast$.
\begin{itemize}
\item
($\psi$ preserves left $H$-actions)
For any $h\in H$,
\begin{eqnarray*}
\psi\left(h\cdot \Big[f\otimes\Big(\sum_i l_i\otimes h^\ast_i\Big)\Big]\right)
&=&
\psi\left(\sum (h_{(1)}\rightharpoonup f)
  \otimes\Big[\sum_i l_i\otimes (h_{(2)}\rightharpoonup h^\ast_i)\Big]\right)  \\
&=&
\sum\psi\left( f_{(1)}\otimes\Big(\sum_i l_i\otimes {h^\ast_i}_{(1)}\Big)
  \langle f_{(2)}{h^\ast_i}_{(2)},h\rangle\right)  \\
&=&
\sum_i (f_{(1)}\# l_i)\otimes f_{(2)}{h^\ast_i}_{(1)}
  \langle f_{(3)}{h^\ast_i}_{(2)},h\rangle  \\
&=&
\sum_i (f_{(1)}\# l_i)\otimes \left(h\rightharpoonup f_{(2)}{h^\ast_i}\right)  \\
&=&
h\cdot \psi\left(f\otimes\Big(\sum_i l_i\otimes h^\ast_i\Big)\right);
\end{eqnarray*}

\item
($\psi$ preserves left $(H/B^+H)^\ast$-actions)
This is easy to know by the definition of $\psi$;

\item
($\psi$ preserves right $(H/B^+H)^\ast$-actions)
This is direct to verify, as the right $(H/B^+H)^\ast$-module structures of both sides of (\ref{eqn:psi}) are determined by the rightmost tensorand $H^\ast$ via the injection $\pi^\ast$ of algebras.
\end{itemize}

Finally, it is straightforward to verify that the linear map
\begin{eqnarray*}
[(H/B^+H)^\ast\# L]\,\square_{B^\ast}H^\ast
&\rightarrow& (H/B^+H)^\ast\otimes(L\otimes H^\ast),  \\
\sum_i (f_i\#l_i)\otimes h^\ast_i
&\mapsto& \sum_i {f_i}_{(1)}\otimes \Big[l_i\otimes S({f_i}_{(2)})h^\ast_i \Big]
\end{eqnarray*}
would be a left inverse of $\psi$, which implies that $\psi$ is a monomorphism in ${}_{(H/B^+H)^\ast}\Rep(H)_{(H/B^+H)^\ast}$.
However, one could find the following linear isomorphisms
\begin{equation}\label{eqn:cotensoriso0}
L\square_{B^\ast}H^\ast\cong L\square_{B^\ast}[B^\ast\otimes (H/B^+H)^\ast]
\cong (L\square_{B^\ast}B^\ast)\otimes (H/B^+H)^\ast\cong L\otimes (H/B^+H)^\ast
\end{equation}
for every right $B^\ast$-comodule $L$, where the the first one is mentioned in \cite[Equation (4.39)]{Li23}, and the second and third ones are mentioned in \cite[Section 0]{Tak77} as properties of cotensor products.
It
would imply that
\begin{eqnarray*}
\dim\left([(H/B^+H)^\ast\# L]\,\square_{B^\ast}H^\ast\right)
&\overset{(\ref{eqn:cotensoriso})}{=}&
\dim\left([(H/B^+H)^\ast\# L]\otimes(H/B^+H)^\ast\right)  \\
&=&
\dim\left((H/B^+H)^\ast\right)\dim\left(L\otimes (H/B^+H)^\ast\right)  \\
&\overset{(\ref{eqn:cotensoriso})}{=}&
\dim\left((H/B^+H)^\ast\right)\dim\left(L\square_{B^\ast}H^\ast\right)  \\
&=&
\dim\left((H/B^+H)^\ast\otimes(L\square_{B^\ast}B^\ast)\right).
\end{eqnarray*}
Thus $\psi$ is indeed an isomorphism as claimed.
\end{proof}

\begin{theorem}\label{thm:inds-leftpartialdual}
Let $H$ be a semisimple Hopf algebra over $\mathbb{C}$. Suppose $B$ is a left coideal subalgebra, and $(H/B^+H)^\ast\#B$ is a left partially dualized quasi-Hopf algebra of $H$.
Then for any $L\in\Rep(B)$,
\begin{equation}\label{eqn:inds(smL)=(Lsq)}
\nu_n((H/B^+H)^\ast\# L)=\nu_n(L\square_{B^\ast}H^\ast)
\;\;\;\;\;\;\;\;(\forall n\geq1),
\end{equation}
where $(H/B^+H)^\ast\#L\in\Rep((H/B^+H)^\ast\#B)$ and $L\square_{B^\ast}H^\ast\in\Rep(H)$.
\end{theorem}

\begin{proof}
Denote $\M=\Rep(H)_{(H/B^+H)^\ast}$, the category of right $(H/B^+H)^\ast$-modules in $\Rep(H)$. This is a finite indecomposable $\Rep(H)$-module category, which is exact and then semisimple. It is implied by \cite[Lemma 1.8]{Mas94} that $L\square_{B^\ast}H^\ast$ becomes an object in $\Rep(H)_{(H/B^+H)^\ast}$ with the left $H$-module structure
(\ref{eqn:cotensorH*modstru})
%$$H\otimes(L\square_{B^\ast}H^\ast)\rightarrow L\square_{B^\ast}H^\ast,\;\;
%h\otimes\Big(\sum_i l_i\otimes h^\ast_i\Big)
%\mapsto\sum_i l_i\otimes (h\rightharpoonup h^\ast_i)$$
and right $(H/B^+H)^\ast$-module structure
%$$(L\square_{B^\ast}H^\ast)\otimes(H/B^+H)^\ast\rightarrow L\square_{B^\ast}H^\ast,\;\;
%\Big(\sum_i l_i\otimes h^\ast_i\Big)\otimes f
%\mapsto\sum_i l_i\otimes h^\ast_i\pi^\ast(f).$$
(\ref{eqn:cotensorH*modstru2}).

It is mentioned in \cite[Corollary 4.23]{Li23} that the dual category $\Rep(H)_\M^\ast$ is tensor equivalent to $\Rep((H/B^+H)^\ast\#B)$.
Let us describe a tensor equivalence
\begin{equation}\label{eqn:Psi0}
\Psi:\Rep((H/B^+H)^\ast\#B)\approx\Rep(H)_\M^\ast
\end{equation}
in details, which is essentially the composition of \cite[(4.20) and (4.5)]{Li23}: Specifically, for each $V\in\Rep((H/B^+H)^\ast\#B)$, we have
$$\Psi(V)= -\otimes_{(H/B^+H)^\ast}(V\square_{B^\ast}H^\ast)\in \Rep(H)_\M^\ast.$$
Thus the tensor equivalence $\Psi$ will send the object $(H/B^+H)^\ast\#L$ to \begin{eqnarray}\label{eqn:Psi2}
\Psi((H/B^+H)^\ast\#L)
&=&
-\otimes_{(H/B^+H)^\ast}\left(\left[(H/B^+H)^\ast\#L\right]\,\square_{B^\ast}H^\ast\right)  \nonumber  \\
&\overset{(\ref{eqn:psi})}{\cong}&
-\otimes_{(H/B^+H)^\ast}[(H/B^+H)^\ast\otimes(L\square_{B^\ast}H^\ast)]  \nonumber \\
&\cong&
-\otimes(L\square_{B^\ast}H^\ast)\in \Rep(H)_\M^\ast.
\end{eqnarray}
It is straightforward that $L\square_{B^\ast}H^\ast$ in (\ref{eqn:Psi2}) is indeed an object in $\M=\Rep(H)_{(H/B^+H)^\ast}$ with structures (\ref{eqn:cotensorH*modstru}) and (\ref{eqn:cotensorH*modstru2}), and the $\Rep(H)$-module structure of the functor $-\otimes(L\square_{B^\ast}H^\ast)$ is the associativity constraint in $\Rep(H)$.

On the other hand, note according to \cite[Theorem 8.33]{ENO05} that $\Rep(H)$ (as well as $\Rep((H/B^+H)^\ast\#B)$) is an integral fusion category. By \cite[Theorem 8.35 and Proposition 8.24]{ENO05}, its dual fusion category $\Rep(H)_\M^\ast$ will be also
%integral and hence
pseudo-unitary.
It follows from \cite[Corollary 6.2]{NS07(a)} that $\Psi$ (\ref{eqn:Psi0}) will preserve their canonical pivotal structures.
Consequently, one could apply \cite[Corollary 4.4]{NS07(a)} to (\ref{eqn:Psi2}) and obtain
\begin{equation}\label{eqn:inds(smL)=(-otimesLsq)}
\nu_n((H/B^+H)^\ast\#L)=\nu_n(-\otimes(L\square_{B^\ast}H^\ast))
\;\;\;\;\;\;\;\;(\forall n\geq1),
\end{equation}
where both indicators are in the sense of Lemma \ref{lem:indsorigin} with respect to canonical pivotal structures.

However, we know by Proposition \ref{prop:intindscoincide} that the indicator $\nu_n(-\otimes(L\square_{B^\ast}H^\ast))$ in Equation (\ref{eqn:inds(smL)=(-otimesLsq)}) could be identified in the sense of Definition \ref{def:dual-inds}. Therefore due to Theorem \ref{thm:bimod-inds}, we have
\begin{equation}\label{eqn:inds(-otimesLsq)=(Lsq)}
\nu_n(-\otimes(L\square_{B^\ast}H^\ast))=\nu_n(L\square_{B^\ast}H^\ast)
\;\;\;\;\;\;\;\;(\forall n\geq1).
\end{equation}
At final, Equation (\ref{eqn:inds(smL)=(Lsq)}) is obtained as a combination of Equations (\ref{eqn:inds(smL)=(-otimesLsq)}) and (\ref{eqn:inds(-otimesLsq)=(Lsq)}).
\end{proof}

\begin{remark}
Evidently, Equation (\ref{eqn:inds(smL)=(Lsq)}) holds for {\pams}s of general forms: Let $H$ be a semisimple Hopf algebra over $\mathbb{C}$ with a {\pams} (\ref{eqn:admissiblemapsys}). Then for any $L\in\Rep(B)$,
\begin{equation}%\label{eqn:inds(smL)=(Lsq)2}
\nu_n(C^\ast\# L)=\nu_n(L\square_{B^\ast}H^\ast)
\;\;\;\;\;\;\;\;(\forall n\geq1),
\end{equation}
where $C^\ast\# L\in\Rep(C^\ast\#B)$ and $L\square_{B^\ast}H^\ast\in\Rep(H)$.
\end{remark}

We could point out two extreme situations for $L\in\Rep(B)$ being the regular module $B$ and the trivial module $\mathbb{C}1$:

\begin{corollary}\label{cor:inds-leftpartialdual2}
Let $H$ be a semisimple Hopf algebra over $\mathbb{C}$. Suppose that $B$ is a left coideal subalgebra of $H$.
Then:
\begin{itemize}
\item[(1)]
The indicators of regular representations in $\Rep((H/B^+H)^\ast\#B)$ and $\Rep(H)$ coincide:
$$\nu_n((H/B^+H)^\ast\#B)=\nu_n(H)
\;\;\;\;\;\;\;\;(\forall n\geq1);$$

\item[(2)]
The indicators of $(H/B^+H)^\ast$ regarded as objects in $\Rep((H/B^+H)^\ast\#B)$ and $\Rep(H)$ coincide.
Specifically, with the notations in (\ref{eqn:b(1)b(2)f(1)f(2)new}), the module structure on $(H/B^+H)^\ast\in\Rep((H/B^+H)^\ast\#B)$ is
\begin{eqnarray}\label{eqn:rightoidealsubalgmod1}
[(H/B^+H)^\ast\#B]\otimes(H/B^+H)^\ast &\rightarrow& (H/B^+H)^\ast,  \\
(f\#b)\otimes g &\mapsto& \sum fg_{(1)}\langle g_{(2)},\iota(b)\rangle,  \nonumber
\end{eqnarray}
while the module structure on $(H/B^+H)^\ast\in\Rep(H)$ is
\begin{equation}\label{eqn:rightoidealsubalgmod2}
H\otimes(H/B^+H)^\ast\rightarrow(H/B^+H)^\ast,\;\;
h\otimes g\mapsto \sum g_{(1)}\langle g_{(2)},h\rangle.
\end{equation}
\end{itemize}
\end{corollary}

\begin{proof}
\begin{itemize}
\item[(1)]
If we choose $L$ to be the regular left $B$-module in Theorem \ref{thm:inds-leftpartialdual}, then the equation
$$\nu_n((H/B^+H)^\ast\# B)=\nu_n(B\square_{B^\ast}H^\ast)
\;\;\;\;\;\;\;\;(\forall n\geq1)$$
is obtained, where the module $(H/B^+H)^\ast\#B$ with structure as in (\ref{eqn:smashLmodstru}) is exactly the regular representation in $\Rep((H/B^+H)^\ast\#B)$. Now let us focus on the left $H$-module $B\square_{B^\ast}H^\ast$:

Consider at first the left $B$-module structure $\rightharpoonup$ on the coalgebra $B^\ast$, and it follows from \cite[Lemma 2.8]{CW17} that there exists a left $B$-module isomorphism
$$\varphi:B\cong B^\ast,\;\;b\mapsto b\rightharpoonup\lambda$$
for some $\lambda\in B^\ast$.
However,
since the right $B^\ast$-coation on $B$ is defined as in (\ref{eqn:Lrightcomodstru}), it is isomorphic to the regular right $B^\ast$-comodule $B^\ast$ via
$$\varphi:b\mapsto b\rightharpoonup\lambda
=\sum\lambda_{(1)}\langle \lambda_{(2)},b\rangle
=\sum\langle\lambda,b_{\langle0\rangle}\rangle b_{\langle1\rangle}$$
defined above.
Thus according to one of the canonical isomorphisms in \cite[Section 0]{Tak77} on cotensor products, we might conclude a linear isomorphism
$$B\square_{B^\ast}H^\ast\cong B^\ast\square_{B^\ast}H^\ast\cong H^\ast,\;\;
\sum_i b_i\otimes h^\ast_i\mapsto\sum_i\langle\lambda,b_i\rangle h^\ast_i,$$
by which the left $H$-action defined in (\ref{eqn:cotensorH*modstru}) is mapped to the hit action $\rightharpoonup$ on $H$.
Furthermore, it is known in \cite[Theorem 2.1.3(3)]{Mon93} that $(H^\ast,\rightharpoonup)$ is also isomorphic to the regular left $H$-module. Therefore we find the equation
$$\nu_n(B\square_{B^\ast}H^\ast)=\nu_n(H)
\;\;\;\;\;\;\;\;(\forall n\geq1)$$
on the indicators of objects in $\Rep(H)$,
and hence the desired claim holds as a consequence.

\item[(2)]
If we choose $L$ to be the trivial left $B$-module $\mathbb{C}1$ (induced by the counit $\e$ restricted on $B$) in Theorem \ref{thm:inds-leftpartialdual}, then the equation
$$\nu_n((H/B^+H)^\ast\#1)=\nu_n(\mathbb{C}1\square_{B^\ast}H^\ast)
\;\;\;\;\;\;\;\;(\forall n\geq1)$$
is obtained. Here the left $(H/B^+H)^\ast\#B$-module $(H/B^+H)^\ast\#1$ could be directly verified isomorphic to $(H/B^+H)^\ast$ with structure (\ref{eqn:rightoidealsubalgmod1}). As for the left $H$-module $\mathbb{C}1\square_{B^\ast}H^\ast$, we could also use the isomorphism
(\ref{eqn:cotensoriso0}) to know that
\begin{eqnarray*}%\label{eqn:cotensoriso}
\mathbb{C}1\square_{B^\ast}H^\ast
&\cong& \mathbb{C}1\square_{B^\ast}[B^\ast\otimes (H/B^+H)^\ast]
~\cong~ (\mathbb{C}1\square_{B^\ast}B^\ast)\otimes (H/B^+H)^\ast  \\
&\cong& \mathbb{C}1\otimes (H/B^+H)^\ast  \\
&\cong& (H/B^+H)^\ast,  \\
1\otimes\pi^\ast(f) &\mapsfrom& ~~~~~~~~f
\end{eqnarray*}
holds.
According to the latter equation in (\ref{eqn:iotapi*}), this isomorphism preserves left $H$-actions if $(H/B^+H)^\ast$ is regarded to have the structure (\ref{eqn:rightoidealsubalgmod2}).
\end{itemize}
\end{proof}

If the left coideal subalgebra $B$ is furthermore chosen to be the trivial Hopf subalgebra $\mathbb{C}1$, then a semisimple case of \cite[Theorem 3.13(f)]{Shi12} or \cite[Proposition 3.4(1)]{Shi15(b)} could be obtained as a corollary:

\begin{corollary}\label{cor:Hopfdualind}
Let $H$ be a semisimple Hopf algebra over $\mathbb{C}$. Then
$$\nu_n(H^\ast)=\nu_n(H)\;\;\;\;\;\;\;\;(\forall n\geq1).$$
\end{corollary}

\begin{proof}
It is known in \cite[Corollary 6.1(2)]{Li23} that the dual Hopf algebra $H^\ast$ is also a left partial dual of $H$. Thus the desired claim is implied by Corollary \ref{cor:inds-leftpartialdual2}(1).
%Since the algebra $\intHom(\mathbb{C},\mathbb{C})=H^\ast$ by Examples in \cite[Section 3.3]{Ost03}, the left $\Rep(H)$-module category $\Vec$ could be identified with $\Rep(H)_{H^\ast}$ under the equivalence (\ref{eqn:bimodcat}), which is determined by
%$$\Vec\rightarrow \Rep(H)_{H^\ast},\;\;\;\;
%  \mathbb{C}\mapsto \intHom(\mathbb{C},\mathbb{C})=H^\ast.$$
%The left $H$-module structure on $H^\ast$ is given by the hit action $\rightharpoonup$, which makes $H^\ast$ a right $H^\ast$-module in $\Rep(H)$ via the multiplication of $H^\ast$.
%
%Consider the left $\Rep(H)$-module functor $-\otimes H^\ast:\Rep(H)_{H^\ast}\rightarrow\Rep(H)_{H^\ast}$. Recall that the equivalence (\ref{eqn:Hopfdualequiv}) sends $-\otimes H^\ast$ to the left $H^\ast$-module $\mathbb{C}\otimes H^\ast=H^\ast$. Thus we could describe by Theorem \ref{thm:bimod-inds} the indicators $\nu_n(H^\ast)$ of the regular left $H^\ast$-module, which would equal to $\nu_n(H)$, since the object $H^\ast\in\Rep(H)$ is isomorphic to the left regular $H$-module by \cite[Theorem 2.1.3(3)]{Mon93}.
\end{proof}

%\begin{question}
%What is the monoidal structure of the equivalence (\ref{eqn:dualequiv-Hopfalg}) when $B\#(H/B^+H)^\ast$ is also a Hopf algebra?
%\end{question}

\subsection{Some equations on indicators of bismash products and quantum doubles}\label{subsection:5.4}

%\textcolor{red}{----List props instead of a subsection.}
Let $(F,G,\triangleright,\triangleleft)$ be a matched pair of finite groups (\cite[Definition 2.1]{Tak81}), where $\triangleright:G\times F\rightarrow F$ and $\triangleleft:G\times F\rightarrow G$ are group actions satisfying
$$x\triangleright bc=(x\triangleright b)((x\triangleleft b)\triangleright c)
\;\;\;\;\;\;\text{and}\;\;\;\;\;\;
xy\triangleleft b=(x\triangleleft(y\triangleright b))(y\triangleleft b)$$
for all $b,c\in F$ and $x,y\in G$. We know that it determines two bismash products $\mathbb{C}^G\#\mathbb{C}F$ and $\mathbb{C}G\#\mathbb{C}^F$, which are semisimple complex Hopf algebras dual to each other. The detailed notions and constructions could be found in \cite[Section 1]{Mas02} and \cite[Sections 1 and 2]{BGM96}.

Firstly, note according to \cite[Proposition 6.8]{Li23} that $\mathbb{C}^G\#\mathbb{C}F$ is a left partial dualized Hopf algebra of the group algebra $\mathbb{C}(F\bowtie G)$. Thus
Corollary \ref{cor:inds-leftpartialdual2}(1) could be applied to induce that the indicators of their regular representations coincide:

\begin{corollary}(\cite[Corollary 5.7]{KMN12})
Let $(F,G)$ be a matched pair of finite groups. Then
$$\nu_n(\mathbb{C}^G\#\mathbb{C}F)=\nu_n(F\bowtie G)
\;\;\;\;\;\;\;\;(\forall n\geq1).$$
\end{corollary}

As for the other bismash product $\mathbb{C}G\#\mathbb{C}^F$, it should be a right partially dualized (coquasi-)Hopf algebra of $\mathbb{C}(F\bowtie G)$ in the sense of \cite[Definition 5.4]{Li23}.
However, with the help of Theorem \ref{thm:inds-leftpartialdual}, we could study its indicators more explicitly by proving that $\mathbb{C}G\#\mathbb{C}^F$ is a left partial dual of $\mathbb{C}^{F\bowtie G}$. Specifically:

\begin{lemma}\label{lem:bismashprodPD}
With notations above, suppose $\{p_a\}_{a\in F}$ is the basis of $\mathbb{C}^G$ which is dual to $G$, and suppose $\{p_{(a,x)}\}_{(a,x)\in F\times G}$ is the basis of $\mathbb{C}^{F\bowtie G}$ dual to $F\bowtie G$, and
define the maps
$$\begin{array}{rclcrclcrclcrcl}
\iota_F:F&\rightarrow&F\times G,&& \pi_F:F\times G&\rightarrow&F,  &&
\iota_G:G&\rightarrow&F\times G,&& \pi_G:F\times G&\rightarrow&G,  \\
a&\mapsto&(a,1),&&(a,x)&\mapsto&a,  &&
x&\mapsto&(1,x),&&(a,x)&\mapsto&x.
\end{array}$$
Then the bismash product $\mathbb{C}G\#\mathbb{C}^F$ is the left partially dualized Hopf algebra determined by the {\pams}
\begin{equation}\label{eqn:PAMSbismashprod}
\begin{array}{ccc}
\xymatrix{
\mathbb{C}^F \ar@<.5ex>[r]^{\pi_F^\ast\;\;\;\;}
& \mathbb{C}^{F\bowtie G} \ar@<.5ex>@{-->}[l]^{\iota_F^\ast\;\;\;\;}
  \ar@<.5ex>[r]^{\;\;\;\;\iota_G^\ast}
& \mathbb{C}^G \ar@<.5ex>@{-->}[l]^{\;\;\;\;\pi_F^\ast}   }
&\;\;\text{and}\;\;&
\xymatrix{
\mathbb{C} G \ar@<.5ex>[r]^{\iota_G\;\;\;\;\;\;}
& \mathbb{C}(F\bowtie G) \ar@<.5ex>@{-->}[l]^{\pi_G\;\;\;\;\;\;} \ar@<.5ex>[r]^{\;\;\;\;\;\;\pi_F}
& \mathbb{C} F \ar@<.5ex>@{-->}[l]^{\;\;\;\;\;\;\iota_F}
},
\end{array}
\end{equation}
where the left $\mathbb{C}^{F\bowtie G}$-comodule structure on $\mathbb{C}^F$ is given by
\begin{equation}\label{eqn:leftcomodmatchedpair}
p_a\mapsto
\sum_{\substack{(b,x)\in F\bowtie G,\;c\in F, \\ b(x\triangleright c)=a}}
p_{(b,x)}\otimes p_c
\;\;\;\;\;\;(\forall a\in F).
\end{equation}
\end{lemma}

\begin{proof}
It is clear that the left $\mathbb{C}^{F\bowtie G}$-comodule structure (\ref{eqn:leftcomodmatchedpair}) on $\mathbb{C}^F$ corresponds to the left $\mathbb{C}(F\bowtie G)$-module structure
$$(a,x)\otimes b\mapsto a(x\triangleright b)\;\;\;\;\;\;\;\;(\forall a,b\in F,\;\forall x\in G)$$
on $\mathbb{C}F$, which fits the requirement of a {\pams} for the latter diagram in (\ref{eqn:PAMSbismashprod}).
Thus one could find that
the claims desired are completely analogous to \cite[Lemma 6.6 and Proposition 6.8]{Li23}.
\end{proof}

\begin{proposition}\label{prop:indicatorsbismashprod}
Let $(F,G,\triangleright,\triangleleft)$ be a matched pair of finite groups.
For any $a\in F$, regard $\mathbb{C}p_a\in\Rep(\mathbb{C}^F)$ as the 1-dimensional representation. Then the indicator
$\nu_n(\mathbb{C}^G\#p_a)$ of the object $\mathbb{C}^G\#p_a\in\Rep(\mathbb{C}^G\#\mathbb{C}F)$ equals the cardinal number of
$$\{x\in G\mid(a,x)^n=(1,1)\text{ in }F\bowtie G\}$$
for each $n\geq1$.
\end{proposition}

\begin{proof}
By applying Theorem \ref{thm:inds-leftpartialdual} to the left partial dual determined in Lemma \ref{lem:bismashprodPD}, one could obtain that
\begin{equation}\label{eqn:inds(smL)=(Lsq)bismashprod}
\nu_n(\mathbb{C}^G\#p_a)=\nu_n(p_a\square_{\mathbb{C}F}\mathbb{C}(F\bowtie G))
\;\;\;\;\;\;\;\;(\forall n\geq1).
\end{equation}

Let us focus on the object appearing at the right side of Equation (\ref{eqn:inds(smL)=(Lsq)bismashprod}): The left $\mathbb{C}F$-comodule structure on $\mathbb{C}(F\bowtie G)$ should be induced via $\pi_F$ as
$$\mathbb{C}(F\bowtie G)\rightarrow \mathbb{C}F\otimes\mathbb{C}(F\bowtie G),\;\;
(a,x)\mapsto b\otimes(b,x)\;\;\;\;\;\;(\forall b\in F,\;\forall x\in G).$$
It follows that the isomorphism $\mathbb{C}(F\bowtie G)\cong \mathbb{C}F\otimes\mathbb{C}G$, $(b,x)\mapsto b\otimes x$ preserves left $\mathbb{C}F$-coactions. As a consequence,
\begin{equation}\label{eqn:cotensoriso}
p_a\otimes\mathbb{C}G \cong p_a\square_{\mathbb{C}F}\mathbb{C}(F\bowtie G),\;\;
p_a\otimes x\mapsto p_a\otimes(a,x)\;\;\;\;\;\;(\forall x\in G)
\end{equation}
is a linear isomorphism, since it is evident that $\mathbb{C}p_a\in\Rep(\mathbb{C}^F)$ has the right $\mathbb{C}F$-comodule structure $p_a\mapsto p_a\otimes a$.

Now we define a left $\mathbb{C}^{F\bowtie G}$-action on $p_a\otimes\mathbb{C}G$ as
$$\begin{array}{cccl}
\mathbb{C}^{F\bowtie G}\otimes(p_a\otimes\mathbb{C}G) &\rightarrow& p_a\otimes\mathbb{C}G, \\
p_{(b,y)}\otimes(p_a\otimes x) &\mapsto& \delta_{a,b}\delta_{x,y}(p_a\otimes x)
\end{array}
\;\;\;\;\;\;(\forall b\in F,\;\forall x,y\in G).$$
Then one could directly verify that (\ref{eqn:cotensoriso}) is also an isomorphism of left $\mathbb{C}^{F\bowtie G}$-modules, as the left $\mathbb{C}^{F\bowtie G}$-action on the object $p_a\square_{\mathbb{C}F}\mathbb{C}(F\bowtie G)$ is given in (\ref{eqn:cotensorH*modstru}) such that
$$p_{(b,y)}\cdot[p_a\otimes (a,x)]
=p_a\otimes[ p_{(b,y)}\rightharpoonup(a,x)]=\delta_{a,b}\delta_{x,y}(p_a\otimes x)
\;\;\;\;\;\;(\forall b\in F,\;\forall x,y\in G).$$

Furthermore, it is clear that $\{p_a\otimes x\}_{x\in G}$ is a basis of $p_a\otimes\mathbb{C}G$ with dual basis $\{a\otimes p_x\}_{x\in G}$. Thus the character $\chi$ of $p_a\otimes\mathbb{C}G\in\Rep(\mathbb{C}^{F\bowtie G})$ satisfies that
$$\langle\chi,p_{(b,y)}\rangle=\sum_{x\in G}\langle a\otimes p_x,p_{(b,y)}\cdot(p_a\otimes x)\rangle
=\sum_{x\in G}\langle a\otimes p_x,\delta_{a,b}\delta_{x,y}(p_a\otimes x)\rangle=\delta_{a,b}
\;\;\;\;(\forall b\in F,\;\forall y\in G).$$

On the other hand, the normalized integral of the Hopf algebra $\mathbb{C}^{F\bowtie G}$ is $p_{(1,1)}$, whose $n$-th Sweedler power is
$$p_{(1,1)}{}^{[n]}=\sum_{\substack{(b,y)^n=(1,1) \\\text{in }F\bowtie G}}p_{(b,y)}$$
for all $n\geq1$.
As a conclusion, we find by Remark \ref{rmk:inds-Hopfalgcoincide} the following equations on the indicators of left $\mathbb{C}^{F\bowtie G}$-modules:
\begin{eqnarray*}
\nu_n(\mathbb{C}^G\#p_a)
&\overset{(\ref{eqn:inds(smL)=(Lsq)bismashprod})}=&
= \nu_n(p_a\square_{\mathbb{C}F}\mathbb{C}(F\bowtie G))
~=~ \nu_n(p_a\otimes\mathbb{C}G) ~=~ \langle\chi,p_{(1,1)}{}^{[n]}\rangle   \\
&=& \sum_{\substack{(b,y)^n=(1,1) \\\text{in }F\bowtie G}}\langle\chi,p_{(b,y)}\rangle
= \sum_{\substack{(b,y)^n=(1,1) \\\text{in }F\bowtie G}} \delta_{a,b},
\end{eqnarray*}
which equals the cardinal number of
$\{x\in G\mid(a,x)^n=(1,1)\text{ in }F\bowtie G\}$
for each $n\geq1$.
\end{proof}

Another example of left partially dualized Hopf algebra is the finite-dimensional (generalized) quantum double introduced in \cite{DT94}. Let us recall the concept in terms of Hopf pairings instead of skew ones:

Suppose $H$ and $K$ are
finite-dimensional Hopf algebras over $\k$ with Hopf pairing $\sigma:K^{\ast}\otimes H\rightarrow\k$,
which induces two Hopf algebra maps
\begin{equation*}\label{eqn: sigma lr}
\sigma_l: K^\ast\rightarrow H^\ast,\;\;k^\ast\mapsto\sigma(k^\ast, -)
\;\;\;\;\text{and}\;\;\;\;\sigma_r: H\rightarrow K,\;\;h\mapsto\sigma(-, h).
\end{equation*}
Then it determines the \textit{quantum double} denoted by $K^{\ast\,\cop}\bowtie_\sigma H$,
which will become the Drinfeld double $D(H)$ (\cite{Dri86}) if $K=H$ and $\sigma$ is the evaluation.

\begin{lemma}\label{lem:quantumdoublePD}
(\cite[Section 3.1]{HKL25})
There is a {\pams}
\begin{equation}\label{eqn:PAMSdouble}
\begin{array}{ccc}
\xymatrix{
H\ar@<.5ex>[r]^{\iota\;\;\;\;\;\;}
& K^\op\otimes H
 \ar@<.5ex>@{-->}[l]^{\zeta\;\;\;\;\;\;} \ar@<.5ex>[r]^{\;\;\;\;\;\;\pi}
& K^\op \ar@<.5ex>@{-->}[l]^{\;\;\;\;\;\;\gamma}  }
&\;\;\text{and}\;\;&
\xymatrix{
K^{\ast\cop}\ar@<.5ex>[r]^{\pi^\ast\;\;\;\;\;\;}
& K^{\ast\cop}\otimes H^\ast
 \ar@<.5ex>@{-->}[l]^{\gamma^\ast\;\;\;\;\;\;}
 \ar@<.5ex>[r]^{\;\;\;\;\;\;\;\;\iota^\ast}
& H^\ast \ar@<.5ex>@{-->}[l]^{\;\;\;\;\;\;\;\;\zeta^\ast}  }
\end{array}
\end{equation}
Furthermore, the left partially dualized Hopf algebra determined by (\ref{eqn:PAMSdouble}) is exactly the quantum double $K^{\ast\cop}\bowtie_\sigma H$ .

In particular, the Drinfeld double $D(H)$ is a left partially dualized Hopf algebra of $H^\op\otimes H$.
\end{lemma}

\begin{remark}
The detailed construction of the {\pams} (\ref{eqn:PAMSdouble}) is found in \cite[Section 3.1]{HKL25}. Here we remark the definition of $\iota$ for subsequent usages:
\begin{equation}\label{eqn:iotadouble}
\iota(h)=\sum \sigma_r(S^{-1}(h_{(2)}))\otimes h_{(1)}\in K^\op\otimes H\;\;\;\;\;\;\;\;(\forall h\in H).
\end{equation}
\end{remark}

\begin{proposition}\label{prop:indicatorsquantumdouble}
Let $H$ and $K$ be semisimple Hopf algebras over $\mathbb{C}$ with Hopf pairing $\sigma:K^\ast\otimes H\rightarrow\mathbb{C}$. Then for any $V\in\Rep(H)$,
$$\nu_n(K^{\ast\cop}\bowtie_\sigma V)=\sum_{W\in\mathsf{Irr}(K)}\nu_n^H(V\otimes W)\nu_n^{K^\cop}(W)
\;\;\;\;\;\;\;\;(\forall n\geq1),$$
where $\mathsf{Irr}(K)$ denotes the set of isoclasses of irreducible left $K$-modules, and:
\begin{itemize}
\item
$K^{\ast\cop}\bowtie_\sigma V\in\Rep(K^{\ast\cop}\bowtie_\sigma H)$ with structure similar to (\ref{eqn:smashLmodstru}), namely:
$$(k^\ast\bowtie h)\otimes(l^\ast\bowtie v)
\mapsto\sum \sigma(l^\ast_{(3)}, h_{(1)})(k^\ast l^\ast_{(2)}\bowtie h_{(2)}v)\sigma(S^{-1}(l^\ast_{(1)}), h_{(3)}).$$

\item
For each $W\in\mathsf{Irr}(K)$, the object
$V\otimes W\in\Rep(H)$ with diagonal $H$-action via:
\begin{equation}\label{eqn:VotimesWmodstru}
H\otimes(V\otimes W)\rightarrow V\otimes W,\;\;
h\otimes(v\otimes w)\mapsto\sum h_{(1)}v\otimes \sigma_r(h_{(2)})w,
\end{equation}
and $\nu_n^{K^\cop}(W)$ is the $n$-th indicator of $W\in\Rep(K^\cop)$.
\end{itemize}
\end{proposition}

\begin{proof}
By applying Theorem \ref{thm:inds-leftpartialdual} to the left partial dual determined in Lemma \ref{lem:quantumdoublePD}, one could obtain that
\begin{equation}\label{eqn:inds(smL)=(Lsq)double}
\nu_n(K^{\ast\cop}\bowtie_\sigma V)=\nu_n(V\square_{H^\ast}(K^{\ast\cop}\otimes H^\ast))
\;\;\;\;\;\;\;\;(\forall n\geq1).
\end{equation}
However, consider the left $K^\op\otimes H$-module $V\otimes K^\ast$ with structure defined as
\begin{equation}\label{eqn:VotimesK*modstru}
\begin{array}{ccccc}
\cdot&:&(K^\op\otimes H)\otimes(V\otimes K^\ast) &\rightarrow& V\otimes K^\ast,\;\; \\
&& (k\otimes h)\otimes (v\otimes k^\ast) &\mapsto& \sum h_{(1)}v\otimes [\sigma_r(h_{(2)})\rightharpoonup k^\ast\leftharpoonup k].
\end{array}
\end{equation}
Then we explain similarly to \cite[Lemma 3.12]{HKL25} that
$$\phi:V\square_{H^\ast}(K^{\ast\cop}\otimes H^\ast)\rightarrow V\otimes K^\ast,\;\;
\sum_i v_i\otimes(k^\ast_i\otimes h^\ast_i)\mapsto\sum_i v_i\otimes k^\ast_i\langle h^\ast_i,1\rangle$$
is a well-defined isomorphism of left $K^\op\otimes H$-modules.
In fact, it is straightforward to verify that
$$\psi:V\otimes K^\ast\rightarrow V\square_{H^\ast}(K^{\ast\cop}\otimes H^\ast),\;\;
v\otimes k^\ast\mapsto
\sum v_{\langle0\rangle}\otimes (k^\ast_{(1)}\otimes v_{\langle1\rangle}\sigma_l(k^\ast_{(2)}))$$
is the inverse of $\phi$ in the same way as the proof of \cite[Lemma 3.12]{HKL25}. Moreover, we could find by the following calculation that $\phi$ preserves $K^\op\otimes H$-actions: For any element
$\sum_i v_i\otimes(k^\ast_i\otimes h^\ast_i)\in V\square_{H^\ast}(K^{\ast\cop}\otimes H^\ast)$, the definition of cotensor product $-\square_{H^\ast}-$ provides that
\begin{eqnarray*}\label{eqn:cotensordouble}
\sum_i hv_i\otimes(k^\ast_i\otimes h^\ast_i)
&=&
\sum_i v_i\otimes[(k^\ast_i\otimes h^\ast_i)\leftharpoonup\iota(h)]  \nonumber  \\
&\overset{(\ref{eqn:iotadouble})}=&
\sum_i v_i\otimes[(k^\ast_i\otimes h^\ast_i)\leftharpoonup(\sigma_r(S^{-1}(h_{(2)}))\otimes h_{(1)})]  \nonumber  \\
&=&
\sum_i v_i\otimes[(\sigma_r(S^{-1}(h_{(2)}))\rightharpoonup k^\ast_i)\otimes (h^\ast_i\leftharpoonup h_{(1)})]
\;\;\;\;\;\;\;\;(\forall h\in H),
\end{eqnarray*}
which implies that
\begin{eqnarray*}
&& \phi\Big((k\otimes h)\cdot\sum_i v_i\otimes(k^\ast_i\otimes h^\ast_i)\Big)
~=~ \phi\Big(\sum_i v_i\otimes[(k\otimes h)\rightharpoonup(k^\ast_i\otimes h^\ast_i)]\Big)  \\
&=& \phi\Big(\sum_i v_i\otimes[(k^\ast_i\leftharpoonup k)\otimes (h\rightharpoonup h^\ast_i)]\Big)
~=~ \sum_i v_i\otimes(k^\ast_i\leftharpoonup k)\langle h\rightharpoonup h^\ast_i,1\rangle  \\
&=& \sum_i v_i\otimes(k^\ast_i\leftharpoonup k)\langle h^\ast_i,h\rangle
~=~ \sum_i v_i\otimes\big(\sigma_r(h_{(3)})\sigma_r(S^{-1}(h_{(2)}))\rightharpoonup k^\ast_i\leftharpoonup k\big)\langle h^\ast_i\leftharpoonup h_{(1)},1\rangle  \\
&=& \sum_i h_{(1)}v_i\otimes\big(\sigma_r(h_{(2)})\rightharpoonup k^\ast_i\leftharpoonup k\big)\langle h^\ast_i,1\rangle
\overset{(\ref{eqn:VotimesK*modstru})}=
(k\otimes h)\cdot \phi\Big(\sum_i v_i\otimes(k^\ast_i\otimes h^\ast_i)\Big)
\end{eqnarray*}
hold for all $k\in K$ and $h\in H$.

Now denote by $\mathcal{S}$ the set of simple subcoalgebras of the cosemisimple Hopf algebra $K^\ast$, and then $K^\ast=\bigoplus_{D\in\mathcal{S}}D$ would also be a decomposition of simple $K$-$K$-bimodules with actions $\rightharpoonup$ and $\leftharpoonup$. Furthermore, it is clear that there is a bijection
$$\mathcal{S}\rightarrow \mathsf{Irr}(K),\;\;D\mapsto W$$
satisfying that $D\cong W\otimes W^\ast$ as $K$-$K$-bimodules, where the left and right module structures on $W\otimes W^\ast$ are determined at the respective tensorands, namely:
$$k\cdot(w\otimes w^\ast)=kw\otimes w^\ast\;\;\;\;\text{and}\;\;\;\;
(w\otimes w^\ast)\cdot k=w\otimes w^\ast k\;\;\;\;
(\forall k\in K,\;\forall w\in W,\;\forall w^\ast\in W^\ast).$$

As a conclusion, we could obtain following $K^\op\otimes H$-module isomorphisms from $V\otimes K^\ast$ with structure (\ref{eqn:VotimesK*modstru}):
$$V\otimes K^\ast\cong V\otimes \Big(\bigoplus_{D\in\mathcal{S}}D\Big)
\cong V\otimes \Big(\bigoplus_{W\in\mathsf{Irr}(K)}W\otimes W^\ast\Big)
\cong \bigoplus_{W\in\mathsf{Irr}(K)}(V\otimes W)\otimes W^\ast,$$
where $V\otimes W\in\Rep(H)$ with diagonal $H$-module structure (\ref{eqn:VotimesWmodstru}), and the $K^\op\otimes H$-action on $(V\otimes W)\otimes W^\ast$ is determined via
$$(k\otimes h)\cdot((v\otimes w)\otimes w^\ast)=[h\cdot(v\otimes w)]\otimes w^\ast k
\;\;\;\;\;\;\;\;(\forall k\in K,\;\forall v\in V,\;\forall w\in W,\;\forall w^\ast\in W^\ast).$$
Consequently, we compute its indicators with the help of Lemma \ref{lem:indicatorformula} that
\begin{eqnarray*}
\nu_n(V\otimes K^\ast)
&=& \nu_n\Big(\bigoplus_{W\in\mathsf{Irr}(K)}(V\otimes W)\otimes W^\ast\Big)
\overset{(\ref{eqn:indicatorformula2})}=
\sum_{W\in\mathsf{Irr}(K)}\nu_n ((V\otimes W)\otimes W^\ast)  \\
&\overset{(\ref{eqn:indicatorformula1})}=&
\sum_{W\in\mathsf{Irr}(K)}\nu_n^H (V\otimes W)\nu_n^{K^\op}(W^\ast)
\overset{(\ref{eqn:indicatorformula3})}=
\sum_{W\in\mathsf{Irr}(K)}\nu_n^H (V\otimes W)\nu_n^{K^\cop}(W)
\end{eqnarray*}
for all $n\geq1$.

Finally, since $\phi:V\square_{H^\ast}(K^{\ast\cop}\otimes H^\ast)\cong V\otimes K^\ast$ is an isomorphism in $\Rep(K^\op\otimes H)$, it follows from Equation (\ref{eqn:inds(smL)=(Lsq)double}) that
$$\nu_n(K^{\ast\cop}\bowtie_\sigma V)=\sum_{W\in\mathsf{Irr}(K)}\nu_n^H(V\otimes W)\nu_n^{K^\cop}(W)
\;\;\;\;\;\;\;\;(\forall n\geq1).$$
\end{proof}

An analogous formula on indicators of $D(H)$ could be obtained as a particular case:

\begin{corollary}\label{cor:indicatorDrinfelddouble}
Let $H$ be a semisimple Hopf algebra over $\mathbb{C}$. Then for any $V\in\Rep(H)$,
$$\nu_n(H^{\ast\cop}\bowtie V)=\sum_{W\in\mathsf{Irr}(H)}\nu_n^H(V\otimes W)\nu_n^{H^\cop}(W)
\;\;\;\;\;\;\;\;(\forall n\geq1)$$
holds, where $H^{\ast\cop}\bowtie V$ is a $D(H)$-module with structure
$$(f\bowtie h)\otimes(g\bowtie v)
\mapsto\sum \langle g_{(3)}, h_{(1)}\rangle (f g_{(2)}\bowtie h_{(2)}v) \langle S^{-1}(g_{(1)}), h_{(3)}\rangle.$$
\end{corollary}

Besides, according to the realization in Lemma \ref{lem:leftPDsplitext}, one might deduce similar formulas on indicators for certain representations of semisimple Hopf algebras fitting into split abelian extensions, but we would not do it in this paper.

\subsection{Invariance for the exponents of semisimple (quasi-)Hopf algebras under partial dualities}

As consequences of the results in Subsection \ref{subsection:mainthmPD},
we collect some properties on the exponent and Frobenius-Schur exponent of semisimple left partially dualized quasi-Hopf algebras in this subsection.
The first one is directly obtained by Corollary \ref{cor:inds-leftpartialdual2}(1) as follows:

\begin{corollary}\label{cor:partialdualFSexp0}
Let $H$ be a semisimple Hopf algebra over $\mathbb{C}$. Suppose that $C^\ast\#B$ is a left partially dualized quasi-Hopf algebra of $H$.
Then
\begin{equation*}
\FSexp(C^\ast\# B)=\FSexp(H).
\end{equation*}
\end{corollary}

\begin{proof}
As $\Rep(C^\ast\# B)$ and $\Rep(H)$ are regarded as integral fusion categories with canonical pivotal structures, the pivotal traces (\ref{eqn:canonicalptr}) of the identity morphisms on their regular modules would be
$$\ptr(\id_{C^\ast\# B})=\FPdim(C^\ast\# B)=\dim(C^\ast\# B)$$
and
$$\ptr(\id_H)=\FPdim(H)=\dim(H),$$
respectively. Moreover, it is mentioned in Remark \ref{rmk:partialdualdim} that $\dim(C^\ast\# B)=\dim(H)$, and hence
\begin{equation}\label{eqn:leftpartialdualptr}
\ptr(\id_{C^\ast\# B})=\ptr(\id_H)
\end{equation}
holds.

On the other hand, we know in Corollary \ref{cor:inds-leftpartialdual2}(1) that the regular modules $C^\ast\# B$ and $H$ have the same $n$-th indicators for each $n\geq1$. Thus one could combine the definition of Frobenius-Schur exponent (\ref{eqn:objFSexp}) and Equation (\ref{eqn:leftpartialdualptr}) to find that
$\FSexp(C^\ast\# B)=\FSexp(H)$ hold.
\end{proof}

\begin{remark}
According to \cite[Proposition 5.3]{NS07(b)}, if $\C$ is a pseudo-unitary fusion category, its Frobenius-Schur exponent $\FSexp(\C)$ equals to the Frobenius-Schur exponent of an object containing every simple object of $\C$.
The indicators of the regular object in $\C$ were studied in \cite{Shi12},
and it should be worth comparing indicators between the regular objects of $\C$ and $\C_\M^\ast$. However, the indicators of the regular object in $\C_\M^\ast$ might not be defined analogously, since we do not know whether $\C_\M^\ast$ becomes pivotal in general.
\end{remark}

%\subsection{Invariance for the exponent of semisimple (quasi-)Hopf algebras under partial dualities}

Now
let $\C$ be a fusion category over $\mathbb{C}$. Recall that the \textit{exponent} of $\C$ is defined in \cite[Remark 6]{Eti02} as the order of the square of the braiding $\sigma$ for its left center $\Z(\C)$. Namely,
\begin{equation}
\exp(\C):=\min\{n\geq1\mid \forall V,W\in\Z(\C),\;\;
(\sigma^W_V\circ\sigma^V_W)^n=\id_V\otimes\id_W\}.
\end{equation}
It is a finite positive integer, which is also related with the Frobenius-Schur exponent according to \cite[Section 6]{NS07(b)}.

However in this subsection, we make some discussions on $\exp(\Rep(H))$ where $H$ is a semisimple Hopf algebra over $\mathbb{C}$. Note at first that $\exp(\Rep(H))$ equals to
$$\exp(H):=\min\big\{n\geq1\mid \forall h\in H,\;\;
  \sum h_{(1)}h_{(2)}\cdots h_{(n)}=\e(h)1_H\big\},$$
which is referred in \cite{EG99} to be the \textit{exponent} of the semisimple Hopf algebra $H$ (with involutory antipode). It is also a finite positive integer, and it divides $\dim(H)^3$ according to \cite[Theorem 4.3]{EG99}.
Moreover in character $0$, the number $\exp(H)$ is known to be the period of the sequence
\begin{equation}\label{eqn:nu_n(H)}
\nu_n(H),\;\;\;\;n=1,2,\cdots
\end{equation}
of indicators (\ref{eqn:inds-Hopfalg}) for the regular representation $H$.
This is a consequence of the results in \cite{KSZ06}, and some
specific descriptions could be found in \cite[Section 1]{NS07(b)}.

Besides, we have mentioned in Remark \ref{rmk:inds-Hopfalgcoincide} the fact stated in \cite[Remark 3.4]{NS08}: If the semisimple Hopf algebra $H$ is regarded as a quasi-Hopf algebra, then the regular indicators $\nu_n(H)$ in the senses of Definition \ref{def:inds-quasiHopfalg}(1) and Equation (\ref{eqn:inds-Hopfalg}) coincide.
Furthermore, it follows from \cite[Proposition 5.3]{NS07(b)} that $\FSexp(H)$ is indeed the period of the sequence (\ref{eqn:nu_n(H)}), since the regular module $H$ is clearly a projective generator of $\Rep(H)$. Thus we could conclude that:

\begin{lemma}(\cite[Section 5]{NS07(b)})\label{lem:HopfalgFSexp=exp}
Let $H$ be a (finite-dimensional) semisimple Hopf algebra over $\mathbb{C}$. Then its exponent and Frobenius-Schur exponent are equal. Namely,
$\FSexp(H)=\exp(H).$
\end{lemma}

As for a semisimple quasi-Hopf algebra, its exponent and Frobenius-Schur exponent could be different with amounts up to a factor $2$ according to \cite[Example 5.4 and Corollary 6.2]{NS07(b)}.
However,
the following proposition implies that these two exponents coincide for semisimple left partial duals, which might generalize a (quasi-)Hopf algebraic version of \cite[Proposition 6.3(2)]{Eti02}:

\begin{proposition}%(cf. \cite[Proposition 6.3(2)]{Eti02})
\label{cor:partialdualexp}
Let $H$ be a semisimple Hopf algebra over $\mathbb{C}$. Suppose that $C^\ast\# B$ is a left partial dualized quasi-Hopf algebra of $H$.
Then the numbers
\begin{equation*}
\FSexp(C^\ast\# B),\;\;\;\;\exp(C^\ast\# B),
\;\;\;\;\FSexp(H)\;\;\;\;\text{and}\;\;\;\;\exp(H)
\end{equation*}
are all equal.
\end{proposition}

\begin{proof}
Combining Corollary \ref{cor:partialdualFSexp0} and Lemma \ref{lem:HopfalgFSexp=exp}, we know at first that
$$\FSexp(C^\ast\# B)=\FSexp(H)=\exp(H).$$

Moreover, due to the reconstruction $\Rep(H)_{\Rep(B)}^\ast\approx\Rep(C^\ast\# B)$ according to \cite[Corollary 4.23]{Li23}, one could also find that
$$\exp(H)=\exp(C^\ast\# B)$$
 holds as an application of \cite[Proposition 6.3(2)]{Eti02} to the $\Rep(H)$-module category $\Rep(B)$.
\end{proof}

%The following is evident.
%
%\begin{corollary}(\cite[Proposition 2.2(4)]{EG99})
%Let $H$ be a semisimple Hopf algebra over $\mathbb{C}$. Then
%$$\exp(H^\ast)=\exp(H).$$
%\end{corollary}

Note in \cite[Proposition 6.8]{Li23} or Subsection \ref{subsection:5.4} that for a matched pair $(F,G)$ of finite groups, the bismash product $\mathbb{C}^G\#\mathbb{C} F$ is a left partial dualized Hopf algebra of the group algebra $\mathbb{C}(F\bowtie G)$. One might apply Corollary \ref{cor:partialdualexp} to obtain the result \cite[Proposition 3.1]{LMS06} over $\mathbb{C}$:

\begin{corollary}(\cite[Proposition 3.1]{LMS06})
For a matched pair $(F,G,\triangleright,\triangleleft)$ of finite groups, we have
\begin{equation}\label{eqn:bismashprodexp}
\exp(\mathbb{C}^G\#\mathbb{C} F)=\exp(\mathbb{C}(F\bowtie G))=\exp(F\bowtie G),
\end{equation}
where $\exp(F\bowtie G)$ denotes the exponent of the group $F\bowtie G$.
\end{corollary}

In fact, the bismash product Hopf algebra $\mathbb{C}^G\#\mathbb{C} F$ arises from the split abelian extension
$$\mathbb{C}^G\rightarrow\mathbb{C}^G\#\mathbb{C} F
\rightarrow\mathbb{C} F.$$
Conversely, every semisimple split abelian extension
$B\rightarrow H\rightarrow C$
of Hopf algebras over $\mathbb{C}$ determines a matched pair $(F,G,\triangleright,\triangleleft)$ of finite groups, such that
%$$B\cong \mathbb{C}^G,\;\;\;\;C\cong\mathbb{C} F
%\;\;\;\;\text{and}\;\;\;\;H\cong\mathbb{C}^G\#\mathbb{C} F.$$
$B\cong \mathbb{C}^G$, $C\cong\mathbb{C} F$ and $H\cong\mathbb{C}^G\#\mathbb{C} F$.
The relevant results are concluded in \cite{Mas02}.
%
%Let us recall the definition of abelian extensions with our notations for convenience:
%
%\begin{definition}(cf. \cite[Definition 6.5.2]{Sch02})\label{def:splitHopfalgext}
%Let $H$, $B$ and $C$ be Hopf algebras.
%Then an extension
%\begin{equation}\label{eqn:abelext}
%B\xrightarrow{\iota} H\xrightarrow{\pi} C
%\end{equation}
%of Hopf algebras is said to be split, if there exist a left $B$-module coalgebra map $\zeta:H\rightarrow B$ and a right $C$-comodule algebra map $\gamma:C\rightarrow H$, such that $(\iota\circ\zeta)\ast(\gamma\circ\pi)=\id_H$.
%
%Furthermore, we say that the extension (\ref{eqn:abelext}) is abelian, if $B$ is commutative and $C$ is cocommutative.
%\end{definition}

%\begin{remark}
%With the notations in Definition \ref{def:splitHopfalgext},
%it is clear that $\zeta$ preserves the counits, and $\gamma$ preserves the units. Moreover
%according to the arguments in the proof of \cite[Theorem 9]{DT86} and
%\cite[Lemma 2.15]{Mas92}, we could assume that $\zeta$ and $\gamma$ both preserve the units and the counits since $\zeta(1)$ must be a group-like element. Details are also found in the paragraphs before \cite[Lemma 2.5]{Li23}.
%
%Consequently, a split extension (\ref{eqn:abelext}) of finite-dimensional Hopf algebras implies a {\pams} $(\zeta,\gamma^\ast)$ of $\iota$, such that $\zeta$ is a left $B$-module coalgebra map and $\gamma$ is a right $C$-comodule algebra map.
%\end{remark}

On the other hand, there is an inspiring conjecture by Kashina on the exponent of Hopf algebras, which states that:
\begin{conjecture}(\cite{Kas99,Kas00})
Let $H$ be a semisimple and cosemisimple Hopf algebra over a field. Then
$\exp(H)\mid\dim(H).$
\end{conjecture}

We remark that
one could also deduce by Corollary \ref{cor:splitabelext} that: $\exp(H)\mid\dim(H)$
holds when $H$ is a semisimple complex Hopf algebra fitting into a split abelian extension, since the exponent of a finite group must divide its order.

\begin{remark}
More generally, suppose $B\rightarrow H\rightarrow C$ is a (non-abelian) split extension of semisimple Hopf algebras over $\mathbb{C}$. Since it is known in \cite[Proposition 2.2(4) and Corollary 2.6]{EG99} that
$\exp(H^\ast)=\exp(H^\op)=\exp(H^\cop)=\exp(H),$
one could formulate the bismash product $B^\biop\#C^\biop$ in the sense of Lemma \ref{lem:leftPDsplitext}, and obtain
$$\exp(H)=\exp(B^\biop\#C^\biop)$$
due to Proposition \ref{cor:partialdualexp}.
\end{remark}

As for semisimple Hopf algebras admitting non-split abelian extensions (or group-theoretical Hopf algebras), the indicators are studied in \cite{Shi12,Sch15,Sch16}, which are related to formulas established in \cite[Section 7]{NS08}.
Besides, there are results in \cite[Section 5.2]{Nat07} on the upper bounds of the exponent.

\section*{Acknowledgement}

The author is extremely grateful to Professor Kenichi Shimizu, who provided the motivation and question, as well as taught author the methods and comments for this paper with so much patience. The author would also like to thank Professors Gongxiang Liu and Zhimin Liu for useful discussions and suggestions. Besides, the final version of the manuscript was completed during the visit at University of Tsukuba, the author would like to thank Professor Akira Masuoka for his considerate hospitality and discussions.

\end{document}